\documentclass[options]{amsart}
\usepackage[english]{babel}
\usepackage{tikz,color}
\usetikzlibrary{matrix,arrows}
\usepackage{mathrsfs}
\usepackage{amssymb,amsmath,amsthm, enumerate, mathrsfs,textcomp,verbatim, stmaryrd}
\usepackage{hyperref}
\usepackage[all]{xy}
\usepackage[all]{xy}
\usepackage[T1]{fontenc}

\newcommand{\cod}
 {{\rm cod}}

\newcommand{\comp}
 {\circ}

\newcommand{\Cont}
 {{\bf Cont}}

\newcommand{\Gmod}
{\mathbb{G}\mbox{-mod}(\Set)}

\newcommand{\Gemod}
{\mathbb{G}_{e}\mbox{-mod}(\Set)}

\newcommand{\N}
{\mathbb{N}}

\newcommand{\Z}
{\mathbb{Z}}

\newcommand{\Q}
{\mathbb{Q}}

\newcommand{\R}
{\mathbb{R}}

\newcommand{\T}
{\mathbb{T}}

\newcommand{\G}
{\mathbb{G}}

\renewcommand{\O}
{\mathbb{O}}

 \newcommand{\kgen}
 {k-\rm{Gen}_{\vec{z}}(\vec{x}, \vec{y})}

\newcommand{\dom}
 {{\rm dom}}

\newcommand{\empstg}
 {[\,]}

\newcommand{\epi}
 {\twoheadrightarrow}

\newcommand{\hy}
 {\mbox{-}}

\newcommand{\im}
 {{\rm im}}

\newcommand{\imp}
 {\!\Rightarrow\!}

\newcommand{\Ind}[1]
 {{\rm Ind}\hy #1}

\newcommand{\mono}
 {\rightarrowtail}
 
\newcommand{\name}[1]
 {\mbox{$\ulcorner #1 \urcorner$}}

\newcommand{\ob}
 {{\rm ob}}

\newcommand{\op}
 {^{\rm op}}

\newcommand{\Set}
 {{\bf Set }}

\newcommand{\Sh}
 {{\bf Sh}}

\newcommand{\sh}
 {{\bf sh}}

\newcommand{\Sub}
 {{\rm Sub}}

\newtheorem{theorem}{Theorem}[section]
\newtheorem{lemma}[theorem]{Lemma}
\newtheorem{proposition}[theorem]{Proposition}

\theoremstyle{definition}
\newtheorem{definition}[theorem]{Definition}
\newtheorem{example}[theorem]{Example}
\newtheorem{obs}[theorem]{Remark}
\newtheorem{remark}[theorem]{Remark}

\newcommand{\E}{\mathscr{E}}
\renewcommand{\P}{\mathbb{P}^*}

\title[Lattice-ordered abelian groups and perfect MV-algebras]{Lattice-ordered abelian groups and perfect MV-algebras: a topos-theoretic perspective}

\author{Olivia Caramello}

\author{Anna Carla Russo}

\begin{document}
\baselineskip=15pt

\begin{abstract}
We establish, generalizing Di Nola and Lettieri's categorical equiva\-lence, a Morita-equivalence between the theory of lattice-ordered abelian groups and that of perfect MV-algebras. Further, after observing that the two theo\-ries are not bi-interpretable in the classical sense, we identify, by conside\-ring appropriate topos-theoretic invariants on their common classifying topos, three levels of bi-intepretability holding for particular classes of formulas: irreducible formulas, geometric sentences and imaginaries. Lastly, by investigating the classifying topos of the theory of perfect MV-algebras, we obtain various results on its syntax and semantics also in relation to the cartesian theory of the variety generated by Chang's MV-algebra, including a concrete representation for the finitely presentable models of the latter theory as finite products of finitely presentable perfect MV-algebras. Among the results established on the way, we mention a Morita-equivalence between the theory of lattice-ordered abelian groups and that of cancellative lattice-ordered abelian monoids with bottom element.
\end{abstract}

\maketitle

\tableofcontents

\section{Introduction}

\renewcommand{\arraystretch}{1.6}

This paper represents a contribution to the investigation, initiated in \cite{Russo}, of notable categorical equivalences arising in the field of many-valued logics from a topos-theoretic perspective.

A sound and complete algebraic semantics for the propositional infinite-valued logic of \L ukasiewicz is provided by the class of MV-algebras, introduced by Chang in 1958. Since then, many mathematicians from different backgrounds have deve\-loped an interest in this class of algebras. Indeed, the literature contains several applications of MV-algebras in different areas such as lattice-ordered abelian group theory and functional analysis (cf. \cite{Mundici} and \cite{P-L}).

In 1994, A. Di Nola and A. Lettieri established a categorical equivalence between the category of perfect MV-algebras and that of lattice-ordered abelian groups (cf. \cite{P-L}). Perfect MV-algebras form an interesting class of MV-algebras, which is directly related to the important problem of incompleteness of first-order \L ukasiewicz logic; indeed, the subalgebra of the Lindenbaum algebra of first-order \L ukasiewicz logic generated by the classes of formulas which are valid but not prova\-ble is a perfect MV-algebra (cf. \cite{BC}).   

In this paper we constructively generalize Di Nola-Lettieri's equivalence by interpreting the involved categories as categories of set-based models of two particular theories, namely the theory $\mathbb{P}$ of perfect MV-algebras and the theory $\mathbb{L}$ of lattice-ordered abelian groups. Furthermore, we show that we can actually `lift' this categorical equivalence between the categories of models of these two theories from the topos of sets to an arbitrary Grothendieck topos $\mathscr{E}$, naturally in $\mathscr{E}$; that is, for any geometric morphism $f:\mathscr{F}\to \mathscr{E}$ changing the universe in which the models live with inverse image functor $f^{\ast}:\mathscr{E}\rightarrow \mathscr{F}$, we have a commutative diagram

\begin{center} 

\begin{tikzpicture}
\node (1) at (0,0) {$\mathbb{L}$-mod$(\mathscr{F})$};
\node (2) at (5,0) {$\mathbb{P}$-mod$(\mathscr{F})$};
\node (3) at (0,2) {$\mathbb{L}$-mod$(\mathscr{E})$};

\node (5) at (-4,2) {$\mathscr{E}$};
\node (6) at (-4,0) {$\mathscr{F}$};
\draw[->] (5) to node [right, midway] {$f^*$} (6);

\node (4) at (5,2) {$\mathbb{P}$-mod$(\mathscr{E})$};
\draw[->] (1) to node [below,midway] {$\tau_{\mathscr{F}}$} (2);
\draw[->] (3) to node [left, midway] {$f^{\ast}$} (1);
\draw[->] (4) to node [right, midway] {$f^*$} (2);
\draw[->] (3) to node [above,midway] {$\tau_{\mathscr{E}}$} (4);
\end{tikzpicture}
\end{center}
where $\tau_{\mathscr{E}}:\mathbb{L}$-mod$(\mathscr{E})\rightarrow \mathbb{P}$-mod$(\mathscr{E})$ is the equivalence between the categories of models of the two theories in the topos $\mathscr{E}$.

This means that the two theories have the same (up to categorical equivalence) classifying topos, i.e. that they are \textit{Morita-equivalent}. 

Morita-equivalences are important for our purposes because they allow us to apply a particular topos-theoretic technique, namely the `bridge technique', introduced by the first author in \cite{Caramello1}, to transfer properties and results between the two theories by using the common classifying topos as a `bridge', on which different kinds of topos-theoretic invariants can be considered in relation to its different representations.  

Applications of this technique in the context of our Morita-equivalence produce a variety of insights on the theories, which would be not visible or hardly attainable by using alternative methods. For instance, whilst the two theories are not classically bi-interpretable, as we prove in the paper, the Morita-equivalence between them yields three different levels of bi-interpretability between particular classes of formulas: irreducible formulas, geometric sentences and imaginaries. Other applications are described in section \ref{sec:someapplications}.

In the final part of the paper, we study in detail the classifying topos of the theo\-ry of perfect MV-algebras, representing it as a subtopos of the classifying topos for the algebraic theory axiomatizing the variety generated by Chang's MV-algebra. This investigation sheds light on the relationship between these two theories, notably leading to a representation theorem for finitely generated (resp. finitely presented) algebras in Chang's variety as finite products of finitely generated (resp. finitely presented) perfect MV-algebras. It is worth to note that this result, unlike most of the representation theorems available in the literature, is fully constructive. Among the other insights, we mention a characterization of the perfect MV-algebras which correspond to finitely presented lattice-ordered abelian groups via Di Nola-Lettieri's equivalence as the finitely presented objects of Chang's variety which are perfect MV-algebras, and the property that the theory axiomatizing Chang's variety proves all the cartesian sequents (in particular, all the algebraic identities) which are valid in all perfect MV-algebras. Further applications are given in section \ref{sec:someapplications}.

The paper is organised as follows.

In section \ref{topos theory} we introduce the basic topos-theoretic background relevant for understanding the paper.

In section \ref{perfect} we recall the most important notions and notations for MV-algebras. Further, we present an axiomatization for the theory of perfect MV-algebras and give a list of provable sequents in this theory that will be useful in the following. An analogous treatment for lattice-ordered abelian groups is provided in section \ref{groups}.

After reviewing Di Nola-Lettieri's equivalence in section \ref{equivalence in set}, we generalize it to an arbitrary Grothendieck topos obtaining the main result of section \ref{equivalence in topos}: the theory of perfect MV-algebras and the theory of lattice-ordered abelian groups are Morita-equivalent.

In section \ref{intermediary}, with the purpose of better understanding the relationship between perfect MV-algebras and the associated lattice-ordered abelian groups, we establish an intermediary Morita-equivalence between the theory of lattice-ordered abelian groups and the theory of positive cones of these groups (i.e., the theory $\mathbb{M}$ of cancellative subtractive lattice-ordered abelian monoids with a bottom element). This analysis yields an equivalence between the cartesian syntactic categories of these two theories, providing in particular an alternative description of the  Grothendieck group of a model $\mathcal{M}$ of $\mathbb{M}$ as a subset, rather than a quotient as in the classical definition, of the product $\mathcal{M}\times \mathcal{M}$. 

In section \ref{sct:intepretability}, we show that, whilst the theories $\mathbb L$ and $\mathbb P$ are non bi-interpretable, there is an interpretation of the theory $\mathbb L$ in the theory $\mathbb P$, which can be described explicitly by using the intermediary Morita-equivalence established in section \ref{intermediary}. Moreover, the consideration of different kinds of invariants on the classifying topos of the theories $\mathbb{L}$ and $\mathbb{P}$ yields the above-mentioned three levels of bi-interpretatability between them. 

In section \ref{finitely presented} we characterize the finitely presentable perfect MV-algebras, that is the perfect MV-algebras which correspond to finitely presented lattice-ordered abelian groups via Di Nola-Lettieri's equivalence, as the finitely presentable objects of Chang's variety which are perfect MV-algebras. 

Section \ref{classifying topos} is devoted to the study of the classifying topos of the theory $\mathbb{P}$. We represent this topos as a subtopos of the classifying topos of the theory $\mathbb C$ axiomatizing Chang's variety, and explicitly describe the associated Grothendieck topology. This investigation leads in particular to a representation result genera\-lizing the Stone representation of finite Boolean algebra as powersets: every finitely presented (resp. finitely generated) MV-algebra in the variety generated by Chang's algebra is a finite product of finitely presented (resp. finitely generated) perfect MV-algebras. We also show that every MV-algebra in Chang's variety is a weak subdirect product of perfect MV-algebras. These results have close ties with the existing literature on weak Boolean products of MV-algebras. Finally, we generalize to the setting of MV-algebras in Chang's variety the Lindenbaum-Tarski characterization of Boolean algebras which are isomorphic to powersets as the complete atomic Boolean algebras, obtaining an intrinsic characterization of the MV-algebras in Chang's variety which are arbitrary products of perfect MV-algebras. These results show that Chang's variety constitutes a particularly natural MV-algebraic setting extending the variety of Boolean algebras. 

Section \ref{sec:relatedMoritaequivalence} discusses the relationship between Di Nola-Lettieri's equivalence and Mundici's equivalence, which we extended to a topos-theoretic setting in \cite{Russo}. Specifically, generalizing results in \cite{Yosida}, we show that a theory of \emph{pointed} perfect MV-algebras is Morita-equivalent to the theory of lattice-ordered abelian groups with a distinguished strong unit (and hence to that of MV-algebras).  

Section \ref{sec:someapplications} discusses some further applications of the Morita-equivalence esta\-blished in section \ref{Moritaeq}  and of the `calculation' of the classifying topos of the theory $\mathbb P$ carried out in section \ref{classifying topos}. These applications, obtained by considering appropriate invariants on the classifying topos of  the theory $\mathbb P$, yield insights on the syntax and semantics of $\mathbb P$ also in relation to the theory $\mathbb C$. Finally, in section \ref{sec:transferring} we transfer the above-mentioned representation theorems for the MV-algebras in Chang's variety in terms of perfect MV-algebras into the context of lattice-ordered abelian groups with strong unit.

\section{Topos-theoretic background}\label{topos theory}

For the basic background of topos theory needed for understanding the paper, we refer the reader to \cite{toposbackground} or, for a more succinct overview, to the Appendix of \cite{Russo}.

In this section, we limit ourselves to recalling a few central notions that will play a crucial role in the paper. 

\subsection{Grothendieck topologies and sheaves}\label{back-groth}

The notion of Grothendieck topo\-logy on a category is a categorical generalization of the classical concept of topology on a space. The open sets of the space are replaced by objects of the category and the families of open subsets of a given open set which cover it are replaced by families of arrows in the category with common codomain satisfying appropriate conditions. 

A \textit{sieve} on an object $c$ of a small category $\mathscr{C}$ is a set $S$ of arrows with codomain $c$ such that $f\circ g\in S$ whenever $f\in S$. 

\begin{definition}
A \textit{Grothendieck topology} on a category $\mathscr{C}$ is a function $J$ which assigns to each object $c\in \mathscr{C}$ a collection $J(c)$ of sieves on $c$ in such a way
\begin{itemize}
\item[(i)] the maximal sieve $\{f\mid cod(f)=c\}$ is in $J(c)$;
\item[(ii)] (stability axiom) if $S\in J(c)$, then $h^*(S)\in J(d)$ for any morphism $h:d\rightarrow c$, where with the symbol $h^*(S)$ we mean the sieve whose morphisms are the pullbaks along $h$ of the morphisms in $S$;
\item[(iii)] (transitivity axiom) if $S\in J(c)$ and $R$ is a sieve on $c$ such that $h^*(R)\in J(d)$ for all $h:d\rightarrow c$ in $S$, then $R\in J(c)$.
\end{itemize}
\end{definition}

The sieves $S\in J(c)$ are called the \emph{$J$-covering} sieves.

A \emph{site} is a pair $(\mathscr{C}, J)$ consisting of a small category $\mathscr{C}$ and a Grothendieck topology $J$ on $\mathscr{C}$.

\begin{definition}
Let $(\mathscr{C}, J)$ be a site.
\begin{enumerate}[(a)] 
\item A \emph{presheaf} on a category $\mathscr{C}$ is a functor $P:\mathscr{C}^{\textrm{op}}\to \Set\footnote{This is the category whose objects are sets and whose morphisms are functions between sets.}$.

\item A \emph{sheaf} on $(\mathscr{C}, J)$ is a presheaf $P:\mathscr{C}^\textrm{op}\to \Set$ on $\mathscr{C}$ such that for every $J$-covering sieve $S\in J(c)$ and every family $\{x_{f}\in P(\dom(f)) \mid f\in S\}$ such that $P(g)(x_f)=x_{f\circ g}$ for any $f\in S$ and any arrow $g$ in $\mathscr{C}$ composable with $g$ there exists a unique element $x\in P(c)$ such that $x_{f}=P(f)(x)$ for all $f\in S$.

\item The category $\Sh(\mathscr{C}, J)$ of sheaves on the site $(\mathscr{C}, J)$ has as objects the sheaves on $(\mathscr{C}, J)$ and as arrows the natural transformations between them, regarded as functors $\mathscr{C}^{\textrm{op}}\to \Set$.

\item A \emph{Grothendieck topos} is a category equivalent to a category $\Sh(\mathscr{C}, J)$ of sheaves on a site.
\end{enumerate}
\end{definition}

\begin{definition}\cite[section C2.1]{SE}
A sieve $R$ on an object $U$ of $\mathscr{C}$ is called \textit{effective-epimorphic} if it forms a colimit cone under the diagram consisting of the domains of all morphisms in $R$ and all the morphisms over $U$.
A Grothendieck topology is said to be \textit{subcanonical} if all its covering sieves are effective-epimorphic.
\end{definition}

A Grothendieck topology $J$ on $\mathscr{C}$ is subcanonical if and only if every representable functor $\mathscr{C}^\textrm{op}\to \Set$ is a $J$-sheaf.

\begin{definition}
Given a site $(\mathscr{C}, J)$, and a set $I$ of objects of $\mathscr{C}$. If for any arrow $f:a\to b$ in $\mathscr{C}$, $b\in I$ implies $a\in I$, we say that $I$ is an \textit{ideal}. If further for any $J$-covering sieve $S$ on an object $c$ of $\mathscr{C}$, if $dom(f)\in I$ for all $f\in S$ then $c\in I$, we say that $I$ is a \emph{$J$-ideal}.
\end{definition}

The $J$-ideals on $\mathscr{C}$ correspond bijectively to the subterminal objects of the topos $\Sh({\mathscr{C}}, J)$.

Given a site $(\mathscr{C}, J)$, a sieve $S$ on an object $c$ of $\mathscr{C}$ is said to be \emph{$J$-closed} if for every arrow $f$ with codomain $c$, $f^{\ast}(S)\in J(dom(f))$ implies $f\in S$. 

If the representable functor $Hom_{\mathscr{C}}(-, c)$ is a $J$-sheaf then the $J$-closed sieves on $c$ are in natural bijection with the subobjects of $Hom_{\mathscr{C}}(-, c)$ in the topos $\Sh(\mathscr{C}, J)$.  

\begin{definition}\cite[Definition C2.2.18]{SE}\label{rigidity}
A Grothendieck topology $J$ on a small category $\mathscr{C}$ is said to be \emph{rigid} if for every object $c$ of $\mathscr{C}$, the set of arrows from $J$-irreducible objects of $\mathscr{C}$ (i.e., the objects of $\mathscr{C}$ on which the only $J$-covering sieves are the maximal ones) generates a $J$-covering sieve. 
\end{definition}

\subsection{Geometric logic and classifying toposes}

From the point of view of toposes as classifying spaces for mathematical theories, the logic underlying Grothendieck toposes is geometric logic.  

\begin{definition}
A \textit{geometric theory} is a theory over a first-order signature $\Sigma$ whose axioms can be presented in the sequent form $(\phi\vdash_{\vec{x}}\psi)$ (intuitively meaning `$\phi$ entails $\psi$ in the context $\vec{x}$, i.e. for any values of the variables in $\vec{x}$'), where $\phi$ and $\psi$ are \emph{geometric formulas}, that is formulas with a finite number of free variables, all of which occurring in the context $\vec{x}$, built up from atomic formulas over $\Sigma$ by only using finitary conjunctions, infinitary disjunctions and existential quantifications.
\end{definition}

A geometric theory is said to be \emph{finitary algebraic} if its signature does not contain relation symbols and its axioms can be presented in the form $(\top \vdash_{\vec{x}} t=s)$, where $t$ and $s$ are terms over its signature.  

A geometric theory is said to be \emph{cartesian} if its axioms can be presented in the sequent form $(\phi\vdash_{\vec{x}}\psi)$, where $\phi$ and $\psi$ are \emph{$\mathbb T$-cartesian formulas}, that is formulas with a finite number of free variables, all of which occurring in the context $\vec{x}$, built up from atomic formulas by only using finitary conjunctions and $\mathbb T$-provably unique existential quantifications. Such sequents are called $\mathbb T$-cartesian sequents.  

A geometric theory is said to be \emph{coherent} if it is finitary, that is if its axioms can be presented in the sequent form $(\phi\vdash_{\vec{x}}\psi)$, where $\phi$ and $\psi$ are \emph{coherent formulas}, that is formulas with a finite number of free variables, all of which occurring in the context $\vec{x}$, built up from atomic formulas by only using finitary conjunctions, finitary disjunctions and existential quantifications. 

A \emph{quotient} of a geometric theory $\mathbb T$ over a signature $\Sigma$ is a geometric theory $\mathbb{T}'$ over $\Sigma$ such that every geometric sequent over $\Sigma$ which is provable in $\mathbb T$ is provable in $\mathbb{T}'$.

One can consider models of geometric theories in arbitrary Grothendieck toposes. Given a geometric theory $\mathbb T$ and a Grothendieck topos $\E$, we denote by $\mathbb{T}$-mod$(\E)$ the category of $\mathbb T$-models in $\E$ and model homomorphisms between them. 

\begin{definition}
A Grothendieck topos $\E$ \emph{classifies} a geometric theory $\mathbb{T}$ if there exists a categorical equivalence between the category of models of $\mathbb{T}$ in an arbitrary Grothendieck topos $\mathscr{F}$ and the category of geometric morphisms from $\mathscr{F}$ to $\E$, naturally in $\mathscr{F}$ or, in other words, if there exists a model $U$ of $\mathbb T$ in $\mathscr{E}$, called `the' universal model of $\mathbb T$ in $\mathscr{E}$, such that any other model of $\mathbb T$ in a topos $\mathscr{F}$ is isomorphic to $f^{\ast}(U)$ for a unique (up to isomorphism) geometric morphism $\mathscr{F}\to \mathscr{E}$. 
\end{definition}

By a fundamental theorem of Joyal-Reyes-Makkai, every geometric theory is classified by a unique (up to categorical equivalence) Grothen\-dieck topos. Conversely, every Grothendieck topos is the classifying topos of some geometric theory. If two geometric theories have the same classifying topos (up to equivalence), they are said to be \textit{Morita-equivalent}. By the universal property of the classifying topos, two theories are Morita-equivalent if and only if they have equivalent categories of models in every Grothendieck topos $\E$, naturally in $\E$. Notice that if the functors defining the equivalences between the categories of models of the two theories only involve \emph{geometric constructions}, i.e., constructions entirely expressible in terms of finite limits and arbitrary colimits, then the resulting equivalence is automatically natural. 

We indicate the classifying topos of a geometric theory $\mathbb{T}$ with the symbol $\mathscr{E}_\mathbb{T}$.

Classifying toposes can be built canonically by means of a syntactic construction.

\begin{definition}
Let $\mathbb T$ be a geometric theory over a signature $\Sigma$. The \emph{geometric syntactic category} $\mathscr{C}_{\mathbb T}$ of $\mathbb T$ has as objects the geometric formulas-in-context $\{\vec{x}.\phi\}$ over $\Sigma$ and as arrows from $\{\vec{x}.\phi\}$ to $\{\vec{y}.\psi\}$ the $\mathbb T$-provable equivalence classes $[\theta]$ of geometric formulas $\theta(\vec{x},\vec{y})$, where $\vec{x}$ and $\vec{y}$ are disjoint contexts, which are $\mathbb T$-provably functional from $\{\vec{x}.\phi\}$ to $\{\vec{y}.\psi\}$ in the sense that the sequents
\begin{itemize}
\item[-]$(\theta\vdash_{\vec{x},\vec{y}}(\phi\wedge\psi))$
\item[-]$(\theta\wedge \theta[\vec{z}/\vec{y}]\vdash_{\vec{x},\vec{y},\vec{z}}(\vec{z}=\vec{y}))$
\item[-]$(\phi\vdash_{\vec{x}}(\exists\vec{y})\theta)$
\end{itemize} 
are provable in $\mathbb{T}$.
\end{definition}

The classifying topos $\mathscr{E}_\mathbb{T}$ of a geometric theory $\mathbb T$ can always be represented as the category of sheaves $\Sh({\mathscr{C}_{\mathbb T}}, J_{\mathbb T})$ on the syntactic category ${\mathscr{C}_{\mathbb T}}$ of $\mathbb T$ with respect to the canonical topology $J_{\mathbb T}$ on it. A sieve $S=\{[\theta_{i}]:\{\vec{x_{i}}. \phi_{i}\} \to \{\vec{x}. \phi\} \mid i\in I\}$ in $\mathscr{C}_{\mathbb T}$ is $J_{\mathbb T}$-covering if and only if the sequent $(\phi \vdash_{\vec{x}} \mathbin{\mathop{\textrm{\huge $\vee$}}\limits_{i\in I}}(\exists \vec{x_{i}})\theta_{i})$ is provable in $\mathbb T$. 

The classifying topos of a coherent theory $\mathbb T$ can also be represented as the topos $\Sh({\mathscr{C}_{\mathbb T}}^{\textrm{coh}}, J_{\mathbb T}^{\textrm{coh}})$ of sheaves on the full subcategory $\mathscr{C}_{\mathbb T}^{\textrm{coh}}$ of ${\mathscr{C}_{\mathbb T}}$ on the coherent formulae with respect to the coherent topology $J_{\mathbb T}^{\textrm{coh}}$ on it (generated by finite $J_{\mathbb T}$-covering families).  

The classifying topos of a cartesian theory $\mathbb T$ can be represented as the presheaf topos $[{\mathscr{C}_{\mathbb T}^{\textrm{cart}}}^{\textrm{op}}, \Set]$, where $\mathscr{C}_{\mathbb T}^{\textrm{cart}}$ is the full subcategory of ${{\mathscr{C}_{\mathbb T}}}$ on the $\mathbb T$-cartesian formulas.

\subsection{The internal language of a topos}

Every topos can be regarded as a genera\-lized universe of sets, by means of its internal language. Recall that the \emph{internal language} of a topos $\mathscr{E}$ consists of a sort $\name{A}$ for each object $A$ of $\mathscr{E}$, a function symbol $\name{f}:\name{A_{1}}\dots\name{A_{n}}\to \name{B}$ for each arrow $f:A_{1}\times \cdots \times A_{n}\to B$ in $\mathscr{E}$ and a relation symbol $\name{R}\mono \name{A_{1}}\dots\name{A_{n}}$ for each subobject $R\mono A_{1}\times \cdots \times A_{n}$ in $\mathscr{E}$. There is a tautological $\Sigma_\mathscr{E}$-structure $\mathcal{S}_\mathscr{E}$ in $\mathscr{E}$, obtained by interpreting each $\name{A}$ as $A$, each $\name{f}$ as $f$ and each $\name{R}$ as $R$. For any objects $A_{1}, \ldots, A_{n}$ of $\mathscr{E}$ and any first-order formula $\phi(\vec{x})$ over $\Sigma_\mathscr{E}$, where $\vec{x}=(x_{1}^{\name{A_{1}}}, \ldots, x_{n}^{\name{A_{n}}})$, the expression $\{\vec{x}\in A_{1}\times \cdots \times A_{n} \mid \phi(\vec{x})\}$ can be given a meaning, namely the interpretation of the formula $\phi(\vec{x})$ in the $\Sigma_\mathscr{E}$-structure $\mathcal{S}_\mathscr{E}$.

The internal language allows to prove results concerning objects and arrows in the topos by formally arguing in an analogous way as we do in classical set theory. There are only two important exceptions to this rule which it is essential to keep in mind: the logic of a topos is sound only with respect to intuitionistic principles, so a classical proof can be lifted to a proof valid in an arbitrary topos written in its internal language only if it is constructive, in the sense of not involving applications of the law of excluded middle or the axiom of choice (nor of any other non-constructive principle). We shall exploit this fact at various points of the paper. An example of a reformulation of basic properties of sets in the internal language of a topos is provided by the following proposition:

\begin{proposition}\cite[cf. Lemma D1.3.11]{SE}\label{internal_language}
Let $\mathscr{E}$ be a topos. The following statements hold
\begin{itemize}
\item[(i)] $f:A\rightarrow A$ is the identity arrow if and only if $(\top\vdash_{x}f(x)=x)$ holds in $\mathscr{E}$.
\item[(ii)] $f:A\rightarrow C$ in the composite of $g:A\rightarrow B$ and $h:B\rightarrow C$ if and only if $(\top\vdash_{x}f(x)=h(g(x)))$ holds in $\mathscr{E}$.
\item[(iii)] $f:A\rightarrow B$ is monic if and only if $(f(x)=f(x')\vdash_{x}x=x')$ holds in $\mathscr{E}$.
\item[(iv)] $f:A\rightarrow B$ is an epimorphism if and only if $(\top\vdash_{x}(\exists x)(f(x)=y))$ holds in $\mathscr{E}$.
\item[(v)] $A$ is a terminal object if and only if the sequents $(\top\vdash (\exists x)\top)$ and $(\top\vdash_{x,x'}(x=x'))$ hold in $\mathscr{E}$.
\end{itemize}
\end{proposition}

\subsection{Theories of presheaf type}\label{review}

By definition, a \textit{theory of presheaf type} is a geometric theory whose classifying topos is (equivalent to) a topos of presheaves.

This class contains all the finitary algebraic (and, more generally, all the cartesian) theories as well as many other interesting, even infinitary, theories, such as the theory of lattice-ordered abelian groups with a distinguished strong unit considered in \cite{Russo} or the theory of algebraic extensions of a base field considered in \cite{Caramello5}. We shall see below that the theory $\mathbb P$ of perfect MV-algebras is also of presheaf type, it being Morita-equivalent to the cartesian theory $\mathbb L$ of lattice-ordered abelian groups.

In this section we recall some fundamental results on the class of geometric theories classified by a presheaf topos.  For a comprehensive investigation of this class of theories, containing various kinds of characterization theorems, we refer the reader to \cite{Caramello5}.  

\begin{definition}\cite{GU}
Let $\mathbb{T}$ be a geometric theory. A model $\mathcal{M}$ of $\mathbb{T}$ in $\Set$ is \textit{finitely presentable} if the representable functor $Hom(\mathcal{M},-):\mathbb{T}$-mod$(\Set)\rightarrow \Set$ preserves filtered colimits.
\end{definition}

As shown in \cite{Caramello7}, the classifying topos of a theory of presheaf type $\mathbb{T}$ can always be canonically represented as the functor category $[\textrm{f.p.}\mathbb{T}$-mod$(\Set),\Set]$, where $\textrm{f.p.}\mathbb{T}$-mod$(\Set)$ is the full subcategory of $\mathbb{T}$-mod$(\Set)$ on the finitely presentable $\mathbb T$-models.

\begin{definition}\cite{Caramello4}\label{fin_p.ati-cat}
Let $\mathbb{T}$ be a geometric theory over a one-sorted signature $\Sigma$ and $\phi(\vec{x})=\phi(x_1,\dots,x_n)$ be a geometric formula over $\Sigma$. We say that a $\mathbb{T}$-model $\mathcal{M}$ in $\Set$ is \textit{finitely presented} by $\phi(\vec{x})$ (or that $\phi(\vec{x})$ \textit{presents} $\mathcal{M}$) if there exists a string of elements $(a_1,\dots,a_n)\in \mathcal{M}^n$, called \textit{generators} of $\mathcal{M}$, such that for any $\mathbb{T}$-model $\mathcal{N}$ in $\Set$ and any string of elements $(b_1,\dots,b_n)\in [[\vec{x}.\phi]]_{\mathcal{N}}$, there exists a unique arrow $f:\mathcal{M}\rightarrow \mathcal{N}$ in $\mathbb{T}$-mod$(\Set)$ such that $f(a_i)=b_i$ for $i=1,\dots,n$. 
\end{definition}

This definition can be clearly generalized to multi-sorted theories.

The two above-mentioned notions of finitely presentability of a model coincide for cartesian theories (cf. pp. 882-883 \cite{SE}). More generally, as shown in \cite{Caramello2}, they coincide for all theories of presheaf type.

\begin{definition}
Let $\mathbb T$ be a geometric theory over a signature $\Sigma$ and $\phi(\vec{x})$ a geometric formula-in-context over $\Sigma$. Then $\phi(\vec{x})$ is said to be \emph{$\mathbb T$-irreducible} if for any family $\{\theta_{i} \mid i\in I\}$ of $\mathbb T$-provably functional geometric formulas $\{\vec{x_{i}}, \vec{x}.\theta_{i}\}$ from $\{\vec{x_{i}}. \phi_{i}\}$ to $\{\vec{x}. \phi\}$ such that $\phi \vdash_{\vec{x}} \mathbin{\mathop{\textrm{\huge $\vee$}}\limits_{i\in I}}(\exists \vec{x_{i}})\theta_{i}$ is provable in $\mathbb T$, there exist $i\in I$ and a $\mathbb T$-provably functional geometric formula $\{\vec{x}, \vec{x_{i}}. \theta'\}$ from $\{\vec{x}. \phi\}$ to $\{\vec{x_{i}}. \phi_{i}\}$ such that $\phi \vdash_{\vec{x}} (\exists \vec{x_{i}})(\theta' \wedge \theta_{i})$ is provable in $\mathbb T$. 
\end{definition}

We indicate with the symbol $\mathscr{C}_{\mathbb{T}}^{\textrm{irr}}$ the full subcategory of $\mathscr{C}_{\mathbb{T}}$ on $\mathbb{T}$-irreducible formulas. Notice that a formula $\{\vec{x}. \phi\}$ is $\mathbb T$-irreducible if and only if it is $J_{\mathbb T}$-irreducible as an object of the syntactic category $\mathscr{C}_{\mathbb T}$ of $\mathbb T$ (in the sense of Definition \ref{rigidity}).

\begin{theorem}
Let $\mathbb{T}$ be a geometric theory. Then $\mathbb T$ is of presheaf type if and only if the syntactic topology $J_{\mathbb T}$ on $\mathscr{C}_{\mathbb T}$ is rigid. 
\end{theorem}

\begin{theorem}\cite[Theorem 3.13]{Caramello2}\label{irr-f.p.}
Let $\mathbb{T}$ be a theory of presheaf type over a signature $\Sigma$. Then
\begin{itemize}
\item[(i)] Any finitely presentable $\mathbb{T}$-model in $\Set$ is presented by a $\mathbb{T}$-irreducible geometric formula $\phi(\vec{x})$ over $\Sigma$;
\item[(ii)] Conversely, any $\mathbb{T}$-irreducible geometric formula $\phi(\vec{x})$ over $\Sigma$ presents a $\mathbb{T}$-model.
\end{itemize}
In particular, the category $f.p.\mathbb{T}$-mod$(\Set)^{\textrm{op}}$ is equivalent to the full subcategory $\mathscr{C}^{\textrm{irr}}_{\mathbb{T}}$ of the geometric syntactic category $\mathscr{C}_{\mathbb{T}}$ of $\mathbb T$ on the $\mathbb{T}$-irreducible formulas.
\end{theorem}

\begin{remark}\label{rem:cartirr}
Given a cartesian theory $\mathbb T$, we saw above that its classifying topos can be represented as the presheaf topos $[{\mathscr{C}_{\mathbb T}^{\textrm{cart}}}^{\textrm{op}}, \Set]$. So $\mathbb T$ is a theory of presheaf type and the equivalence of classifying toposes $[{\mathscr{C}_{\mathbb T}^{\textrm{cart}}}^{\textrm{op}}, \Set]\simeq [{\mathscr{C}^{\textrm{irr}}_{\mathbb{T}}}^{\textrm{op}}, \Set]$ restricts, since both the categories $\mathscr{C}^{\textrm{irr}}_{\mathbb{T}}$ and ${\mathscr{C}_{\mathbb T}^{\textrm{cart}}}$ are Cauchy-complete, to an equiva\-lence of categories $\mathscr{C}^{\textrm{irr}}_{\mathbb{T}}\simeq {\mathscr{C}_{\mathbb T}^{\textrm{cart}}}$. Indeed, for any small category $\mathscr{C}$, the Cauchy-completion of $\mathscr{C}$ is equivalent to the full subcategory of $[\mathscr{C}^{\textrm{op}}, \Set]$ on its irreducible objects (cf. \cite{Caramello2}).
\end{remark}

Theories of presheaf type enjoy a very strong form of definability.

\begin{theorem}\label{definability}\cite[Corollary 3.2]{Caramello3}
Let $\mathbb{T}$ be a theory of preshef type and suppose that we are given, for every finitely presentable $\Set$-model $\mathcal{M}$ of $\mathbb{T}$, a subset $R_{\mathcal{M}}$ of $\mathcal{M}^n$ in such a way that every $\mathbb{T}$-model homomorphism $h:\mathcal{M}\rightarrow \mathcal{N}$ maps $R_{\mathcal{M}}$ into $R_{\mathcal{N}}$. Then there exists a geometric formula-in-context $\phi(x_1,\dots,x_n)$ such that $R_{\mathcal{M}}=[[\vec{x}.\phi]]_{\mathcal{M}}$ for each finitely presentable $\mathbb T$-model $\mathcal{M}$.
\end{theorem}

\begin{remark}\label{rmk:definability}
\begin{enumerate}[(a)]
\item The proof of the definability theorem in \cite{Caramello3} also shows that, for any two geometric formulas $\phi(\vec{x})$ and $\psi(\vec{y})$ over the signature of $\mathbb T$, every assignment $M\to f_{M}:[[\vec{x}. \phi]]_{M}\to [[\vec{y}. \psi]]_{M}$ (for finitely presentable $\mathbb T$-models $M$) which is natural in $M$ is definable by a $\mathbb T$-provably functional formula $\theta(\vec{x}, \vec{y})$ from $\phi(\vec{x})$ to $\psi(\vec{y})$.

\item If the property $R$ of tuples $\vec{x}$ of elements of set-based $\mathbb T$-models as in the statement of the theorem is also preserved by filtered colimits of $\mathbb T$-models then we have $R_{\mathcal{M}}=[[\vec{x}.\phi]]_{\mathcal{M}}$ for each set-based $\mathbb T$-model $M$, that is $R$ is definable by the formula $\phi(\vec{x})$.  

\item If $\mathbb T$ is coherent and the property $R$ is not only preserved but also reflected by arbitrary $\mathbb T$-model homomorphisms then the formula $\phi(\vec{x})$ in the statement of the theorem can be taken to be coherent and $\mathbb T$-Boolean (in the sense that there exists a coherent formula $\psi(\vec{x})$ in the same context such that the sequents $(\phi \vdash \psi \vdash_{\vec{x}} \bot)$ and $(\top \vdash_{\vec{x}} \phi \vee \psi)$ are provable in $\mathbb T$). Indeed, the theorem can be applied both to the property $R$ and to the negation of it yielding two geometric formulas $\phi(\vec{x})$ and $\psi(\vec{x})$ such that $(\phi \vdash \psi \vdash_{\vec{x}} \bot)$ and $(\top \vdash_{\vec{x}} \phi \vee \psi)$ are provable in $\mathbb T$. Hence, since every geometric formula is provably equivalent to a disjunction of coherent formulas and $\mathbb T$ is coherent, we can suppose $\phi$ and $\psi$ to be coherent without loss of generality (cf. \cite{Caramello2}).  
\end{enumerate}
\end{remark}

\section{Perfect MV-algebras}\label{perfect}

As a reference for this section use \cite{CDM}, if not otherwise specified.

\begin{definition}\label{MVdef}
An \textit{MV-algebra} is a structure $\mathcal{A}=(A,\oplus, \neg,0)$, where $\oplus$ is a binary relation symbol, $\neg$ is a unary relation symbol and $0$ is a constant, satisfying the following axioms:
\begin{enumerate}
\item $(x\oplus y)\oplus z=x\oplus (y\oplus z)$;
\item $x\oplus y= y\oplus x$;
\item $x\oplus 0=x$;
\item $\neg \neg x=x$;
\item $x\oplus \neg 0=\neg 0$;
\item $\neg (\neg x\oplus y)\oplus y=\neg (\neg y\oplus x)\oplus x$.
\end{enumerate}
\end{definition}

One can define in $\mathcal{A}$ the following derived operations:
\begin{itemize}
\item[-] $x\odot y:= \neg(\neg x\oplus \neg y)$
\item[-] $\sup(x, y):= (x\odot \neg y)\oplus y$
\item[-] $\inf(x,y):= (x\oplus \neg y)\odot y$
\item[-] $1:=\neg 0$
\end{itemize}

We write $x\leq y$ if $\inf(x,y)=x$; this relation defines a partial order called \textit{natural order} of $\mathcal{A}$. In the sequel we will use the notations $\inf$ or $\wedge$ and $\sup$ or $\vee$ to indicate respectively the infimum and the supremun of two or more elements in an MV-algebra.

\begin{lemma}\cite[Lemma 1.1.2]{CDM}\label{lemma:order}
Let $\mathcal{A}$ be an MV-algebra and $x,y\in A$. Then the following conditions are equivalent:
\begin{itemize}
\item[(i)] $\neg x\oplus y=1$;
\item[(ii)] $x\odot\neg y=0$;
\item[(iii)] there is an element $z\in A$ such that $x\oplus z=y$;
\item[(iv)] $x\leq y$.
\end{itemize}
\end{lemma}

We write $nx$ for $x\oplus \dots \oplus x$ ($n$ times) and $x^n$ for $x\odot \dots \odot x$ ($n$ times). The least integer for which $nx=1$ is called the \textit{order} of $x$. When such an integer exists, we denote it by $\textrm{ord}(x)$ and we say that $x$ has \textit{finite order}; otherwise we say that $x$ has \textit{infinite order} and we write $\textrm{ord}(x)=\infty$.

The equations of Definition \ref{MVdef}, written in form of sequents, give the axioms of the geometric theory of MV-algebras which we indicate with the symbol $\mathbb{MV}$. Its signature $\mathcal{L}_{MV}$ consists of two function symbols ($\oplus$, $\neg$) and a constant symbol $0$. The axioms of $\mathbb{MV}$ are the following sequents written over $\mathcal{L}_{MV}$:

\begin{itemize}
\item[MV.1] $\top \vdash_{x,y,z}x\oplus(y\oplus z)=(x\oplus y)\oplus z$
\item[MV.2] $\top\vdash_{x,y}x\oplus y=y\oplus x$
\item[MV.3] $\top\vdash_{x}x\oplus 0=x$
\item[MV.4] $\top\vdash_{x}\neg \neg x=x$
\item[MV.5] $\top\vdash_{x}x\oplus \neg 0=\neg 0$
\item[MV.6] $\top\vdash_{x,y}\neg(\neg x\oplus y)\oplus y=\neg(\neg y\oplus x)\oplus x$
\end{itemize}

\begin{example}
Let $[0,1]$ be the unit interval of real numbers. Consider the  operations
\begin{itemize}
\item $x\oplus y:=\min\{1,x+y\}$
\item $\neg x:=1-x$
\end{itemize}
The structure $([0,1],\oplus,\neg,0)$ is an MV-algebra. We shall refer to it as to the \textit{standard MV-algebra}; in fact, this algebra generates the variety of MV-algebras.
\end{example}

The congruence relations on an MV-algebra $\mathcal{A}=(A,\oplus,\neg,0)$ can be identified with its \textit{ideals}, i.e. the non-empty subsets $I$ of $A$ satisfying the conditions
\begin{itemize}
\item[I.1] $0\in I$;
\item[I.2] if $x\in I,y\in A$ and $y\leq x$, then $y\in I$;
\item[I.3] if $x,y\in I$, then $x\oplus y\in I$.
\end{itemize}

The intersection of all the maximal ideals (in the sense of the inclusion relation) of an MV-algebra $\mathcal{A}$ is called the \textit{radical} of $\mathcal{A}$ and denoted by $Rad(\mathcal{A})$.

The radical of an MV-algebra can also be characterized as the set of all the \textit{infinitesimal} elements plus $0$, where by infinitesimal we mean an element $x$ such that $x\neq 0$ and $nx\leq \neg x$ for every integer $n\geq 0$. 

\begin{obs}
Note that every infinitesimal element $x$ has infinite order. Indeed, for each $n\in \mathbb{N}$, we have $nx\leq \neg x<1$.
\end{obs}

If an MV-algebra is generated by its radical, we say that it is perfect. More specifically, we have the following definition.

\begin{definition}
An MV-algebra $\mathcal{A}=(A,\oplus,\neg,0)$ is said to be \textit{perfect} if $\mathcal{A}$ is non-trivial (i.e., $A\neq \{0\}$ or equivalently $1\neq0$) and $A=Rad(\mathcal{A})\cup \neg Rad(\mathcal{A})$, where $\neg Rad(\mathcal{A})=\{x\in A\mid \neg x\in Rad(\mathcal{A})\}$.
\end{definition}

The set $\neg Rad(\mathcal{A})$ is called the \emph{coradical} of $\mathcal{A}$ and it is also denoted by $Corad(\mathcal{A})$ 

Chang's MV-algebra is the prototype of perfect MV-algebras, in the sense that it is a perfect MV-algebra and every perfect MV-algebra is contained in the variety $V(C)$, called \textit{Chang's variety}, generated by it (cf. \cite[Proposition 5(5)]{DiNola1}). It is defined on the following infinite set of formal symbols
\begin{center}
$C=\{0,c,\dots,nc,\dots,1-nc,\dots,1-c,1\}$
\end{center}
with the following operations:
\begin{itemize}
\item 
$x\oplus y:= \begin{cases}
(m+n)c & \textrm{if }x=nc\textrm{ and }y=mc\\
1-(n-m)c & \textrm{if }x=1-nc\textrm{, }y=mc\textrm{ and }0<m<n \\
1 & \textrm{if }x=1-nc\textrm{, } y=mc\textrm{ and }0<n\leq m \\
1 & \textrm{if }x=1-nc\textrm{, }y=1-mc
\end{cases}$

\item 
$\neg x:= 1-x$
\end{itemize}

\begin{remark}
Chang's algebra $C$ is the image $\Sigma(\mathbb{Z})$ under Di Nola-Lettieri's equiva\-lence of the lattice-ordered abelian group $\mathbb{Z}$ (cf. section \ref{equivalence in set}).
\end{remark}

Recalling that the elements of the radical are the infinitesimals plus 0, we have:
\begin{itemize}
\item[-] $Rad(C)=\{nc\mid n\in \mathbb{N}\}$ and $m(nc)=mnc$; thus, $ord(nc)=\infty$
\item[-] $\neg Rad(C)=\{1-nc\mid n\in\mathbb{N}\}$ and $(1-nc)\oplus (1-nc)=1$; thus, $ord(1-nc)=2$.
\end{itemize}

This is not a case. Indeed, in any perfect MV-algebra the radical contains all the elements with infinite order while the coradical contains all the elements with finite order (cf.  \cite[Proposition 5(8)]{DiNola1}). In fact, as we shall see below, all the elements of the coradical of an arbitrary perfect MV-algebra have at most order $2$.

By Proposition 5(6) \cite{DiNola1}, $V(C)$ is axiomatized by the sequent $\xi:(\top \vdash_{x} 2x^2=(2x)^2)$. Moreover, using the axiom of choice it can be proved that all the algebras in $V(C)$ satisfy the sequent $(\top \vdash_{x} 2(2x)^{2}=(2x)^{2})$, expressing the property that every element of the form $2x^2$ is Boolean (cf. Claim 1 in the proof of \cite[Theorem 5.8]{P-L}). We shall denote by $\mathbb C$ the quotient of $\mathbb{MV}$ obtained by adding these two sequents. Note that the models of $\mathbb C$ in a classical set-theoretic universe $\Set$ satisfying the axiom of choice coincide exactly with the MV-algebras in $V(C)$ (by Proposition 5(6)\cite{DiNola1}).

\begin{lemma}\label{lemseq}
The sequent 
\[
\gamma_{n}: (2^nx=1 \vdash_{x} 2x=1)
\]
is provable in $\mathbb C$. 

In particular, every element of finite order of an MV-algebra in ${\mathbb C}\textrm{-mod}(\Set)$ has order at most $2$.
\end{lemma}

\begin{proof}
By the proof of \cite[Theorem 3.9]{Chang}, for each natural number $n$, the sequent $\chi_{n}:(nx^2=1 \vdash 2x=1)$ is provable in the theory $\mathbb{MV}$. 

First, let us prove that the sequent $(\top \vdash_{x} 2^nx^2=(2^nx)^2)$ is provable in $\mathbb C$ by induction on $n$. For $n=1$, it is a tautology. For $n> 1$, we argue (informally) as follows. We have $2^nx^2=2(2^{n-1}x^{2})=2((2^{n-1}x)^{2})=(2(2^{n-1}x))^{2}=(2^{n}x)^{2}
$, where the second equality follows from the induction hypothesis and the third follows from sequent $\xi$.

Now, $(2^nx=1 \vdash_{x} 2x=1)$ is provable in $\mathbb C$ since $2^nx=1$ implies $(2^nx)^{2}=1$. But $(2^nx)^{2}=2^nx^2$. So $2^nx^2=1$ whence, by sequent $\chi_{2^n}$, $2x=1$, as required.   
\end{proof}

\begin{remark}
Let $\mathcal{A}$ be a perfect MV-algebra. For any $x\in \neg Rad(\mathcal{A})$ not equal to $1$, the order of $x$ is equal to $2$. Indeed, as we have already observed, every perfect MV-algebra is in the variety $V(C)$ and the coradical of a perfect MV-algebra contains only elements of finite order. Hence our claim follows from Lemma \ref{lemseq}.
\end{remark}

\begin{lemma}\label{lemmaP3}
The following sequent is provably entailed by the non-triviality axiom $(0=1 \vdash \bot)$ in the theory $\mathbb C$: 
\begin{center}
$\alpha:(x=\neg x\vdash_{x}\perp)$.
\end{center}
\end{lemma}
\begin{proof}

Given a non-trivial $\mathbb C$-model $\mathcal{A}$, suppose that there is an $x\in A$ such that $x=\neg x$; thus, $x\oplus x=1$.
By axiom $\xi$, we have that $x^2\oplus x^2=(x\oplus x)^2$; but

$(x\oplus x)^2=1$

$x^2\oplus x^2=(\neg (\neg x\oplus \neg x))\oplus (\neg (\neg x\oplus \neg x))=0\oplus 0=0$

This is a contradiction since $\mathcal{A}$ is non-trivial.
\end{proof}

\begin{lemma}\label{lemma_radical2}
The sequent $\alpha $ holds in every perfect MV-algebra.
\end{lemma}
\begin{proof}
This result trially follows from previous lemma since every perfect MV-algebra is non-trivial and in ${\mathbb C}\textrm{-mod}(\Set)$ (cf. \cite[Proposition 5(5)]{DiNola1}).
\end{proof}

The class of perfect MV-algebras is a first-order definable subclass of the variety of MV-algebras. Indeed,  \cite[Proposition 6]{DiNola1} states, assuming classical logic, that a MV-algebra $\mathcal{A}$ is perfect if and only if it is non-trivial and it satisfies the sequents:
\begin{itemize}
\item $\sigma:(\top\vdash_x x^2\oplus x^2=(x\oplus x)^2)$
\item $\tau:(x^2=x\vdash_x x=0\vee x=1)$
\end{itemize}

Note that the non-triviality condition can be expressed by the sequent $(0=1 \vdash \bot)$ or, equivalently by Lemma \ref{lemma_radical2}, by the sequent  $(x=\neg x\vdash_{x} \bot)$. 

Let us consider the following sequents:
\begin{itemize}
\item[P.1] $\top\vdash_{x}x^2\oplus x^2=(x\oplus x)^2$
\item[P.2] $\top \vdash_{x} 2(2x)^{2}=(2x)^{2}$ 
\item[P.3] $x\oplus x=x\vdash_{x}x=0\vee x=1$
\item[P.4] $x=\neg x\vdash_{x}\perp$
\end{itemize}

\begin{theorem}\label{axiomatization}
The family of sequents $\{\textrm{P.1}, \textrm{P.2}, \textrm{P.3}\}$ is provably equivalent in the geometric theory $\mathbb{MV}$ to the family of sequents $\{\textrm{P.1}, \beta\}$, where
\[
\beta: (\top\vdash_{x}x\leq \neg x\vee \neg x\leq x).
\]
\end{theorem}
\begin{proof}
To prove that sequent P.2 follows from $\beta$ and P.1 it suffices to observe that if $x\leq \neg x$ then $x^2=0$, while if $\neg x\leq x$ then $2x=1$. 

Let us now show that $\beta$ and P.1 entail sequent P.3. Given $x$ such that $x\oplus x=x$ (equivalently $x\odot x=x$, cf. \cite[Theorem 1.16]{Chang}), we know from the sequent $\beta$ that $x\leq \neg x$ or $\neg x\leq x$. Recall that $x\leq y$ iff $\neg x\oplus y=1$ iff 
$x\odot\neg y=0$. Hence, if $x\leq\neg x$ then $x\odot x=0$ whence $x=0$. On the other hand, if $\neg x\leq x$ then $x\oplus x=1$ whence $x=1$. This proves sequent P.3.

Conversely, let us show that the family of sequents $\{\textrm{P.1}, \textrm{P.2}, \textrm{P.3}\}$ entails $\beta$. Given $x$, the element $2x^2$ is Boolean by sequent P.2, while by sequent P.1, $2x^2=(2x)^{2}$. Sequent P.3 thus implies that either $2x^2=0$ or $(2x)^{2}=1$. But $2x^2=0$ clearly implies $x^2=0$, which is equivalent to $x\leq \neg x$, while $(2x)^{2}=1$ implies $2x=1$, which is equivalent to $\neg x\leq x$.   
\end{proof}

Let us define the geometric theory $\mathbb{P}$ of perfect MV-algebras as the quotient of the theory $\mathbb{MV}$ obtained by adding as axioms sequents P.1, P.2, P.3 and P.4.



The radical of a perfect MV-algebra is definable by a first-order formula, as shown by the following more general result.

\begin{proposition}\label{radical}
Let $\mathcal{A}$ be a MV-algebra in ${\mathbb C}\textrm{-mod}(\mathscr{E})$ (in particular, a perfect MV-algebra). Then
\begin{center}
$Rad(\mathcal{A})=\{x\in A\mid x\leq \neg x\}$.
\end{center}
\end{proposition}

\begin{proof}
We shall verify that the sequent $(\top \vdash_{x}x^2\oplus x^2=(x\oplus x)^2)$ entails the sequents $(x\leq \neg x \vdash_{x} nx\leq \neg x )$ for each $n\in {\mathbb N}$. This will imply our thesis by soundness.

By the equivalence between the quotients of the theory $\mathbb{MV}$ and the theory ${\mathbb L}_{u}$ established in \cite{Russo}, it is equivalent to prove that the image of the sequent $(\top \vdash_{x}x^2\oplus x^2=(x\oplus x)^2)$ under the interpretation functor $I:\mathscr{C}_{\mathbb{MV}}\to \mathscr{C}_{{\mathbb L}_{u}}$ defined in \cite{} entails the image of each of the sequents  $(x\leq \neg x \vdash_{x} nx\leq \neg x )$ in the theory ${\mathbb L}_{u}$. 

Now, the image of the sequent $(\top \vdash_{x}x^2\oplus x^2=(x\oplus x)^2)$ under $I$ is the sequent $(0\leq x\wedge x\leq u\vdash_{x} \sup(0,2\inf(2x,u)-u)=\inf(u,2\sup(2x-u,0))$, while the image under $I$ of the sequent $(x\leq \neg x \vdash_{x} nx\leq \neg x )$ is $( 0\leq x\wedge x\leq u \wedge x\leq (u-x) \vdash_{x} \inf(u,nx)\leq (u-x))$. It is readily seen that the former sequent entails the sequent $(2x\leq u\vdash_{x} 4x\leq u)$ and hence, by induction, that of the sequent $(2x\leq u\vdash_{x} 2^{n}x\leq u)$ for each $n\geq 1$, which in turns entails that of the latter sequent, as required. 
\end{proof}

The following lemma gives a list of sequents that are provable in the theory $\mathbb P$ and which therefore hold in every perfect MV-algebra by soundness.

\begin{lemma}\label{rad_ideal}
The following sequents are provable in $\mathbb P$:
\begin{itemize}
\item[(i)] $(x\leq \neg x\wedge y\leq x\vdash_{x,y}y\leq \neg y)$.
\item[(ii)] $(\neg z\leq z\vdash_{z}\neg z^2\leq z^2)$.
\item[(iii)] $(z\leq \neg z\vdash_{z}2z\leq \neg 2z)$.
\item[(iv)] $(z^2\leq \neg z^2\vdash_{z}z\leq \neg z)$.
\item[(v)] $(x\leq \neg x\wedge y\leq \neg y\vdash_{x,y}\sup(x,y)\leq \neg \sup(x,y))$.
\item[(vi)] $(x\leq \neg x\wedge y\leq \neg y\vdash_{x,y}\inf(x,y)\leq \neg \inf(x,y))$.
\item[(vii)] $(x\leq \neg x\wedge y\leq \neg y\vdash_{x,y}x\oplus y\leq \neg (x\oplus y))$.
\item[(viii)] $(\neg x\leq x\wedge \neg y\leq y\vdash_{x,y}x\oplus y=1)$.
\item[(ix)] $(x\leq \neg x\wedge \neg y\leq y\vdash_{x,y}x\leq y)$.
\end{itemize}
\end{lemma}
\begin{proof}
In the proof of this lemma we shall make an extensive use of the equivalent definitions of the natural order given by Lemma \ref{lemma:order}.
\begin{itemize}
\item[(i)] Given $x\leq \neg x$ and $y\leq x$, we have that:
\begin{center}
$y\leq x\Rightarrow y\odot y\leq x\odot x\Rightarrow y\odot y=0\Leftrightarrow y\leq \neg y$
\end{center}
\item[(ii)] If $\neg z\leq z$, from axiom P.1 and from identities that are provable in the theory of MV-algebras we have that:
\begin{center}
$0=(2\neg z)^2=(\neg z\oplus \neg z)\odot (\neg z\oplus \neg z)=$

$=(\neg z^2)\odot (\neg z^2)=$

$=\neg (z^2\oplus z^2)$,
\end{center}
which means that $\neg z^2\leq z^2$.

\item[(iii)] Given $z\leq \neg z$, we want to prove that $2z \leq \neg (2z)$. But this is equivalent to $(2z)^{2}=0$, which follows from $z^{2}=0$ (which is equivalent to $z\leq \neg z$) since $2z^{2}=(2z)^{2}$ by axiom P.1. 

\item[(iv)] Given $z^2\leq \neg z^2$, by axiom $\beta$ either  $z\leq \neg z$ or $\neg z\leq z$. If $\neg z\leq z$, by point (ii) $\neg z^2\leq z^2$ and hence $\neg z^2 = z^2$. But from Lemma \ref{lemma_radical2} we know that it is false, whence $z\leq \neg z$.
\item[(v)] Given $x\leq \neg x$ and $y\leq \neg y$, we have already observed that $x^2=0$ and $y^2=0$. From this it follows that $\sup(x,y)^3=0$ whence $\sup(x,y)^4=0$. Indeed, by using the identity $x\odot \sup(y,z)= \sup(x\odot y,x\odot z)$ (cf. \cite[Lemma 1.1.6(i)]{CDM}), we obtain that 
\begin{center}
$\sup(x,y)^3=\sup(x,y)\odot \sup(x,y)\odot \sup(x,y)=$

$=\sup(x,y)\odot \sup(\sup(x,y)\odot x,\sup(x,y)\odot y)=$

$=\sup(x,y)\odot \sup(\sup(x^2,x\odot y),\sup(x\odot y,y^2))=$

$=\sup(x,y)\odot \sup(x\odot y,x\odot y)=$

$=\sup(x,y)\odot (x\odot y)=$

$=\sup(x^2,x\odot y)\odot y=$

$=x\odot y\odot y=0$
\end{center}
If $\sup(x,y)^4=0$, then $\sup(x,y)^2 \leq \neg (\sup(x,y)^2) $. By point (iv) we thus have that $\sup(x, y)\leq \neg \sup(x, y)$, as required.
\item[(vi)] Given $x\leq \neg x$ and $y\leq \neg y$, since $\inf(x,y)\leq x,y$, the thesis follows from point (i).
\item[(vii)] Given $x\leq \neg x$ and $y\leq \neg y$, we know from points (iii) and (v) that $2\sup(x, y) \leq \neg (2\sup(x, y))$. But $x\oplus y\leq 2\sup(x,y)$, whence the thesis follows from point (i).
\item[(viii)] Given $\neg x\leq x$ and $\neg y\leq y$, by point (vi) we have that 
\begin{center}
$\neg x\oplus \neg y\leq \neg (\neg x\oplus \neg y)\Leftrightarrow$

$\neg (x\odot y)\leq (x\odot y)\Leftrightarrow$

$(x\odot y)\oplus (x\odot y)=1$
\end{center}
But $(x\odot y)\oplus (x\odot y)=1$ implies $x\oplus y=1$ (cf. \cite[Theorem 3.8]{Chang}), as required.
\item[(ix)] Given $x\leq \neg x$ and $\neg y\leq y$, by point (viii) we have that 
\begin{center}
$\neg x\oplus y=1\Leftrightarrow x\leq y$,
\end{center} 
as required.
\end{itemize}
\end{proof}

\begin{remark}
It will follow from Proposition \ref{prop:cartesianization} that each of the sequents in the statement of the lemma are already provable in the theory $\mathbb C$.
\end{remark}

Let $\mathscr{E}$ be an arbitrary Grothendieck topos. An MV-algebra in $\mathscr{E}$ is a structure $\mathcal{A}=(A,\oplus,\neg 0)$ in $\mathscr{E}$ that satisfies the axioms MV.1-MV.6. If this structure satisfies also the axioms P.1, P.2, P.3 and P.4 we call it a perfect MV-algebra in $\mathscr{E}$. Recall that $A$ is an object in the topos $\E$ and the operations and the constant in $\mathcal{A}$ are arrows in the topos $\E$.
\begin{itemize}
\item $\oplus:A\times A\rightarrow A$
\item $\neg: A\rightarrow A$
\item $0:1\rightarrow A$
\end{itemize}

The arrows of $\mathbb{MV}$-mod$(\mathscr{E})$, as well as those of $\mathbb{P}$-mod$(\mathscr{E})$, are called MV-homomorphisms.

While arguing in the internal language of the topos, we shall adopt the following abbreviations:
\begin{itemize}
\item $1:=\neg 0$
\item $x\odot y:= \neg (\neg x\oplus \neg y)$
\item $\sup(x,y):=(x\odot \neg y)\oplus y$
\item $\inf(x,y):=\neg \sup(\neg x, \neg y)$
\item $x\leq y:=\neg x\oplus y=1$
\end{itemize}

\begin{obs}
If $\mathscr{E}=\mathbf{Set}$, a perfect MV-algebra in $\mathscr{E}$ is exactly a perfect MV-algebra in the traditional sense. 
\end{obs}
For any perfect MV-algebra $\mathcal{A}=(A,\oplus,\neg, 0)$ in $\mathscr{E}$, we define by using the internal language the subobject
\begin{center}
$Rad(\mathcal{A})=\{x\in A\mid x\leq \neg x\}\mono \mathcal{A}$,
\end{center}
and we call it the \textit{radical} of $\mathcal{A}$. Similarly, we define the subobject 
\begin{center}
$Corad(\mathcal{A})=\{x\in A\mid \neg x\leq x\}\mono \mathcal{A}$
\end{center}
and call it the \textit{coradical} of $\mathcal{A}$.
 
Notice that the union of the subobjects $Rad(\mathcal{A})$ and $Corad(\mathcal{A})$ is precisely the interpretation in $\mathcal{A}$ of the formula $\{x. x\leq \neg x \vee \neg x\leq x\}$. In particular, a perfect MV-algebra is generated by its radical also in an arbitrary Grothendieck topos $\E$.

\section{Lattice-ordered abelian groups}\label{groups}

\begin{definition}(cf. \cite{BK})
A \textit{lattice-orderd abelian group} ($\ell$-group, for brevity) is a structure $\mathcal{G}=(G,+,-,\leq, 0)$ such that $ (G,+,-, 0)$ is an abelian group, $(G,\leq)$ is a lattice-ordered set and the following \emph{translation invariance property} holds:
\[\text{for any } x,y,z\in G \qquad x\leq y \text{ implies }x+z\leq y+z. \]

\end{definition}
Any pair of elements $x$ and $y$ of an $\ell$-group has a supremum, which we indicate by $\sup(x,y)$, and an infimum, indicated by $\inf(x,y)$.  We also write $nx$ for $x+\dots+x$, $n$-times.  For each element $x$ of an $\ell$-group, one can define the \textit{positive part} $x^+$, the \textit{negative part} $x^-$, and the \textit{absolute value} $|x|$ as follows:
\begin{enumerate}
\item $x^+:= \sup(0,x)$;
\item $x^-:= \sup(0,-x)$;
\item $|x|:=x^+ + x^-=\sup(x,-x)$.
\end{enumerate}
Recall that, for every $x\in G$, $x=x^+-x^-$.

\begin{example}
A simple example of an $\ell$-group is given by the group of integers with the natural order $(\mathbb{Z}, +,\leq)$.
\end{example}

Given an $\ell$-group $\mathcal{G}$ with a distinguished element $u$, $u$ is said to be a \textit{strong unit} for $\mathcal{G}$ if the following properties are satisfied:
\begin{itemize}
\item $u\geq 0$;
\item for any positive element $x$ of $\mathcal{G}$ there is a natural number $n$ such that $x\leq nu$.
\end{itemize}

We shall refer to $\ell$-groups with strong unit simply as to $\ell$-u groups.

\begin{example}
The structure $(\mathbb{R},+,-,0,\leq)$ is clearly an $\ell$-group. Further, any strictly positive element of $\mathbb{R}$ is a strong unit, $\mathbb{R}$ being archimedean. 
\end{example}

Let $\mathcal{L}_g$ be the first-order signature consisting of  a relation symbol $\leq$, of a constant $0$, and of function symbols $+$, $-$, $\inf$ and $\sup$ formalizing the $\ell$-group operations. We denote by $\mathbb{L}$ the geometric theory of  $\ell$-groups, whose axioms are the following sequents:
\begin{itemize}
\item[L.1] $\top\vdash_{x,y,z} x+(y+z)=(x+y)+z$
\item[L.2] $\top\vdash_{x}x+0=x$
\item[L.3] $\top\vdash_{x}x+(-x)=0$
\item[L.4] $\top\vdash_{x,y}x+y=y+x$
\item[L.5] $\top\vdash_{x}x\leq x$
\item[L.6] $(x\leq y)\wedge(y\leq x)\vdash_{x,y}x=y$
\item[L.7] $(x\leq y)\wedge (y\leq z)\vdash_{x,y,z}x\leq z$
\item[L.8] $\top\vdash_{x,y} \inf(x,y)\leq x\wedge \inf(x,y)\leq y$
\item[L.9] $z\leq x\wedge z\leq y\vdash_{x,y,z} z\leq \inf(x,y)$
\item[L.10] $\top\vdash_{x,y}x\leq \sup(x,y) \wedge y\leq \sup(x, y)$
\item[L.11] $x\leq z \wedge y\leq z\vdash_{x,y,z} \sup(x,y)\leq z$
\item[L.12] $x\leq y\vdash_{x,y,t}t+x\leq t+y$
\end{itemize}

Extending the signature $\mathcal{L}_g$ by adding a new constant symbol $u$, we can define the theory of $\ell$-groups with strong unit $\mathbb{L}_u$, whose axioms are L.1-L.12 plus
\begin{itemize}
\item[L$_u$.1] $\top\vdash u\geq 0$
\item[L$_u$.2] $x\geq 0\vdash_{x} \mathbin{\mathop{\textrm{\huge $\vee$}}\limits_{n\in \mathbb{N}}} (x\leq nu)$
\end{itemize}

A model of $\mathbb{L}$ (respectively of $\mathbb{L}_u$) in $\mathscr{E}$ is called an $\ell$-group in $\mathscr{E}$ (resp. an $\ell$-group with strong unit) and the arrows of $\mathbb{L}$-mod$(\mathscr{E})$ (resp. of $\mathbb{L}_u$-mod$(\mathscr{E})$) are called $\ell$-homomorphisms (resp. unital $\ell$-homomorphisms).

An $\ell$-group in $\mathscr{E}$ is a structure $\mathcal{G}=(G,+,-,\leq,\inf,\sup,0)$ in $\mathscr{E}$ which satisfies the axioms L.1-L.12. Note that such a structure consists of an object $G$ in the topos $\E$ and arrows (resp. subobjects) in the topos interpreting the function (resp. the relation) symbols of the signature $\mathcal{L}_g$: 
\begin{itemize}
\item $+:G\times G\rightarrow G$ 
\item $-:G\rightarrow G$
\item $\leq\rightarrowtail G$
\item $\inf:G\times G\rightarrow G$
\item $\sup:G\times G\rightarrow G$
\item $0:1\rightarrow G$
\end{itemize}

\begin{obs}
An $\ell$-group in $\mathscr{E}$ is an $\ell$-group in the traditional sense if $\mathscr{E}=\mathbf{Set}$.
\end{obs}

\section{Equivalence in $\mathbf{Set}$}\label{equivalence in set}

In this section we briefly review the well-known equivalence between the category of perfect MV-algebras and that of $\ell$-groups established by Di Nola and Lettieri in \cite{P-L}. 

Let $\mathcal{G}$ be an $\ell$-group and $\mathbb{Z}\times_{lex} \mathcal{G}$ be the lexicographic product of the $\ell$-group $\mathbb{Z}$ of integers with $\mathcal{G}$. This is again an $\ell$-group, whose underlying set is the cartesian product $\mathbb{Z}\times G$, whose group operations are defined pointwise and whose order relation is given by the lexicographic order. The element $(1,0)$ is a strong unit of $\mathbb{Z}\times_{lex} \mathcal{G}$; hence, we can consider the MV-algebra $\Sigma(\mathcal{G}):=\Gamma(\mathbb{Z}\times_{lex} \mathcal{G},(1,0))$, where $\Gamma$ is the truncation functor from the category of $\ell$-groups with strong unit to the category of MV-algebras introduced by Mundici in \cite{Mundici}. By definition of lexicographic order, we have that 
\begin{center}
$\Sigma(G)=\{(0,x)\in \Gamma(\mathbb{Z}\times G)\mid x\geq 0\}\cup\{(1,x)\in \Gamma(\mathbb{Z}\times G)\mid x\leq 0\}$,
\end{center}
where $\Sigma(G)$ is the underlying set of $\Sigma(\mathcal{G})$. This MV-algebra is perfect; indeed, $\{(0,x)\in \Gamma(\mathbb{Z}\times G)\mid x\geq 0\}$ is the radical and $\{(1,x)\in \Gamma(\mathbb{Z}\times G)\mid x\leq 0\}$ is the coradical. If $h:\mathcal{G}\rightarrow \mathcal{G'}$ is an $\ell$-homomorphism, the function
\begin{center}
$h^*:(m,g)\in \mathbb{Z}\times_{lex} \mathcal{G}\rightarrow (m,h(g))\in \mathbb{Z}\times_{lex} \mathcal{G}$
\end{center}
is a unital $\ell$-homomorphism. We set $\Sigma(h)=h^*|_{\Gamma(\mathbb{Z}\times_{lex} \mathcal{G})}$. It is easily seen that $\Sigma$ is a functor.

In the converse direction, let $\mathcal{A}$ be a perfect MV-algebra.

\begin{lemma}\label{monoid}
For every MV-algebra $\mathcal{A}$, the structure $(Rad(\mathcal{A}),\oplus,\leq,\inf,\sup,0)$ is a cancellative lattice-ordered abelian monoid.
\end{lemma}
\begin{proof}
As an ideal of $\mathcal{A}$, the radical is a lattice-ordered abelian monoid. It is also cancellative (see Lemma 
3.2 \cite{P-L}).
\end{proof}

From a cancellative lattice-ordered abelian monoid $\mathcal{M}$ we can canoni\-cally define an $\ell$-group by mimicking the construction of the group of integers from the monoid of natural numbers. Its underlying set is the quotient of the product $\mathcal{M}\times \mathcal{M}$ by the equivalence relation which identifies the pairs $(x,y),(z,t)$ such that $x+t=y+z$, where $+$ is the sum operation of $\mathcal{M}$. Let $\Delta(\mathcal{A})$ be the $\ell$-group built from $Rad(\mathcal{A})$ by using this construction. Any  MV-homomorphism $f:\mathcal{A}\rightarrow \mathcal{A}'$  between perfect MV-algebras preserves the radical, the MV-operations and the natural order. Thus $f$ induces by restriction a homomorphism between the associated lattice-ordered abelian monoids, which in turn can be extended to a homomorphism between the corresponding $\ell$-groups, as follows:
\begin{center}
$\Delta(f):[x,y]\in \Delta(\mathcal{A})\rightarrow [f(x),f(y)]\in \Delta(\mathcal{A}')$
\end{center}
It is easy to prove that $\Delta$ is a functor.

The functors $\Sigma$ and $\Delta$ are categorical inverses to each other, i.e. $\Sigma(\Delta(\mathcal{G}))\simeq \mathcal{G}$ and $\Delta(\Sigma(\mathcal{A}))\simeq \mathcal{A}$ for every $\ell$-group $\mathcal{G}$ and every perfect MV-algebra $\mathcal{A}$, naturally in $\mathcal{G}$ and $\mathcal{A}$.


\section{Topos-theoretic generalization}\label{equivalence in topos}

The categorical equivalence reviewed in the previous section can be seen as an equivalence between the categories $\mathbb{P}$-mod($\mathbf{Set}$) and $\mathbb{L}$-mod($\mathbf{Set}$). In this section we show that it can be generalized to an equivalence between $\mathbb{P}$-mod($\mathscr{E}$) and $\mathbb{L}$-mod($\mathscr{E}$), for every Grothendieck topos $\mathscr{E}$, natural in $\E$.

In the following we shall denote models in $\mathscr{E}$ by calligraphic letters $\mathcal{G},\mathcal{A}$ and by $G,A$ their underlying objects.

\subsection{From $\mathbb{L}$-models to $\mathbb{P}$-models}
In every Grothendieck topos $\mathscr{E}$ there is an object generalizing the set of integers which we call $\mathbb{Z}_{\mathscr{E}}$. This object is the coproduct
$\mathbin{\mathop{\bigsqcup}\limits_{z\in \mathbb{Z}}}1$ of $\mathbb Z$ copies of the terminal object $1$ of the topos; 
we denote by $\{\chi_z:1\rightarrow \mathbin{\mathop{\bigsqcup}\limits_{z\in \mathbb{Z}}}1\mid z\in \mathbb{Z}\}$ the canonical coproduct arrows. This object is precisely the image of $\mathbb{Z}$ under the inverse image functor $\gamma_{\mathscr{E}}^*$ of the unique geometric morphism $\gamma_{\mathscr{E}}:\mathscr{E}\rightarrow \mathbf{Set}$. The $\ell$-group structure with strong unit of $\mathbb{Z}$ induces an $\ell$-group structure with strong unit on $\mathbb{Z}_{\mathscr{E}}$, since $\gamma_{\mathscr{E}}^*$ preserves it. In particular the total order relation $\leq$ on $\mathbb{Z}$ induces a total order relation on $\mathbb{Z}_{\mathscr{E}}$ which we indicate, abusing notation, also with the symbol $\leq$.
Note that $\mathbb{Z}_{\mathscr{E}}$ is a decidable object of $\mathscr{E}$, it being the image under $\gamma_{\E}$ of a decidable object, i.e. the equality relation on it is complemented. This allows to define the strict order $<$ as the intersection of $\leq$ with the complement of the equality relation.


Let $\mathcal{G}$ be an  $\ell$-group in $\mathscr{E}$. The \emph{lexicographic product} $\mathbb{Z}_{\mathscr{E}}\times_{lex} \mathcal{G}$ of $\mathbb{Z}_{\mathscr{E}}$ and $\mathcal{G}$ is an  $\ell$-group whose underlying object is the product $\mathbb{Z}_{\mathscr{E}}\times G$, whose group operations are defined componentwise and whose order relation is defined by using the internal language as follows:
\begin{center}
$(a,x)\leq(b,y)$ iff $(a< b)\vee (a=b \wedge x\leq y)$
\end{center}
Note that the infimum and the supremum of two ``elements'' are given by:
\begin{center}
$\inf((a,x),(b,y))= \begin{cases}
(a,x) & \textrm{if }a<b\\
(b,y) & \textrm{if }b<a\\
(a,\inf(x,y)) & \textrm{if }a=b 
\end{cases}$

\

\

$\sup((a,x),(b,y))= \begin{cases}
(a,x) & \textrm{if }a>b\\
(b,y) & \textrm{if }b>a\\
(a,\sup(x,y)) & \textrm{if }a=b 
\end{cases}$
\end{center}

The generalized element $<\chi_{1}, 0>:1 \to \mathbb{Z}_{\mathscr{E}}\times_{lex} \mathcal{G}$ yields a strong unit for the $\ell$-group $\mathbb{Z}_{\mathscr{E}}\times_{lex} \mathcal{G}$, which we denote, abusing notation, simply by $(1, 0)$.

\begin{proposition}
The lexicographic product $\mathbb{Z}_{\mathscr{E}}\times_{lex} \mathcal{G}$ is an $\ell$-group and $(1,0)$ is a strong unit for it.
\end{proposition}
\begin{proof}

It is easy to see that $\mathbb{Z}_{\mathscr{E}}\times_{lex} \mathcal{G}$ satisfies the axioms L.1-L.12. For instance, given $(a,x),(b,y)\in \mathbb{Z}_{\mathscr{E}}\times G$, we have

\begin{itemize}
\item[-]$(a,x)+(b,y)=(a+b,x+y)=$ by definition of sum

$=(a+b,y+x)=$ by L.4 in $\mathbb{Z}_{\mathscr{E}}$ and $\mathcal{G}$

$=(b,y)+(a,x)$;

\item[-]$(a,x)+(0,0)=$

$=(a+0,x+0)=(a,x)$ by L.2 in $\mathbb{Z}_{\mathscr{E}}$ and $\mathcal{G}$.
\end{itemize}
Thus L.2 and L.4 hold. In a similar way it can be shown that the other axioms of $\mathbb L$ hold. Finally, we have to prove that $(1,0)$ is a strong unit, i.e. that it satisfies the axioms L$_u$.1 and L$_u$.2. By definition of order in $\mathbb{Z}_{\mathscr{E}}\times_{lex} \mathcal{G}$, we have that $(1,0)\geq (0,0)$, thus L$_u$.1 holds. Given $(a,x)\geq (0,0)$, this means that $a\geq 0$. From axiom L$_u$.2 applied to $\mathbb{Z}_{\mathscr{E}}$ we know that $\mathbin{\mathop{\textrm{\huge $\vee$}}\limits_{n\in \mathbb{N}}}a\leq n1$. Therefore $\mathbin{\mathop{\textrm{\huge $\vee$}}\limits_{n\in \mathbb{N}}}(a,x)\leq n(1,0)$. Thus, L$_u$.2 holds too.
\end{proof}

We set $\Sigma(\mathcal{G}):=\Gamma(\mathbb{Z}_{\mathscr{E}}\times_{lex} \mathcal{G},(1,0))$, where $\Gamma$ is the unit interval functor from $\mathbb{L}_u$-mod$(\mathscr{E})$ to $ \mathbb{MV}$-mod$(\mathscr{E})$ introduced in \cite{Russo}. The structure $\Sigma(\mathcal{G})$ is thus an MV-algebra in $\mathscr{E}$ whose underlying object is $\Sigma(G)=\{(a,x)\in\mathbb{Z}_{\mathscr{E}}\times G\mid (0,0)\leq (a,x)\leq (1,0) \}$. 

\newpage
\begin{proposition}\label{Sigma}
The MV-algebra $\Sigma(\mathcal{G})$ in $\mathscr{E}$ is perfect. 
\end{proposition}
\begin{proof}

By Theorem \ref{axiomatization}, it suffices to prove that $\Sigma(\mathcal{G})$ satisfies axioms P.1, P.4 and $\beta$. Clearly, $\Sigma(\mathcal{G})$ satisfies $\beta$ and P.4 if and only if it is the disjoint union of its radical and its coradical. Let us prove this by steps:
\begin{itemize}
\item[\em Claim 1.] $\Sigma(G)=\{(0,x)\in \mathbb{Z}_{\mathscr{E}}\times G\mid x\geq 0\}\cup \{(1,x)\in \mathbb{Z}_{\mathscr{E}}\times G\mid x\leq 0\}$;
\item[\em Claim 2.] $Rad(\mathbb{Z}_{\mathscr{E}}\times_{lex} \mathcal{G})=\{(0,x)\in \mathbb{Z}_{\mathscr{E}}\times G\mid x\geq 0\}$;
\item[\em Claim 3.] $Corad(\mathbb{Z}_{\mathscr{E}}\times_{lex} \mathcal{G})=\{(1,x)\in \mathbb{Z}_{\mathscr{E}}\times G\mid x\leq 0\}$.
\end{itemize}
We shall argue informally in the internal language of the topos $\E$ to prove these claims.
\begin{itemize}
\item[\em Claim 1.] Given $(a,x)\in \Sigma(G)$, we have to prove that it belongs to $\{(0,x)\in \mathbb{Z}_{\mathscr{E}}\times G\mid x\geq 0\}$ or $\{(1,x)\in \mathbb{Z}_{\mathscr{E}}\times G\mid x\leq 0\}$. 
Recall that $(0,0)\leq (a,x)\leq (1,0)$. This implies that $0\leq a\leq 1$. In $\mathbb{Z}$ the following sequent holds
\begin{center}
$0\leq a\leq 1\vdash_{a} (a=0)\vee (a=1)$
\end{center}
This is a geometric sequent; thus, it holds in $\mathbb{Z}_{\mathscr{E}}$ too. If $a=0$, we have that $(0,0)\leq (a,x)$ whence $0\leq x$. This implies that $(a,x)\in \{(0,x)\in \mathbb{Z}_{\mathscr{E}}\times G\mid x\geq 0\}$. If instead $a=1$ we have that $(a,x)\leq (1,0)$, whence $x\leq 0$ and $(a,x)\in \{(1,x)\in \mathbb{Z}_{\mathscr{E}}\times G\mid x\leq 0\}$. 
 
\item[\em Claim 2.] Given $(a,x)\in \Sigma(G)$, if $(a,x)\in Rad(\mathbb{Z}_{\mathscr{E}}\times_{lex} \mathcal{G})$ then
\begin{center}
$(0,0)\leq (a,x)\leq (1,0)$;

\

$(a,x)\leq \neg (a,x)=(1-a,-x)$.
\end{center}
It follows that $(a,x)=(0,x)$ with $x\geq 0$. Conversely, for any $x\geq 0$, $(0,x)\leq \neg (0,x)=(1,-x)$; thus, $(0,x)\in Rad(\mathbb{Z}_{\mathscr{E}}\times_{lex} \mathcal{G})$.

\item[\em Claim 3.] The proof is analogous to that of Claim 2.
\end{itemize}

To conclude our proof, it remains to show that $\Sigma(\mathcal{G})$ satisfies axiom P.1. This is straightforward, using the decomposition of the algebra as the disjoint union of its radical and coradical, and left to the reader. 
\end{proof}

Let $h:\mathcal{G}\rightarrow \mathcal{G}'$ be an $\ell$-homomorphism in $\mathscr{E}$. We define the following arrow in $\mathscr{E}$ by using the internal language:
\begin{center}
$h^*:(a,x)\in \mathbb{Z}_{\mathscr{E}}\times G\rightarrow (a,h(x))\in \mathbb{Z}_{\mathscr{E}}\times G'$
\end{center}

This is trivially an $\ell$-homomorphism which preserves the strong unit $(1,0).$ We set $\Sigma(h):=\Gamma(h^*)=h^*|_{\Gamma(\mathbb{Z}_{\mathscr{E}}\times G)}$.

\begin{proposition}
$\Sigma$ is a functor from $\mathbb{L}$-mod$(\mathscr{E})$ to $\mathbb{P}$-mod$(\mathscr{E})$.
\end{proposition}
\begin{proof}
This easily follows from the fact that $\Gamma$ is a functor. 
\end{proof}

\newpage
\subsection{From $\mathbb{P}$-models to $\mathbb{L}$-models}

\begin{lemma}\label{monoid_topos}
The structure $(Rad(\mathcal{A}),\oplus \leq,\inf,\sup,0)$ is a cancellative lattice-ordered abelian monoid in $\mathscr{E}$, i.e.\ it is a model in $\mathscr{E}$ of the theory whose axioms are L.1-L.12 (except axiom L.3) plus 
\begin{itemize}
\item[C.] $(x+a=y+a\vdash_{x,y,a}x=y)$.
\end{itemize}
\end{lemma}
\begin{proof}
From Lemma \ref{rad_ideal}(i)-(ii)-(iii)-(v)-(vi)-(vii) it follows that $Rad(\mathcal{A})$ is a lattice-ordered abelian monoid. Given $x,y,a\in Rad(\mathcal{A})$ such that $x\oplus a=y\oplus a$, we have that 
\begin{center}
$\neg a\odot (x\oplus a)=\neg a\odot (y\oplus a)\Leftrightarrow\inf(\neg a, x)=\inf(\neg a, y)$.  
\end{center}
Lemma \ref{rad_ideal}(ix) thus implies that $x=y$. This completes the proof.
\end{proof}

In \cite{Russo} we showed how to construct an $\ell$-group from a cancellative lattice-ordered abelian monoid in $\mathscr{E}$, generalizing the classical Grothendieck construction to a topos-theoretic setting. The resulting group is called the \textit{Grothendieck group} of the monoid (cf. also section \ref{intermediary} below). We also showed how to extend a homomorphism between two such monoids to an $\ell$-group homomorphism between the associated Grothendieck groups. Applying this construction to the monoid $Rad(\mathcal{A})$, we obtain obtain an  $\ell$-group which we name $\Delta(\mathcal{A})$. The constant, the order relation and the operations on $\Delta(\mathcal{A})$ are defined as follows by using the internal language of the topos $\E$: given $[x,y],[h,k]\in \Delta(\mathcal{A})$
\begin{itemize}
\item $[x,y]+[h,k]:=[x\oplus h,y\oplus k]$;
\item $-[x,y]:=[y,x]$;
\item $Inf([x,y],[h,k]):=[\inf((x\oplus k),(y\oplus h)),y\oplus k]$;
\item $Sup([x,y],[h,k]):=[\sup((x\oplus k),(y\oplus h)),y\oplus k]$;
\item $[x,y]\leq [h,k]$ iff $\inf([x,y],[h,k])=[x,y]$;
\item $[0,0]$ is the identity element.
\end{itemize}

Recall that two ``elements'' $[x,y], [h,k]$ of this $\ell$-group ``coincide'' if and only if $x\oplus k=y\oplus h$.

Any MV-homomorphism $h:\mathcal{A}\rightarrow \mathcal{A}'$ between perfect MV-algebras preserves the natural order, thus $h(Rad(\mathcal{A}))\subseteq Rad(\mathcal{A}')$. Hence the arrow $h^*:=h|_{Rad(\mathcal{A})}:Rad(\mathcal{A})\rightarrow Rad(\mathcal{A}')$ is a lattice-ordered monoid homomorphism. We set
\begin{center}
$\Delta(h):(x,y)\in \Delta(\mathcal{A})\rightarrow [h^*(x),h^*(y)]\in \Delta (\mathcal{A}')$
\end{center}

\begin{proposition}
$\Delta$ is a functor from $\mathbb{P}$-mod$(\mathscr{E})$ to $\mathbb{L}$-mod$(\mathscr{E})$.
\end{proposition}
\begin{proof}
This follows by a straightforward computation.
\end{proof}

\subsection{Morita-equivalence}\label{Moritaeq}
In the previous section we have defined, for each Grothen-dieck topos $\E$, two functors
\begin{center}
$\Sigma: \mathbb{L}$-mod$(\mathscr{E})\rightarrow \mathbb{P}$-mod$(\mathscr{E})$

\

$\Delta: \mathbb{MV}$-mod$(\mathscr{E})\rightarrow \mathbb{L}$-mod$(\mathscr{E})$.
\end{center}

\begin{theorem}\label{equivalence}
For every Grothendieck topos $\mathscr{E}$, the categories $\mathbb{P}$-mod$(\mathscr{E})$ and $\mathbb{L}$-mod$(\mathscr{E})$ are naturally equivalent.
\end{theorem}
\begin{proof}
We have to define two natural isomorphisms
\begin{center}
$\varphi: 1_{\mathbb{L}}\rightarrow \Delta\circ \Sigma$,

\

$\beta: 1_{\mathbb{P}}\rightarrow \Sigma\circ\Delta$,
\end{center}
where $1_{\mathbb{L}}$ and $1_{\mathbb{P}}$ are, respectively, the identity functors on the categories $\mathbb{L}$-mod$(\mathscr{E})$ and $\mathbb{P}$-mod$(\mathscr{E})$.

Let $\mathcal{G}=(G,+,-,\leq,\inf,\sup,0)$ be an abelian $\ell$-group in $\mathscr{E}$. Let $\varphi_{\mathcal{G}}:\mathcal{G}\to (\Delta\circ\Sigma)(\mathcal{G})$ be the arrow defined by using the internal language of the topos as follows: 

\begin{center}
$\varphi_{\mathcal{G}}:g\in G\rightarrow [(0,g^+),(0,g^+-g)]\in \Delta(\Sigma(G))$
\end{center}

\item[\em Claim 1.] $\varphi_{\mathcal{G}}$ is monic. Indeed, for any elements $g_1,g_2\in G$ such that $\varphi_{\mathcal{G}}(g_1)=\varphi_{\mathcal{G}}(g_2)$, we have that $(0,g^+_1)+(0,g^+_2-g_2)=(0,g^+_2)+(0,g^+_1-g_1)$, whence $g_1=g_2$. The monicity of $\varphi_{\mathcal{G}}$ thus follows from Proposition \ref{internal_language}(iii).
\item[\em Claim 2.] $\varphi_{\mathcal{G}}$ is epic. Given $[(0,g_1),(0,g_2)]\in \Delta\circ \Sigma(G)$, the element $g_1-g_2$ satisfies $\varphi_{\mathcal{G}}(g_1-g_2)=[(0, (g_{1}-g_{2})^{+}), (0, (g_{1}-g_{2})^{+}-(g_{1}-g_{2}))]= [(0,g_1),(0,g_2)]$. Proposition \ref{internal_language}(iv) thus implies that $\varphi_{\mathcal{G}}$ is an epimorphisms.
\item[\em Claim 3.] $\varphi_{\mathcal{G}}$ preserves $+$ and $-$. This follows by direct computation.
\item[\em Claim 4.] $\varphi_{\mathcal{G}}$ preserves $\inf$ and $\sup$. Given $g_1,g_2\in G$,
\begin{center}
$\varphi_{\mathcal{G}}(\sup(g_1,g_2))=[(0,\sup(g_1,g_2)^+),(0,\sup(g_1,g_2)^+-\sup(g_1,g_2))]$;

\

$\sup(\varphi_{\mathcal{G}}(g_1),\varphi_{\mathcal{G}}(g_2))=[(0,g_1^++g_2^+),\inf((0,(g_1^++g_2^+-g_2),(g_2^++g_1^+-g_1)))]$.
\end{center} 
Now, the sequent
\begin{center}
$\top\vdash_{g_1,g_2}\sup(g_1,g_2)^++\inf((g_1^++g_2^+-g_2),(g_2^++g_1^+-g_1))=g_1^++g_2^++\sup(g_1,g_2)^+-\sup(g_1, g_2)$
\end{center}
is provable in $\mathbb{L}$, hence it holds in every $\mathbb L$-model by soundness. This ensures that $\varphi_{\mathcal{G}}$ preserves $\sup$. In a similar way it can be shown that $\varphi_{\mathcal{G}}$ preserves $\inf$.

By Claims 1-4 the arrow $\varphi_{\mathcal{G}}$ is an isomorphism in $\mathbb{L}$-mod$(\mathscr{E})$. Further, it is easy to prove that for any $\ell$-homomorphism $h:\mathcal{G}\rightarrow \mathcal{G}'$ in $\E$, the following square commutes:
\begin{center}
\begin{tikzpicture}
\node (a) at (0,0) {$\mathcal{G}'$};
\node (b) at (3,0) {$\Delta\circ \Sigma(\mathcal{G}')$};
\node (c) at (0,3) {$\mathcal{G}$};
\node (d) at (3,3) {$\Delta\circ\Sigma(\mathcal{G})$};
\draw[->](a) to node [below,midway]{$\varphi_{\mathcal{G}'}$} (b);
\draw[->](c) to node [above,midway] {$\varphi_{\mathcal{G}}$} (d);
\draw[->](c) to node [midway, left] {$h$} (a);
\draw[->](d) to node [midway, right] {$\Delta\circ\Sigma(h)$} (b);
\end{tikzpicture}
\end{center}

We set $\varphi$ equal to the natural isomorphism whose components are the $\varphi_{\mathcal{G}}$ (for every abelian $\ell$-group $\mathcal{G}$).

In the converse direction, let $\mathcal{A}$ be a perfect MV-algebra in $\mathscr{E}$. Recall that $A=Rad(\mathcal{A})\cup  Corad(\mathcal{A})$ and that the sequent P.1 holds in $\mathcal{A}$. We define the following arrow by using the internal language
\begin{center}
$\beta_{\mathcal{A}}:x\in A\rightarrow \begin{cases}
(0,[x,0]) & \textrm{for }x\in Rad(\mathcal{A})\\
(1,[0,\neg x]) & \textrm{for }x\in Corad(\mathcal{A})
\end{cases}\in \Sigma(\Delta(\mathcal{A}))$
\end{center}
 
Let us prove that $\beta_{\mathcal{A}}$ preserves $\oplus$. Given $x,y\in A$, we can distinguish three cases:
\begin{itemize}
\item[Case i.] $x,y\in Rad(\mathcal{A})$. By direct computation it follows at once that $\beta_{\mathcal{A}}(x\oplus y)=\beta_{\mathcal{A}}(x)\oplus \beta_{\mathcal{A}}(y)$.
\item[Case ii.] $x,y\in Corad(\mathcal{A})$. From Lemma \ref{rad_ideal}(viii) we have that $x\oplus y=1$; thus $\beta_{\mathcal{A}}(x\oplus y)=(1,[0,0])$. On the other hand, $\beta_{\mathcal{A}}(x)=(1,[0,\neg x])$ and $\beta_{\mathcal{A}}(y)=(1,[0,\neg y])$, whence $\beta_{\mathcal{A}}(x),\beta_{\mathcal{A}}(y)\in Corad(\Sigma(\Delta (\mathcal{A})))$ and $\beta_{\mathcal{A}}(x)\oplus \beta_{\mathcal{A}}(y)=(1,[0,0])$.
\item[Case iii.] $x\in Rad(\mathcal{A}),y\in \neg Rad(\mathcal{A})$. In a similar way we obtain that $\beta_{\mathcal{A}}(x\oplus y)=\beta_{\mathcal{A}}(x)\oplus \beta_{\mathcal{A}}(y)$. 
\end{itemize}

The fact that $\beta_{\mathcal{A}}$ preserves $\neg$ and is both monic and epic is clear. We can thus conclude that $\beta_{\mathcal{A}}$ is an isomorphism.
  
It is clear that if $h:\mathcal{A}\rightarrow \mathcal{A}'$ is an MV-homomorphism then following square commutes:
\begin{center}
\begin{tikzpicture}
\node (a) at (0,0) {$\mathcal{A}'$};
\node (b) at (3,0) {$\Sigma(\Delta(\mathcal{A}'))$};
\node (c) at (0,3) {$\mathcal{A}$};
\node (d) at (3,3) {$\Sigma(\Delta(\mathcal{A}))$};
\draw[->](a) to node [below,midway]{$\beta_{\mathcal{A}'}$} (b);
\draw[->](c) to node [above,midway] {$\beta_{\mathcal{A}}$} (d);
\draw[->](c) to node [midway, left] {$h$} (a);
\draw[->](d) to node [midway, right] {$\Sigma(\Delta(h))$} (b);
\end{tikzpicture}
\end{center}

Thus, we have a natural isomorphism $\beta$ whose components are the arrows $\beta_{\mathcal{A}}$ (for every perfect MV-algebra $\mathcal{A}$).
\end{proof}

Note that all the constructions that we used to define the functors $\Sigma$ and $\Delta$ are geometric. Hence, the categorical equivalence proved in the last theorem is natural in the topos $\E$. This implies that the classifying toposes $\mathscr{E}_{\mathbb{P}}$ and $\mathscr{E}_{\mathbb{L}}$ are equivalent, i.e., that the theories $\mathbb{P}$ and $\mathbb{L}$ are Morita-equivalent. Summarizing, we have the following

\begin{theorem}\label{Moritaeq1}
The functors $\Delta_{\E}$ and $\Sigma_{\E}$ yield a Morita-equivalence between the coherent theory $\mathbb P$ of perfect MV-algebras and the cartesian theory $\mathbb L$ of lattice-ordered abelian groups.
\end{theorem}

\section{An intermediary Morita-equivalence}\label{intermediary}

In this section we shall establish an auxiliary Morita-equivalence involving the theory $\mathbb L$ that we will be useful in the following section. This stems from the observation that the $\ell$-groups arising in the context of MV-algebras as the counterparts of MV-algebras via Mundici's functor, as well as those which correspond to perfect MV-algebras under Di Nola and Lettieri's equivalence, are determined by their positive cones. As we shall see in this section, one can naturally axiomatize the monoids arising as the positive cones of such groups in such a way as to obtain a theory Morita-equivalent to that of $\ell$-groups. 

Specifically, let $\mathcal{L}_M$ be the one-sorted first-order signature consisting of three function symbols $+$, $\inf$, $\sup$, a constant symbol $0$ and a derivable relation symbol: $x\leq y$ iff $\inf(x,y)=x$. Over this signature we define the theory $\mathbb{M}$, whose axioms are the following sequents:
\begin{itemize}
\item[M.1] $\top\vdash_{x,y,z}x+(y+z)=(x+y)+z$
\item[M.2] $\top\vdash_{x}x+0=0$
\item[M.3] $\top\vdash_{x,y}x+y=y+x$
\item[M.4] $\top\vdash_{x}x\leq x$
\item[M.5] $(x\leq y)\wedge (y\leq x)\vdash_{x,y}x=y$
\item[M.6] $(x\leq y)\wedge (y\leq z)\vdash_{x,y,z} x\leq z$
\item[M.7] $\top\vdash_{x,y} \inf(x,y)\leq x\wedge \inf(x,y)\leq y$
\item[M.8] $z\leq x\wedge z\leq y\vdash_{x,y,z} z\leq \inf(x,y)$
\item[M.9] $\top\vdash_{x,y}x\leq \sup(x,y) \wedge y\leq \sup(x, y)$
\item[M.10] $x\leq z \wedge y\leq z\vdash_{x,y,z} \sup(x,y)\leq z$
\item[M.11] $x\leq y\vdash_{x,y,t}t+x\leq t+y$
\item[M.12] $x+y=x+z\vdash_{x,y,z}y=z$
\item[M.13] $\top\vdash_{x}0\leq x$
\item[M.14] $x\leq y\vdash_{x,y}(\exists z)x+z=y$
\end{itemize} 

We call $\mathbb{M}$ the theory of \emph{cancellative subtractive lattice-ordered abelian monoids with bottom element}. This theory is cartesian; indeed, by the cancellation pro\-perty, the existential quantification of the axiom M.14 is provably unique.

Notice that the sequent $(x+z \leq y+z \vdash_{x,y,z} x\leq y)$ is provable in $\mathbb M$. From this it easily follows that the sequent $(\top\vdash_{a,b,c}\inf(a,b)+c=\inf(a+c,b+c))$ is also provable in $\mathbb{M}$.

The models of $\mathbb{M}$ are particular lattice-ordered abelian monoids. We shall prove that $\mathbb{M}$ is the theory of positive cones of $\ell$-groups. 

\begin{obs}
Let $\mathcal{A}$ be an arbitrary MV-algebra and $\mathcal{M}_{\mathcal{A}}$ the cancellative lattice-ordered abelian monoid of good sequences introduced by Mundici in \cite{Mundici}. This is a model of $\mathbb{M}$ in $\Set$. Indeed, the axioms M.1-M.13 are trivially satisfied, while axiom M.14 holds by \cite[Proposition 2.3.2]{CDM}.
\end{obs}
Let $\mathcal{M}=(M,+,\leq,\inf,\sup,0)$ be a model of $\mathbb{M}$ in an arbitrary Grothedieck topos $\E$. We showed in \cite{Russo} how to construct the lattice-ordered Grothendieck group $G({\mathcal{M}})$ associated to $\mathcal{M}$. Specifically, the underlying object of $G({\mathcal{M}})$ is the quotient of $M\times M$ under the following equivalent relation: $(x,y)\sim (h,k)$ if and only if $x+k=y+h$. This equivalence relation, as well as the operations and the order relation below, is defined by using the internal language of the topos. The operations are defined as follows:
\begin{itemize}
\item $[x,y]+[h,k]:=[x+h,y+k]$
\item $-[x,y]:=[y,t]$
\item $\textrm{Inf}([x,y],[h,k]):=[\inf((x+ k),(y+ h)),y+ k]$;
\item $\textrm{Sup}([x,y],[h,k]):=[\sup((x+ k),(y+ h)),y+ k]$;
\item $[x,y]\leq [h,k]$ iff $\textrm{Inf}([x,y],[h,k])=[x,y]$;
\item $[0,0]$ is the identity element.
\end{itemize}

Notice that, for every perfect MV-algebra $\mathcal{A}$, $\Delta(Rad(\mathcal{A}))$ is the lattice-ordered Grothendieck group $G(Rad(\mathcal{A}))$ associa\-ted to $Rad(\mathcal{A})$, where the latter is regarded as a model of $\mathbb{M}$. 

\begin{theorem}\label{monoid-group}
The theories $\mathbb{M}$ and $\mathbb{L}$ are Morita-equivalent.
\end{theorem}
\begin{proof}
We need to prove that the categories of models of the two theories in any Grothendieck topos $\E$ are equivalent, naturally in $\E$.

Let $\E$ be a Grothendieck topos. We can define two functors:
\begin{itemize}
\item $T_{\E}:\mathbb{M}$-mod$(\E)\rightarrow\mathbb{L}$-mod$(\E)$. For any monoid $\mathcal{M}$ in $\mathbb{M}$-mod$(\E)$  we set $T_{\E}(\mathcal{M})$ to be the Grothendieck group $G(\mathcal{M})$. For a $\mathbb M$-model homomorphism $f:\mathcal{M}\rightarrow \mathcal{N}$, we set $T_{\E}(f)$ equal to the function $f^*:G(\mathcal{M})\rightarrow G(\mathcal{N})$ defined by using the internal language of the topos $\E$ as $f^*([x,y])=[f(x),f(y)]$. 

\item $R_{\E}:\mathbb{L}$-mod$(\E)\rightarrow \mathbb{M}$-mod$(\E)$. For every $\ell$-group $\mathcal{G}$ in $\mathbb{L}$-mod$(\E)$, its positive cone is trivially a model of $\mathbb{M}$. We set $R_{\E}(\mathcal{G})=(G^+,+,\leq,\inf,\sup,0)$, where $+,\leq,\inf,\sup$ are the restrictions to the positive cone of $\mathcal{G}$ of the ope\-rations and of the order of $\mathcal{G}$. Every $\ell$-homomorphism preserves the order; thus, we set $R_{\E}(g)=g|_{G^+}$.
\end{itemize}

These two functors are categorical inverses to each other. Indeed, we can define two natural isomorphisms $T_{\E}\circ R_{\E}(\mathcal{G})\simeq \mathcal{G}$ and $R_{\E}\circ T_{\E}(\mathcal{M})\simeq \mathcal{M}$ (for every $\ell$-group $\mathcal{G}$ and for every model $\mathcal{M}$ of $\mathbb{M}$ in an arbitrary Grothendieck topos $\E$). 

Let $\mathcal{M}$ be a model of $\mathbb{M}$ in $\E$. The arrow $\phi_{\mathcal{M}}:\mathcal{M}\rightarrow G(\mathcal{M})^+$ with $\phi_{\mathcal{M}}(x):=[x,0]$ is an isomorphism.

\begin{itemize}
\item[-] $\phi_{\mathcal{M}}$ is injective: given $x,y\in M$, $[x,0]=[y,0]$ iff $x=y$.
\item[-] $\phi_{\mathcal{M}}$ is surjective: given $[x,y]\in G(\mathcal{M})^+$, this means that
\begin{center}
$[0,0]\leq [x,y]\Leftrightarrow Inf([0,0],[x,y])=[0,0]\Leftrightarrow[\inf(x,y),y]=[0,0]\Leftrightarrow \inf(x,y)=y \Leftrightarrow y\leq x$
\end{center}
By axiom $M.14$, there exists $z\in M$ such that $x=z+y$. Thus, $[x,y]=[z,0]=\phi_{\mathcal{M}}(z)$.
\item[-] $\phi_{\mathcal{M}}$ preserves $+$: given $x,y\in M$, $\phi_{\mathcal{M}}(x)+\phi_{\mathcal{M}}(y)=[x,0]+[y,0]=[x+y,0]=\phi_{\mathcal{M}}(x+y)$.
\end{itemize}

In a similar way we can prove that $\phi_{\mathcal{M}}$ preserves the other $\ell$-group operations whence the order relation.

Let $\mathcal{G}$ be a model of $\mathbb{L}$ in $\E$. The arrow  $\chi_{\mathcal{G}}:G\rightarrow G(R_{\E}(\mathcal{G}))$ with $\chi_{\mathcal{G}}(g):=[g^+,g^-]$ is an isomorphism.

\begin{itemize}
\item[-] $\chi_{\mathcal{G}}$ is injective: given $g,h\in G$ such that $[g^+,g^-]=[h^+,h^-]$, we have
\begin{center}
$g^++h^-=g^-+h^+$ iff $g^+-g^-=h^+-h^-$ iff $g=h$
\end{center}
\item[-] $\chi_{\mathcal{G}}$ is surjective: given $[x,y]\in G(R_{\E}(\mathcal{G}))$, there exists $g=x-y$ in $G$.
\begin{center}
$[g^+,g^-]=[x,y]$ iff $g^++y=g^-+x$ iff $g^+-g^-=x-y$
\end{center}
Thus, $\chi_{\mathcal{G}}(g)=[x,y]$
\item[-] $\chi_{\mathcal{G}}$ preserves $+$: given $g,h\in G$, we have that $\chi_{\mathcal{G}}(g+h)=[(g+h)^+,(g+h)^-]$ and $\chi_{\mathcal{G}}(g)+\chi_{\mathcal{G}}(h)=[g^++h^+,g^-+h^-]$. These two elements are equal iff
\begin{center}
$(g+h)^++g^-+h^-=(g+h)^-+g^++h^+$ iff $(g+h)^+-(g+h)^-=g^+-g^-+h^+-h^-$ iff $g+h=g+h$
\end{center}
\item[-] $\chi_{\mathcal{G}}$ preserves $-$: given $g\in G$. We have that $\chi_{\mathcal{G}}(-g)=[(-g)^+,(-g)^-]$ and $-\chi_{\mathcal{G}}(g)=[g^-,g^+]$. These two elements are equal iff
\begin{center}
$(-g)^++g^+=(-g)^-+g^-$ iff $(-g)^+-(-g)^-=g^--g^+$ iff $-g=-g$
\end{center}
\end{itemize}
It is easy to check that $\chi_{\mathcal{G}}$ is a homomorphism.

Finally, the categorical equivalence just established is natural in $\E$; indeed, all the constructions that we have used are geometric.
\end{proof}

Recall that the `bridge technique' introduced in \cite{Caramello1} can be applied to any pair of Morita-equivalent theories. This method allows to transfer properties and constructions from one theory to the other by using the common classifying topos as a `bridge' on which various kinds of topos-theoretic invariants can be considered. 

Since the theories $\mathbb{M}$ and $\mathbb{L}$ are cartesian, they are both of presheaf type. In this case, an interesting invariant to consider is the notion of irreducible object of the classifying topos.

\begin{remark}\label{rem:irreducibles}
For any two Morita-equivalent theories of presheaf type $\mathbb T$ and ${\mathbb T}'$, the equivalence of classifying toposes $[{\mathscr{C}_{\mathbb{T}}^{\textrm{irr}}}^{\textrm{op}}, \Set]\simeq [{\mathscr{C}_{\mathbb{T}'}^{\textrm{irr}}}^{\textrm{op}}, \Set]$ restricts to the full subcategories ${\mathscr{C}_{\mathbb{T}}^{\textrm{irr}}}$ and ${\mathscr{C}_{\mathbb{T}'}^{\textrm{irr}}}$ of irreducible objects (cf. Remark \ref{rem:cartirr}), yielding an equivalence 
\[
{\mathscr{C}_{\mathbb{T}}^{\textrm{irr}}}\simeq {\mathscr{C}_{\mathbb{T}'}^{\textrm{irr}}}.
\]
\end{remark}

Applying this to our theories, we obtain a categorical equivalence 
\[
\mathscr{C}_{\mathbb{M}}^{\textrm{irr}}\simeq \mathscr{C}_{\mathbb{L}}^{\textrm{irr}},
\]
which we can explicitly describe as follows. Since both the theories $\mathbb M$ and $\mathbb L$ are cartesian, we have natural equivalences ${\mathscr{C}_{\mathbb{M}}^{\textrm{irr}}}\simeq {\mathscr{C}_{\mathbb{M}}^{\textrm{cart}}}$ and ${\mathscr{C}_{\mathbb{L}}^{\textrm{irr}}}\simeq {\mathscr{C}_{\mathbb{L}}^{\textrm{cart}}}$. In fact, the $\mathbb T$-irreducible formulas for a cartesian theory $\mathbb{T}$ are precisely the $\mathbb T$-cartesian ones (up to isomorphism in the syntactic category).

Recall that for any cartesian theory $\mathbb T$ and cartesian category $\mathscr{C}$, we have a categorical equivalence

\[
\mathbf{Cart}(\mathscr{C}_{\mathbb{T}}^{\textrm{cart}},\mathscr{C})\simeq \mathbb{T}\textrm{-mod}(\mathscr{C}),
\] 
where $\mathbf{Cart}(\mathscr{C},\mathscr{D})$ is the category of cartesian functors between cartesian categories $\mathscr{C}$ and $\mathscr{D}$. 
In the category $\mathscr{C}_{\mathbb{L}}^{\textrm{cart}}$ there is a canonical model of $\mathbb{L}$ given by the structure $\mathcal{G}_{\mathbb{L}}=(\{x.\top\},+,-,\leq,$ $\inf,\sup,0)$. It is immediate to see that we can restrict the operations $+$, $\inf$ and $\sup$ on $\mathcal{G}_{\mathbb{L}}$ to the subobject $\{x.x\geq 0\}$ of $\{x.\top\}$. The resulting structure $
(\{x.x\geq 0\},+,\leq,\inf,\sup,0)$ is a model $U$ of $\mathbb{M}$ in $\mathscr{C}_{\mathbb{L}}$.

In the converse direction, consider the syntactic category $\mathscr{C}_{\mathbb{M}}^{\textrm{cart}}$ of $\mathbb{M}$ and the canonical model $\mathcal{M}_{\mathbb{M}}=(\{y.\top\},+,\leq,\inf,\sup,0)$ of $\mathbb{M}$ in it. The $\ell$-group associated to a model of $\mathbb{M}$ in an arbitrary Grothendieck topos $\E$ via the Morita-equivalence described above is the Grothendieck group of $\mathcal{M}$, whose elements, we recall, are equivalence classes $[x,y]$ of pairs of elements of $\mathcal{M}$. Given a pair of elements $(x, y)$ of $\mathcal{M}$, consider $\inf(x,y)$; since $\inf(x,y)\leq x$ and $\inf(x,y)\leq y$, by axiom M.14 there exist exactly two elements $u,v$ such that $x=\inf(x,y)+u$ and $y=\inf(x,y)+v$. These elements clearly satisfy $[x,y]=[u, v]$; moreover, $\inf(u,v)=0$. Indeed, 
\[
\inf(u,v)+\inf(x,y)=\inf(u+\inf(x,y),v+\inf(x,y))=\inf(x,y),
\]
whence $\inf(u,v)=0$ by axiom M.12.

Note that the pair $(u, v)$ does not depend on the equivalence class of $(x, y)$. Indeed, if $[x,y ]=[u',v']$ and $\inf(u',v')=0$ then $x+v'=y+u'$ and the following identities hold:
\begin{center}
$\inf(x,y)+u'=\inf(x+u',y+u')=\inf(x+u',x+v')=x+\inf(u',v')=x$,
\end{center} 
which implies that $u=u'$. In an analogous way we can prove that $v=v'$.

This allows us to choose the pair $(u, v)$ defined above as a canonical representative for the equivalence class $[x,y]$ in $G(\mathcal{M})$.

We are thus led to consider the following structure in $\mathscr{C}_{\mathbb{M}}^{\textrm{cart}}$:
\begin{itemize}
\item underlying object: $\{(u,v).\inf(u,v)=0\}$. 
\item sum: $[z+v+b = t+u+a\wedge \inf(z,t)=0]:\{(u,v).\inf(u,v)=0\}\times\{(a,b).a\wedge b=0\}\rightarrow \{(z,t).z\wedge t=0\}$. 
\item opposite: $[a=v\wedge b=u]:\{(u,v).\inf(u,v)=0\}\rightarrow \{(a,b).\inf(a,b)=0\}$
\item zero: $[u=0,v=0]:\{[].\top\}\rightarrow \{(u,v).\inf(u,v)=0\}$
\item $\textrm{Inf}:[z+u+b=t+\inf(u+b,v+a)\wedge \inf(z,t)=0]:\{(u,v).\inf(u,v)=0\}\times\{(a,b).\inf(a,b)=0\}\rightarrow \{(z,t).\inf(z,t)=0\}$
\item $\textrm{Sup}:[z+u+b=t+\sup(u+b,v+a)\wedge \inf(z,t)=0]:\{(u,v).\inf(u,v)=0 \}\times\{(a,b).\inf(a,b)=0\}\rightarrow \{(z,t).\inf(z,t)=0\}$ 
\end{itemize}
 
It can be easily seen that this structure is a model $V$ of $\mathbb{L}$ inside $\mathscr{C}_{\mathbb{M}}^{\textrm{cart}}$.

Let $F_{U}:\mathscr{C}_{\mathbb{M}}^{\textrm{cart}}\rightarrow \mathscr{C}_{\mathbb{L}}^{\textrm{cart}}$ and $F_{V}:\mathscr{C}_{\mathbb{L}}^{\textrm{irr}}\rightarrow \mathscr{C}_{\mathbb{M}}^{\textrm{cart}}$ be the cartesian functors respectively induced by the models $U$ and $V$. For every object $\{\vec{x}.\phi\}$ of $\mathscr{C}_{\mathbb{M}}$, $F_{U}(\{\vec{x}.\phi\}):=\{\vec{x}.\phi\wedge \vec{x}\geq 0\}$. The functor $F_{V}$ admits the following inductive definition:
\begin{itemize}
\item $F_{V}(\{\vec{y}.\top\}):= \{(\vec{u},\vec{v}).\inf(\vec{u},\vec{v})=0\}$
\item $F_{V}((\{\vec{y},\vec{x}).\vec{y}+\vec{x}\}):= \{(\vec{u},\vec{v},\vec{a},\vec{b}).(\vec{u},\vec{v})+(\vec{a},\vec{b})\wedge \inf(\vec{u},\vec{v})=0\wedge \inf(\vec{a},\vec{b})=0\}$
\item $F_{V}(\{\vec{y}.-\vec{y}\}):=\{(\vec{v},\vec{u}).\inf(\vec{v},\vec{u})=0\}$
\item $F_{V}(\{(\vec{y},\vec{x}).\textrm{Inf}(\vec{y},\vec{x})\}):=\{(\vec{u},\vec{v},\vec{a},\vec{b}).(\inf(\vec{u},\vec{a}),\inf(\vec{v},\vec{b}))\wedge \inf(\vec{u},\vec{v})=0\wedge \inf(\vec{a},\vec{b})=0\}$
\item $F_{V}(\{(\vec{y},\vec{x}).\textrm{Sup}(\vec{y},\vec{x})\}):=\{(\vec{u},\vec{v},\vec{a},\vec{b}).(\sup(\vec{u},\vec{a}),\sup(\vec{v},\vec{b}))\wedge \inf(\vec{u},\vec{v})=0\wedge \inf(\vec{a},\vec{b})=0\}$
\end{itemize}

Let us now proceed to show that the functors $F_{U}$ and $F_{V}$ are categorical inverses to each other.

\begin{itemize}
\item[\textit{Claim 1}] The formulas $\{x.\top\}$ and $\{(u,v).\inf(u,v)=0\}$ are isomorphic in $\mathscr{C}_{\mathbb{M}}^{\textrm{irr}}$. 

To see this, consider the following arrow in $\mathscr{C}_{\mathbb{M}}$:
\begin{center}
$[u=x,v=0]: \{x.\top\}\rightarrow \{(u,v).\inf(u,v)=0\wedge (u,v)\geq (0,0)\}$. 
\end{center}
 All the ``elements'' of the object $\{(u,v).\inf(u,v)=0\wedge (u,v)\geq (0,0)\}$ are of the form $(u,0)$; indeed, $(u,v)\geq (0,0)$ iff $\inf((0,0),(u,v))=(0,0)$, and this means that $v=\inf(u,v)=0$. It follows that the arrow just defined is an isomorphism.
\item[\textit{Claim 2}] The formulas $\{x.\top\}$ and $\{(u,v).(u\wedge v=0)\wedge u\geq 0\wedge v\geq 0\}$ are isomorphic in $\mathscr{C}_{\mathbb{L}}^{\textrm{irr}}$. 

To see this, consider the following arrow in $\mathscr{C}_{\mathbb{L}}$:
\begin{center}
$[u=x^+,v=x^-]:\{x.\top\}\rightarrow \{(u,v).(u\wedge v=0)\wedge u\geq 0\wedge v\geq 0\}$ 
\end{center}
It is well-defined because $\inf(x^+,x^-)=0$. In addition, taken $u,v$ such that $\inf(u,v)=0$, we can consider $x=u-v$. We have that:
\begin{center}
$u=(u-v)^+$ and $v=(u-v)^-$ iff

\

$v+u=v+(u-v)^+$ and $v+u=u+(u-v)^-$ iff

\

$v+u=v+(\sup((u-v), 0))$ and $v+u=u+(\sup((v-u),0))$ iff \cite[Proposition 1.2.2]{BK}

\

$v+u=\sup((v+u-v),(v+ 0))$ and $v+u=\sup((u+v-u),(u+0))$ iff

\

$v+u=\sup(u, v)$ and $v+u=\sup(v, u)$.
\end{center}
But the sequent $\top\vdash_{x,y}x+y=\sup(x, y)+\inf(x, y)$ is provable in $\mathbb{L}$ (cf. \cite[Proposition 1.2.6]{BK}). Hence, the arrow $[u=x^{+}, v=x^{-}]$ is an isomorphism, as required.
\end{itemize}

From Claims 1 and 2 it follows at once that the functors $F_{U}$ and $F_{V}$ are categorical inverses to each other.
 
Summarizing, we have the following result.

\begin{proposition}\label{monodi-groups irriducibili}
The functors $F_{U}:\mathscr{C}_{\mathbb{M}}^{cart} \to \mathscr{C}_{\mathbb{L}}^{cart}$ and $F_{V}:\mathscr{C}_{\mathbb{L}}^{cart} \to \mathscr{C}_{\mathbb{M}}^{cart}$ defined above form a categorical equivalence.
\end{proposition}

\begin{remark}
The equivalence of Proposition \ref{monodi-groups irriducibili} induces, in light of \cite[Theorem 4.3]{Caramello2}, an equivalence of categories $\textrm{f.p.}\mathbb{M}$-mod$(\Set)\simeq \textrm{f.p.}\mathbb{L}$-mod$(\Set)$, which is precisely the restriction of the Morita-equivalence of Theorem \ref{monoid-group}.
\end{remark}

\begin{lemma}\label{Grothendieck group}
The $\ell$-Grothendieck group $G(\mathcal{M})$ associated to a   model of $\mathcal{M}$ of $\mathbb M$ in $\Set$ satisfies the following universal property:
\begin{itemize}
\item[(*)]there exists an $\ell$-monoid homomorphism $i:\mathcal{M}\rightarrow G(\mathcal{M})$ of models such that for every $\ell$-monoid homomorphism $f:\mathcal{M}\rightarrow \mathcal{H}$, where $H$ is an $\ell$-group, there exists a unique $\ell$-group homomorphism $g:G(\mathcal{M})\rightarrow \mathcal{H}$ such that $f=g\circ i$.
\end{itemize}
\end{lemma}
\begin{proof}
Set $i:\mathcal{M}\rightarrow G(M)$ equal to the function $i(x)=[x,0]$. This is a $\ell$-monoid homomorphism since it is the composite of the $\ell$-monoid isomorphism $\phi_{\mathcal{M}}:\mathcal{M}\rightarrow G(\mathcal{M})^+$ considered in the proof of Theorem \ref{monoid-group} with the  inclusion $G(\mathcal{M})^+ \hookrightarrow G(\mathcal{M})$, which is an $\ell$-monoid homomorphism since the $\ell$-monoid structure on $G(\mathcal{M})^+$ is induced by restriction of that on $G(\mathcal{M})$.
 
Given an $\ell$-monoid homomorphism $f:\mathcal{M}\rightarrow \mathcal{H}$, where $H$ is an $\ell$-group, in order to have $g\circ i=f$, we are forced to define $g$ as $g:[x,y]\in G(\mathcal{M})\rightarrow f(x)-f(y) \in \mathcal{H}$. This is clearly a well-defined group homomorphism. It remains to show that it also preserves the lattice structure.
\begin{itemize}
\item $g$ preserves Inf: $g(\textrm{Inf}([x,y],[h,k]))=g([\inf(x+k,y+h),y+k])=\inf(f(x+k),f(y+h))-f(y+k)=\inf(f(x+k)-f(y+k),f(y+h)-f(y+k))=\inf(f(x)-f(y),f(h)-f(k))=\textrm{Inf}(g([x,y]),g([h,k]))$
\item $g$ preserves Sup: the proof is analogous to that for Inf.
\end{itemize}
\end{proof}

\begin{proposition}\label{commutation}
The functors $F_{U}$ and $F_{V}$ correspond to the functors $T_{\Set}$ and $R_{\Set}$ under the canonical equivalences $\mathscr{C}_{\mathbb{M}}^{irr} \simeq  f.p.\mathbb{M}$-mod$(\Set)^{op}$ and $\mathscr{C}_{\mathbb{L}}^{irr} \simeq  f.p.\mathbb{L}$-mod$(\Set)^{op}$:

\begin{center}
\begin{tikzpicture}
\node (C) at (0,0) {$\mathscr{C}_{\mathbb{M}}^{irr}$};
\node (A) at (3,0) {$f.p.\mathbb{M}$-mod$(\Set)^{op}$};
\node (3) at (1,0) {$\simeq$};
\node (4) at (1,2) {$\simeq$};
\node (D) at (0,2) {$\mathscr{C}_{\mathbb{L}}^{irr}$ };
\node (B) at (3,2) {$f.p.\mathbb{L}$-mod$(\Set)^{op}$};
\draw[->] (A) to node [midway,right] {$T_{\Set}$} (B);
\draw[->] (C) to node [midway, left] {$F_{U}$} (D);
\node (C') at (7,0) {$\mathscr{C}_{\mathbb{M}}^{irr}$};
\node (A') at (10,0) {$f.p.\mathbb{M}$-mod$(\Set)^{op}$};
\node (3') at (8,0) {$\simeq$};
\node (4') at (8,2) {$\simeq$};
\node (D') at (7,2) {$\mathscr{C}_{\mathbb{L}}^{irr}$ };
\node (B') at (10,2) {$f.p.\mathbb{L}$-mod$(\Set)^{op}$};
\draw[->] (B') to node [midway,right] {$R_{\Set}$} (A');
\draw[->] (D') to node [midway,left] {$F_{V}$} (C');
\end{tikzpicture}
\end{center} 
\end{proposition}
\begin{proof}
Since $T_{\Set}$, $R_{\Set}$ and $F_{U}$, $F_{V}$ are respectively categorical inverses to each other, it is sufficient to prove that the diagram on the left-hand side commutes (up to natural isomorphism). 

From Lemma \ref{Grothendieck group} it follows that if $\mathcal{N}$ is a model of $\mathbb{M}$ presented by a formula $\{\vec{x}.\phi\}$ in $\mathbb{M}$, then the model $T_{\Set}(\mathcal{N})$ of $\mathbb{L}$ is presented by the formula $\{\vec{x}.\phi\wedge \vec{x}\geq 0\}$, that is by the image of the object $\{\vec{x}.\phi\}$ under the functor $F_{U}$. This immediately implies our thesis.
\end{proof}

\begin{obs}
From Proposition \ref{commutation} it follows in particular that $\mathbb{N}\times \mathbb{N}$, as a model of $\mathbb{M}$, is presented by the formula $\{(u,v).\inf(u,v)=0\}$.
Indeed, $\mathbb{Z}\times \mathbb{Z}$ is presented as an $\mathbb L$-model by the formula $\{x. \top\}$, it being the free $\ell$-group on one generator (namely, $(1,-1)$), whence $\mathbb{N}\times \mathbb{N}=R_{\Set}(\mathbb{Z}\times \mathbb{Z})$ is presented by the formula $F_{V}(\{x. \top\})=\{(u,v).\inf(u,v)=0\}$; a pair of generators is given by $((1,0), (0,1))$.
\end{obs}

\section{Interpretability}\label{sct:intepretability}

In section \ref{equivalence in topos} we proved that the theories $\mathbb{P}$ and $\mathbb{L}$ are Morita-equivalent by establishing a categorical equivalence between the categories of models of these two theo\-ries in any Grothendieck topos $\mathscr{E}$, naturally in $\mathscr{E}$. This result would be trivial if the theories were bi-interpretable. In this section we show that this is not the case, i.e., the theories $\mathbb{P}$ and $\mathbb{L}$ are not bi-intepretable in a global sense.
Nevertheless, if we consider particular categories of formulas we have three different forms of bi-interpretability.

\begin{definition}
A geometric theory $\mathbb{T}$ is \textit{interpretable} in a geometric theory $\mathbb{S}$ if there is a geometric functor $I:\mathscr{C}_{\mathbb{T}}\rightarrow \mathscr{C}_{\mathbb{S}}$ between the respective syntactic categories. If $I$ is also a categorical equivalence we say that $\mathbb{T}$ and $\mathbb{S}$ are \textit{bi-interpretable.}
\end{definition}

\begin{obs}
Two bi-interpretable theories are trivially Morita-equivalent. Indeed, this immediately follows from the the standard representation of the classifying topos of a geometric theory $\mathbb{T}$ as the category of sheaves on its syntactic site $(\mathscr{C}_{\mathbb{T}},J_{\mathbb{T}})$.
\end{obs}

Recall that for any geometric category $\mathscr{C}$ there is a categorical equivalence
\begin{center}
$\textbf{Geom}(\mathscr{C}_{\mathbb{T}},\mathscr{C})\simeq \mathbb{T}$-mod$(\mathscr{C})$,
\end{center}
where $\mathbf{Geom}(\mathscr{C}_{\mathbb{T}},\mathscr{C})$ is the category of geometric functors from $\mathscr{C}_{\mathbb{T}}$ to $\mathscr{C}$, which sends any $\mathbb{T}$-model $\mathcal{M}$ in $\mathscr{C}$ to the geometric functor $F_\mathcal{M}:\mathscr{C}_{\mathbb{T}}\rightarrow \mathscr{C}$ defined by setting 
\begin{center}
$F_\mathcal{M}(\{\vec{x}.\phi\})=[[\vec{x}. \phi]]_\mathcal{M}$.
\end{center}

\begin{theorem}\label{interpretability}
The theory $\mathbb{L}$ is interpretable in the theory $\mathbb{P}$ but not bi-interpretable.
\end{theorem}
\begin{proof}
By Lemma \ref{rad_ideal}, the object $\{x. x\leq\neg x\}$ of $\mathscr{C}_{\mathbb{P}}$ has the structure of a cancellative lattice-ordered abelian monoid with bottom element, and therefore defines a model $M$ of the theory $\mathbb M$ inside the category $\mathscr{C}_{\mathbb{P}}$. This induces a geometric functor $\textrm{Rad}:\mathscr{C}_{\mathbb{M}}\to \mathscr
{C}_{\mathbb{P}}$, that is an interpretation of the theory $\mathbb M$ in the theory $\mathbb P$. Composing this functor with the functor $F_{V}:\mathscr{C}_{\mathbb{L}}^{\textrm{cart}} \to \mathscr{C}_{\mathbb{M}}^{\textrm{cart}}$ of Proposition \ref{monodi-groups irriducibili}, we obtain a cartesian functor $\mathcal{C}_{\mathbb{L}}\to \mathscr{C}_{\mathbb{P}}$, which corresponds to a model of $\mathbb{L}$ in $\mathscr{C}_{\mathbb{P}}$ whose underlying object is the formula-in-context $\{(u,v).\inf(u,v)=0\wedge u\leq \neg u\wedge v\leq \neg v\}$, and hence to an interpretation functor $\mathscr{C}_{\mathbb{L}} \to \mathscr
{C}_{\mathbb{P}}$.
 

Suppose now that $\mathbb{P}$ and $\mathbb{L}$ were bi-interpretable. Then there would be in particular an interpretation functor
\begin{center}
$J:\mathscr{C}_{\mathbb{P}}\rightarrow \mathscr{C}_{\mathbb{L}}$,
\end{center}
inducing a functor
\begin{center}
$s_J:\mathbb{L}\textrm{-mod}(\mathbf{Set})\rightarrow \mathbb{P}\textrm{-mod}(\mathbf{Set})$.
\end{center}
Notice that if $\mathcal{M}$ is a $\mathbb{L}$-model in $\mathbf{Set}$ and $\mathcal{N}=s_J(\mathcal{M})$, we would have that $F_\mathcal{N}=F_\mathcal{M}\circ J$.

Now, let $\mathcal{M}$ be the trivial model of $\mathbb{L}$ in $\mathbf{Set}$, that is the model whose underlying set is $\{0\}$, $\mathcal{N}=s_J(\mathcal{M})$ and $J(\{\vec{x}.\top\})=\{\vec{x}.\psi\}$. We would have
\begin{center}
$F_\mathcal{N}(\{\vec{x}.\top\})\cong F_\mathcal{M}\{\vec{x}. \psi\}$

\

$[[\vec{x}. \top]]_\mathcal{N}\cong [[\vec{x}. \psi]]_\mathcal{M}$

\

$N\simeq [[\vec{x}.\psi]]_\mathcal{M}\subseteq M^n\cong M$
\end{center}

Hence the domain of $\mathcal{N}$ would be contained in $\{0\}$. But we know from \cite[Proposition 5(1)]{P-L} that the only finite perfect MV-algebra is the one whose underlying set is $\{0,1\}$. This is a contradiction. 
\end{proof}

Even though, as we have just seen, the theories of perfect MV-algebras and of $\ell$-groups are not bi-interpretable in the classical sense, the Morita-equivalence between them, combined with the fact that both theories are of presheaf type, guarantees that there is a bi-interpretation between them holding at the level of irreducible formulas (cf. Remark \ref{rem:irreducibles}). More specifically, the following result holds.

\begin{theorem}\label{irreducible_formulas}
The categories of irreducible formulas of the theories $\mathbb P$ of perfect MV-algebras and $\mathbb L$ of $\ell$-groups are equivalent. 

In particular, the functor $\mathscr{C}_{\mathbb{L}}^{\textrm{irr}}=\mathscr{C}_{\mathbb{L}}^{\textrm{cart}}\to \mathscr{C}_{\mathbb{P}}$ given by the composite of the functor $F_{V}:\mathscr{C}_{\mathbb{L}}^{\textrm{cart}} \to \mathscr{C}_{\mathbb{M}}^{\textrm{cart}}$ of Proposition \ref{monodi-groups irriducibili} with the restriction to $\mathscr{C}_{\mathbb{M}}^{\textrm{cart}}$ of the functor $\textrm{Rad}:\mathscr{C}_{\mathbb{M}}\to \mathscr
{C}_{\mathbb{P}}$ yields a categorical equivalence
\begin{center}
$\mathscr{C}_{\mathbb{L}}^{\textrm{irr}}=\mathscr{C}_{\mathbb{L}}^{\textrm{cart}}\simeq \mathscr{C}_{\mathbb{P}}^{\textrm{irr}}$.
\end{center}
\end{theorem}\qed

\begin{obs}
It follows from the theorem that the $\mathbb{P}$-irreducible formulas are precisely, up to isomorphism in the syntactic category, the ones that come from the $\mathbb{M}$-cartesian formulas via the functor $\textrm{Rad}:\mathscr{C}_{\mathbb{M}}\to \mathscr
{C}_{\mathbb{P}}$.
\end{obs}

Changing the invariant property to consider on the classifying topos of the theo\-ries $\mathbb P$ and $\mathbb L$, we uncover another level of bi-interpretability. Specifically, the invariant notion of subterminal object of the classifying topos yields a categorical equivalence between the full subcategories of $\mathscr{C}_{\mathbb{P}}$ and $\mathscr{C}_{\mathbb{L}}$ on the geometric sentences. Recall that a geometric sentence is a geometric formula without any free variables. For any geometric theory $\mathbb T$, the subterminal objects of its classifying topos $\Sh(\mathscr{C}_{\mathbb T}, J_{\mathbb T})$ can be exactly identified with the geometric sentences over the signature of $\mathbb T$, considered up to the following equivalence relation: $\phi \sim_{\mathbb T} \psi$ if and only if $(\phi\vdash_{[]}\psi)$ and $(\psi\vdash_{[]}\phi)$ are provable in $\mathbb T$.

Since the theories $\mathbb{P}$ and $\mathbb{L}$ are Morita-equivalent, we thus obtain the following 
\begin{theorem}
There is a bijective correspondence between the classes of geometric sentences of $\mathbb{P}$ and of $\mathbb{L}$.
\end{theorem}

We can explicitly describe this correspondence by using the bi-interpretation between irreducible formulas provided by Theorem \ref{irreducible_formulas} and the concept of ideal on a category.

\begin{lemma}\label{ideal-sentences}
Let $\mathbb{T}$ be a theory of presheaf type and
\begin{center}
$A=\{\mathbb{T}$-classes of geometric sentences$\}$,

\
$B=\{$ideals of $\mathscr{C}_{\mathbb{T}}^{\textrm{irr}}\}$.
\end{center}
There is a canonical bijection between $A$ and $B$.
\end{lemma}
\begin{proof}
For any object $\{\vec{x}.\psi\}\in \mathscr{C}_{\mathbb{T}}^{irr}$ there is a unique arrow $!_{\psi}:\{\vec{x}.\psi\}\rightarrow \{[].\top\}$ in $\mathscr{C}_{\mathbb{T}}$, where $\{[].\top\}$ is the terminal object of $\mathscr{C}_{\mathbb{T}}$. Given $\{[].\phi\}\in A$, we set $I_{\phi}:=\{\{\vec{x}.\psi\}\in\mathscr{C}_{\mathbb{T}}^{\textrm{irr}}\mid$ \ $!_{\psi}$ factors through $\{[].\phi\}\mono \{[]. \top\}\}$. This is an ideal of $\mathscr{C}_{\mathbb{T}}^{\textrm{irr}}$. Indeed, if $\{\vec{x}.\psi\}\in I_{\phi}$ and $f:\{\vec{y}.\chi\}\rightarrow \{\vec{x}.\psi\}$ is an arrow in $\mathscr{C}_{\mathbb{T}}^{\textrm{irr}}$, the commutativity of the following diagram guarantees that the arrow $!_{\chi}$ factors through $\{[].\phi\}$, i.e. that $\{\vec{y}.\chi\}\in I_{\psi}$:
\begin{center}
\begin{tikzpicture}
\node(1) at (0,0) {$\{\vec{x}.\psi\}$};
\node(2) at (5,0) {$\{[].\phi\}$};
\node(3) at (5,2) {$\{[].\top\}$};
\node(4) at (0,2) {$\{\vec{y}.\chi\}$};
\draw[->] (1) to node[below, midway]{} (2);
\draw[->] (1) to node[above, midway]{$!_{\psi}$} (3);
\draw[->] (2) to node[below] {} (3);
\draw[->] (4) to node[left,midway]{$f$} (1);
\draw[->] (4) to node[above,midway]{$!_{\chi}$}(3);
\end{tikzpicture}
\end{center}

The assignment $\phi \to I_{\phi}$ defines a map $f:A\rightarrow B$. 

In the converse direction, suppose that $I\in B$. For any $\{\vec{x}.\psi\}\in I$, the arrow $!_{\psi}$ factors through the subobject $\{[].\exists\vec{x}\psi(\vec{x})\}\mono \{[]. \top\}$. We can consider the union $\{[]. \phi_{I}\}\mono \{[]. \top\}$ of these subobjects for all the objects in $I$. In other words, we set $\phi_{I}$ equal to the ($\mathbb T$-class of) the formula $\mathbin{\mathop{\textrm{\huge $\vee$}}\limits_{\{\vec{x}.\psi\}\in I}}\exists\vec{x}\psi(\vec{x})$. 

The assignment $I \to \phi_{I}$ defines a map $g:B\rightarrow A$.

\begin{itemize}
\item[\textit{Claim} 1.]  $g\circ f=id_{A}$. Given $\{[].\phi\}\in A$, we need to prove that $\phi\sim_{\mathbb T}\phi_{I_{\phi}}$. Recall that, since $\mathbb{T}$ is a theory of presheaf type, there exists a $J_{\mathbb T}$-covering $\{\{\vec{x}_j.\psi_j\}\mid j\in J\}$ of $\{[].\phi]\}$ by $\mathbb{T}$-irreducible formulas, i.e. $(\phi\vdash \mathbin{\mathop{\textrm{\huge $\vee$}}\limits_{j\in J}}\exists\vec{x}_j\psi_j(\vec{x}_j))$ is provable in $\mathbb T$. Thus, for each $j\in J$, the arrow $!_{\psi_j}$ factors through $\{[].\phi\}\mono \{[]. \top\}$ whence $\{\vec{x}_j.\phi_j\}\in I_{\phi}$. Therefore $(\mathbin{\mathop{\textrm{\huge $\vee$}}\limits_{j\in J}}\exists\vec{x}_j\psi_j(\vec{x}_j)\vdash \phi_{I_{\phi}})$ is provable in $\mathbb T$ whence $(\phi\vdash \phi_{I_{\phi}})$ is provable in $\mathbb T$. On the other hand, it is clear that $(\phi_{I_{\phi}}\vdash \phi)$ is provable in $\mathbb T$. Therefore $\phi\sim_{\mathbb T}\phi_{I_{\phi}}$.

\item[\textit{Claim} 2.] $f\circ g=id_{B}$. Given $I\in B$, we have to prove that $I=I_{\phi_I}$. The elements of $I_{\phi_I}$ are the $\mathbb{T}$-irreducible formulas $\{\vec{x}.\psi\}$ such that the arrow $!_{\psi}$ factors through $\{[].\phi_I\}\mono \{[]. \top\}$. By definition of $\phi_I$, this property is satisfied by the elements of $I$ whence $I\subseteq I_{\phi_I}$. The converse inclusion can be proved as follows.  Given $\{\vec{x}.\psi\}\in I_{\phi_I}$, the unique arrow $!_{\psi}$ factors through $\{[].\phi_I\}$. Since $\{[]. \phi_{I}\}$ is $J_{\mathbb T}$-covered by $\mathbb T$-irreducible formulae in $I$, it follows that $\{\vec{x}. \psi\}$ is also $J_{\mathbb T}$-covered by $\mathbb T$-irreducible formulae in $I$. Since the formula $\{\vec{x}.\psi\}$ is $\mathbb{T}$-irreducible, this covering is trivial whence $\{\vec{x}.\psi\}\in I$. 
\end{itemize}
\end{proof} 

\begin{obs}
Applying Lemma \ref{ideal-sentences} to the theory $\mathbb P$ of perfect MV-algebras and to the theory $\mathbb L$ of  $\ell$-groups we obtain two bijections:
\begin{itemize}
\item[1] $\{\mathbb{P}$-classes of sentences$\}\simeq \{$ideals of $\mathscr{C}_{\mathbb{P}}^{\textrm{irr}}\}$
\item[2] $\{\mathbb{L}$-classes of sentences$\}\simeq \{$ideals of $\mathscr{C}_{\mathbb{L}}^{\textrm{irr}}\}$
\end{itemize}

From these bijections and Theorem \ref{irreducible_formulas} we obtain a bijection between the $\mathbb{P}$-classes of geometric sentences and the $\mathbb{L}$-classes of geometric sentences.
\end{obs}

Since the theories $\mathbb P$ of perfect MV-algebras and $\mathbb L$ of $\ell$-groups are both coherent, we have a third level of bi-interpretability between them.

Let $\mathbb{T}$ be a coherent theory; starting from its coherent syntactic category $\mathscr{C}^{\textrm{coh}}_{\mathbb{T}}$, we can construct the category $\mathscr{C}_{\mathbb{T}}^{\textrm{eq}}$ of \textit{imaginaries} of $\mathbb T$ (also called the \textit{effective positivization of} $\mathscr{C}^{\textrm{coh}}_{\mathbb{T}}$) by adding formal finite coproducts and coequalizers of equivalence relations in $\mathscr{C}_{\mathbb{T}}^{\textrm{coh}}$. 

\begin{theorem}\label{imaginaries} \cite[Theorem D3.3.7]{SE}
Let $\mathbb{T}$ be a coherent theory. Then the category $\mathscr{C}_{\mathbb{T}}^{\textrm{eq}}$ is equivalent to the full subcategory of its classifying topos of the coherent objects.
\end{theorem}

From Theorem \ref{imaginaries} it follows that, if two coherent theories are Morita-equivalent, then the respective categories of imaginaries are equivalent. Notice that the topos-theoretic invariant used in this application of the `bridge' technique is the notion of coherent object. Specializing this to our Morita-equivalence between $\mathbb P$ and $\mathbb L$ yields the following result.

\begin{theorem}
The effective positivizations of the syntactic categories of the theories $\mathbb{P}$ and $\mathbb{L}$ are equivalent:
\begin{center}
$\mathscr{C}_{\mathbb{P}}^{\textrm{eq}}\simeq \mathscr{C}_{\mathbb{L}}^{\textrm{eq}}$.
\end{center}
\end{theorem}

\begin{obs}
It is natural to wonder whether we can give an explicit description of this equivalence. Consider the functor $F:\mathscr{C}_{\mathbb{L}}\rightarrow \mathscr{C}_{\mathbb{P}}$ given by the composite of the functor $\textrm{Rad}:\mathscr{C}_{\mathbb{M}}\rightarrow \mathscr{C}_{\mathbb{P}}$ with the functor $H:\mathscr{C}_{\mathbb L} \to \mathscr{C}_{\mathbb M}$ corresponding to the model $V$ of $\mathbb L$ in $\mathscr{C}_{\mathbb M}$ introduced in section \ref{intermediary}. The formal extension $F^{\textrm{eq}}:\mathscr{C}_{\mathbb{L}}^{\textrm{eq}}\rightarrow \mathscr{C}_{\mathbb{P}}^{\textrm{eq}}$ of $F$ is part of a categorical equivalence whose other half is the functor $\mathscr{C}_{\mathbb{P}}^{\textrm{eq}} \to \mathscr{C}_{\mathbb{L}}^{\textrm{eq}}$ induced by the model $Z$ of $\mathbb{P}$ in $\mathscr{C}_{\mathbb{L}}^{\textrm{eq}}$ defined as follows. Recall that, for any $\ell$-group $\mathcal{G}$ in $\Set$, the corresponding perfect MV-algebra is given by $\Gamma(\mathbb{Z}\times_{lex}\mathcal{G})$. Now, this set is isomorphic to the coproduct $G^+\sqcup G^-$, where $G^+$ and $G^-$ are respectively the positive and the negative cone of the $\ell$-group $\mathcal{G}$. The model $Z$ has as underlying object in $\mathscr{C}_{\mathbb{L}}^{\textrm{eq}}$ the coproduct $\{x.x\leq 0\}\sqcup\{x.x\geq 0\}$, whereas the operations and the order relation are defined as follows:

\begin{itemize}
\item $\oplus:(\{x.x\leq 0\}\sqcup\{x'.x'\geq 0\}) \times (\{y.y\leq 0\}\sqcup\{y'.y'\geq 0\}) \cong $

$\{x, y.x\leq 0 \wedge y\leq 0\} \sqcup \{u, v.u\leq 0 \wedge v\geq 0\} \sqcup \{w, p.w\geq 0 \wedge p\leq 0\} \sqcup  $

$\{q, r.q\geq 0 \wedge r\geq 0\} \to \{\alpha.\alpha\leq 0\} \sqcup \{\beta.\beta\geq 0\}$

is given by $[x, y. \alpha=0] \sqcup [u,v.\alpha=\inf(u+v,0)] \sqcup $

$[w,p.\alpha=\inf(w+p,0)] \sqcup [q,r.\beta=q+r]$;

\item $\neg:\{x.x\leq 0\}\sqcup\{x'.x'\geq 0\} \to \{y.y\leq 0\}\sqcup\{y'.y'\geq 0\}$

is given by $[y=-x'] \sqcup [y'=-x]$;

\item $0:\{[]. \top\} \to \{\alpha.\alpha\leq 0\} \sqcup \{\beta.\beta\geq 0\}$ is given by $[\beta=0]$. 
\end{itemize}
\end{obs}

\section{Finitely presented perfect MV-algebras}\label{finitely presented}

We have proved that the categories of irreducible formulas of the theory of perfect MV-algebras and of the theory of $\ell$-groups are equivalent (cf. Theorem \ref{irreducible_formulas}). By Theorem \ref{irr-f.p.}, the semantical counterpart of this equivalence is the categorical equivalence between the categories of finitely presented models of the two theories in $\Set$. In symbols $\textrm{f.p.}\mathbb{P}$-mod$(\Set)\simeq \textrm{f.p.}\mathbb{L}$-mod$(\Set)$. The finitely presentable perfect MV-algebras are thus the images of the finitely presented $\ell$-groups under Di Nola-Lettieri's equivalence.  

Since the class of finitely presentable perfect MV-algebras is a subclass of the variety of MV-algebras, it is natural to wonder whether it is also contained in the class of finitely presented MV-algebras. The answer is negative since Chang's algebra is finitely presented as a model of $\mathbb P$ by the formula $\{x. x\leq \neg x\}$ but it is not finitely presented as an MV-algebra. Indeed, every finitely presented MV-algebra is \textit{semisimple} (see Theorem 3.6.9 \cite{CDM}), i.e., its radical is $\{0\}$, hence the only finitely presented MV-algebra that is also perfect is $\{0,1\}$. 

In general, for any MV-algebra $\mathcal{A}$, the MV-algebra homomorphisms $C\to {\mathcal A}$ are in natural bijection with the set $Inf(\mathcal{A})$ of infinitesimal elements of $\mathcal A$ plus the zero. Thus, whilst the property of  an element of an MV-algebra to be an infinitesimal or $0$ is preserved by filtered colimits of arbitrary algebras in $V(C)$ (since it is defined by the geometric formula $\{x. x\leq \neg x\}$ for algebras in $V(C)$, as stated by Proposition \ref{radical}), this is no longer true for arbitrary MV-algebras, by the definability theorem for theories of presheaf type (Theorem \ref{definability}) applied to the theory $\mathbb{MV}$. Indeed, this property is clearly preserved under arbitrary homomorphisms of MV-algebras and, if it were also preserved by filtered colimits, it would be definable by a geometric formula (cf. Remark \ref{rmk:definability}(b)) and hence $C$ would be finitely presented as an MV-algebra, which we have seen it is not true.

However, as we shall prove below, every finitely presentable perfect MV-algebra is finitely presentable as an algebra in Chang's variety, that is, as a model of $\mathbb C$; conversely, every finitely presentable model of $\mathbb C$ which is perfect is finitely presentable as a model of $\mathbb P$.

\begin{theorem}\label{fin.pre.perf}
Every finitely presentable perfect MV-algebra is finitely presented as an algebra in ${\mathbb C}\textrm{-mod}(\Set)$.
\end{theorem} 
\begin{proof}
Let $\mathcal{A}$ be a finitely presentable perfect MV-algebra, presented by a $\mathbb{P}$-irreducible geometric formula $\{\vec{x}.\phi\}$, with $\vec{x}=(x_1, \dots, x_n)$. This MV-algebra is finitely presented as an MV-algebra in ${\mathbb C}\textrm{-mod}(\Set)$ by the formula $\{\vec{x}.\phi\wedge x_{1}\leq \neg x_{1} \wedge \cdots \wedge x_{n}\leq \neg x_{n}\}$. Indeed, for any MV-algebra $\mathcal{B}$ in ${\mathbb C}\textrm{-mod}(\Set)$ and any tuple $\vec{y}\in[[\vec{x}.\phi\wedge x_{1}\leq \neg x_{1} \wedge \cdots \wedge x_{n}\leq \neg x_{n}]]_{\mathcal{B}}$, $y_{1}, \ldots, y_{n}\in Rad(\mathcal{B})$. Now, the MV-subalgebra of $\mathcal{B}$ generated by $Rad(\mathcal{B})$ is perfect, whence there exists a unique MV-algebra homomorphism $f:\mathcal{A}\rightarrow \langle Rad(\mathcal{B})\rangle\hookrightarrow {\mathcal B}$ such that $f(\vec{x})=\vec{y}$.  
\end{proof}

\section{The classifying topos for perfect MV-algebras}\label{classifying topos}

Recall that the theory $\mathbb{P}$ of perfect MV-algebras is a quotient of the theory ${\mathbb C}$ of MV-algebras in Chang's variety. From the Duality Theorem of \cite{Caramello1} we know that the classifying topos of $\mathbb{P}$ can be represented as a subtopos $\mathbf{Sh}(\textrm{f.p.}\mathbb{C}$-mod$(\Set)^{op},J_{\mathbb{P}})$ of the classifying topos $[\textrm{f.p.}\mathbb{C}$-mod$(\Set),\Set]$ of $\mathbb{C}$, where $J_{\mathbb{P}}$ is the Grothendieck topology associated to the quotient $\mathbb{P}$. 

\begin{theorem}(Theorem 6.26 \cite{Caramello5})
Let $\mathbb{T}'$ be a quotient of a theory of presheaf type $\mathbb{T}$ corresponding to a Grothendieck topology $J$ on the category $\textrm{f.p.}\mathbb{T}$-mod$(\Set)^{op}$ under the duality theorem of \cite{Caramello1}. Suppose that $\mathbb{T}'$ is itself of presheaf type. Then every finitely presentable $\mathbb{T}'$-model is finitely presentable also as a $\mathbb{T}$-model if and only if the topology $J$ is rigid.
\end{theorem}

From the theorem it follows that the topology $J_{\mathbb{P}}$ is rigid, since we proved in the last section that $\textrm{f.p.}\mathbb{P}$-mod$(\Set)\subseteq \textrm{f.p.}\mathbb{C}$-mod$(\Set)$. Moreover, from the remark following Theorem 6.26 \cite{Caramello5} we know that the $J_{\mathbb{P}}$-irreducible objects are precisely the objects of the category $\textrm{f.p.}\mathbb{P}$-mod$(\mathbf{Set})$.
 
We can describe the Grothendieck topology $J_{\mathbb{P}}$ explicitly as follows (cf. \cite{Caramello5} for the general method for calculating the Grothendieck topology associated to a quotient of a theory of presheaf type). The theory $\mathbb{P}$ of perfect MV-algebras  is obtained from $\mathbb{C}$ by adding the axioms 
\begin{itemize}
\item[P.3] $x\oplus x=x\vdash_{x}x=0\vee x=1$
\item[P.4] $x=\neg x\vdash_{x}\perp$
\end{itemize} 
or equivalently, 
\begin{itemize}
\item[ P.3'] $\inf(x, \neg x)=0\vdash_{x}x=0\vee x=1$
\item[ P.4 ] $x=\neg x\vdash_{x}\perp$
\end{itemize}
where $d$ is defined as follows: for any $x,y\in A$, $d(x,y)=(x\ominus y)\oplus (y\ominus x)$.

The axioms P.3' and P.4 generate two cosieves $S_{P.3'}$ and $S_{P.4}$ in $\textrm{f.p.}\mathbb{C}$-mod$(\Set)$, and consequently two sieves in $\textrm{f.p.}\mathbb{C}$-mod$(\Set)^{\textrm{op}}$. The topology $J_{\mathbb{P}}$ on $\textrm{f.p.}\mathbb{C}$-mod$(\Set)^{\textrm{op}}$ is generated by these sieves. Specifically,

\begin{itemize}
\item the cosieve $S_{P.3'}$ is generated by the canonical projections 
\begin{center}
$p_1:Free_{x}/(\inf(x,\neg x))\rightarrow Free_{x}/(x)$,
$p_2:Free_{x}/(\inf(x,\neg x))\rightarrow Free_{x}/(\neg x)$ 
\end{center}
\item the cosieve $S_{P.4}$ is the empty one on the trivial algebra in Chang's variety.
\end{itemize}

The cotopology induced by $J_{\mathbb{P}}$ on the category $\textrm{f.p.}\mathbb{C}$-mod$(\Set)$ is thus genera\-ted by the empty cosieve on the trivial algebra and the finite `multicompositions' of the pushouts of the generating arrows of the cosieve $S_{P.3'}$ along arbitrary homomorphisms in $\textrm{f.p.}\mathbb{C}$-mod$(\Set)$. We can describe these pushouts explici\-tly. Let $f:Free_{x}/(\inf(x,\neg x))\rightarrow \mathcal{A}$ be an MV-homomorphism in $\textrm{f.p.}\mathbb{C}$-mod$(\Set)$; then the pushouts of the generating arrows of $S_{P.3'}$ along $f$ are: $f_1:\mathcal{A}\rightarrow \mathcal{A}/(a)$ and $f_2:\mathcal{A}\rightarrow \mathcal{A}/(\neg a)$, where $a=f([x])\in A$ satisfies $\inf(a,\neg a)=0$.


We shall say that an MV-algebra $\mathcal A$ is a \emph{weak subdirect product} of a family $\{{\mathcal{A}_{i} \mid i\in I}\}$ of MV-algebras if the arrows $\mathcal{A}\to \mathcal{A}_{i}$ are jointly injective (equivalently, jointly monic).


Note that every weak subdirect product of finitely presented perfect MV-algebras is in $\mathbb{C}$-mod$(\Set)$. Indeed, perfect MV-algebras are in $\mathbb{C}$-mod$(\Set)$ and the identities that define this variety are preserved by weak subdirect products. It is natural to wonder if the converse is true, that is if every algebra in $\mathbb{C}$-mod$(\Set)$ is a weak subdirect product of finitely presented perfect MV-algebras. We shall prove in the following that the answer is affirmative. 

\begin{theorem}\label{thm:subdirect_product}
Every finitely presented non-trivial MV-algebra in $\mathbb{C}$-mod$(\Set)$ is a direct product of a finite family of finitely presented perfect MV-algebras. In fact, the topology $J_{\mathbb{P}}$ is subcanonical.
\end{theorem}
\begin{proof}

Let $\mathcal{A}\in \mathbb{C}$-mod$(\Set)$ be a finitely presented non-trivial MV-algebra. This algebra satisfies the axiom P.4 (cf. [Lemma \ref{lemmaP3}]); thus, the only non-trivial $J_{\mathbb{P}}$-coverings of $\mathcal{A}$ are those which contain a cosieve generated by finite multicompositions of the pushouts of $p_1$ and $p_2$, that is

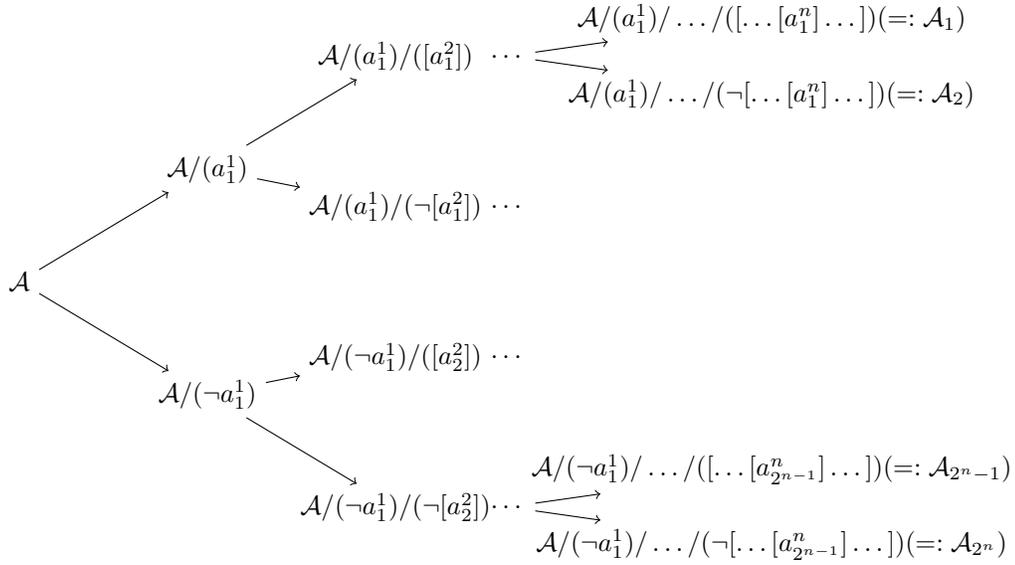
\begin{center}
\begin{tikzpicture}
\node (1) at (0,0) {$\mathcal{A}$};
\node (2) at (2.5,1.5) {$\mathcal{A}/(a_1^1)$};
\node (3) at (2.5,-1.5) {$\mathcal{A}/(\neg a_1^1)$};
\node (4) at (5,3) {$\mathcal{A}/(a_1^1)/([a_1^2])$};
\node (5) at (5,1) {$\mathcal{A}/(a_1^1)/(\neg[a_1^2])$};
\node (6) at (5,-1) {$\mathcal{A}/(\neg a_1^1)/([a_2^2])$};
\node (7) at (5,-3) {$\mathcal{A}/(\neg a_1^1)/(\neg[a_2^2])$};
\node (8) at (6.5,3) {$\dots$} ;
\node (9) at (6.5,1) {$\dots$};
\node (10) at (6.5,-1) {$\dots$};
\node (11) at (6.5,-3) {$\dots$};
\node (12) at (10,3.5) {$\mathcal{A}/(a_1^1)/\dots/([\dots[a^{n}_1]\dots])(=:\mathcal{A}_1)$};
\node (13) at (10,2.5) {$\mathcal{A}/(a_1^1)/\dots/(\neg[\dots[a^{n}_1]\dots])(=:\mathcal{A}_2)$};
\node (14) at (10,-2.5) {$\mathcal{A}/(\neg a_1^1)/\dots/([\dots[a^{n}_{2^{n-1}}]\dots])(=:\mathcal{A}_{2^{n}-1})$};
\node (15) at (10,-3.5) {$\mathcal{A}/(\neg a_1^1)/\dots/(\neg[\dots[a^{n}_{2^{n-1}}]\dots])(=:\mathcal{A}_{2^n})$};
\draw[->](1) to node [below]{} (2);
\draw[->](1) to node [below]{} (3);
\draw[->](2) to node [below]{} (4);
\draw[->](2) to node [below]{} (5);
\draw[->](3) to node [below]{} (6);
\draw[->](3) to node [below]{} (7);
\draw[->](8) to node [below]{} (12);
\draw[->](8) to node [below]{} (13);
\draw[->](11) to node [below]{} (14);
\draw[->](11) to node [below]{} (15);
\end{tikzpicture}
\end{center}


Now, the pushout-pullback lemma (Lemma 7.1 \cite{DubucPoveda}) asserts that for any MV-algebra $\mathcal{A}$ and any elements $x,y \in \mathcal{A}$, the following pullback diagram is also a pushout:
\[  
\xymatrix {
\mathcal{A}\slash (\inf(x, y)) \ar[d] \ar[r] &  \mathcal{A}\slash (y) \ar[d]  \\
\mathcal{A}\slash (x) \ar[r]  & \mathcal{A}\slash (\sup(x, y)).}
\] 

Note that if $\inf(x, y)=0$ then $\sup(x, y)=x\oplus y$; in particular, for any Boolean element $x$ of $\mathcal A$, $\mathcal A$ is the product of $\mathcal{A}\slash (x)$ and $\mathcal{A}\slash (\neg x)$. The same reasoning can be repeated for every pair of arrows in the diagram. It follows that the MV-algebra $\mathcal{A}$ is the direct product of the $\mathcal{A}_i$. 

Since $J_{\mathbb{P}}$ is rigid and the $J_{\mathbb{P}}$-irreducible objects are the finitely presented perfect MV-algebras, there is a $J_{\mathbb{P}}$-covering of $\mathcal{A}$ such that all the $\mathcal{A}_i$ are finitely presented perfect MV-algebras.

Finally, we observe that for any Boolean element $x$ of an MV-algebra $\mathcal{A}$, there is a unique arrow $\mathcal{A}\slash (x) \to \mathcal{A}\slash (\neg x)$ over $\mathcal{A}$ if and only if $x=0$, whence the sieve generated by the family $\{\mathcal{A} \to \mathcal{A}\slash (x), \mathcal{A} \to \mathcal{A}\slash (\neg x) \}$ is effective epimorphic in $\textrm{f.p.}\mathbb{C}$-mod$(\Set)^{\textrm{op}}$ if and only if $\{\mathcal{A} \to \mathcal{A}\slash (x), \mathcal{A} \to \mathcal{A}\slash (\neg x) \}$ is a product diagram in $\textrm{f.p.}\mathbb{C}$-mod$(\Set)$. 

This proves our statement. 
\end{proof}

We can give a more explicit description of a family of finitely presented perfect MV-algebras $\{\mathcal{A}_1,\dots,\mathcal{A}_{m}\}$ such that the family of arrows $\{\mathcal{A}\to \mathcal{A}_{i} \mid i\in \{1, \ldots m\}\}$ as in the proof of Theorem \ref{thm:subdirect_product} generates a $J_{\mathbb P}$-covering sieve. 

\begin{lemma}\label{Boolgen}
Let $\mathcal{A}$ be an MV-algebra in $\mathbb{C}$-mod$(\Set)$ generated by elements $\{x_1,\dots,x_n\}$. Then the boolean skeleton\footnote{Recall that the Boolean skeleton of an MV-algebra is the subalgebra  $\mathcal{B}(\mathcal{A})=\{x\in A\mid x\oplus x=x\}$ of Boolean elements of $\mathcal{A}$, cf. p. 25 \cite{Mundici}.} $\mathcal{B}(\mathcal{A})$ of $\mathcal{A}$ is finitely generated by the family $\{(2x_1)^2,\dots,(2x_n)^2\}$.
\end{lemma}
\begin{proof}
By axiom P.2 for every $x\in A$, $(2x)^2\in \mathcal{B}(\mathcal{A})$ and from Theorem 5.12 \cite{P-L} we know that an MV-algebra $\mathcal{A}$ is in $\mathbb{C}$-mod$(\Set)$ if and only if $\mathcal{A}\slash {Rad(\mathcal{A})}\simeq \mathcal{B}(\mathcal{A})$, where the isomorphism is given by the following map: 
\begin{center}
$f:x\in A\rightarrow (2x)^2\in \mathcal{B}(\mathcal{A})$
\end{center} 
If $\mathcal{A}$ is an MV-algebra in $\mathbb{C}$-mod$(\Set)$ generated by $\{x_1,\dots,x_n\}$ then the quotient $\mathcal{A}\slash {Rad(\mathcal{A})}$ is generated by $\{[x_1],\dots,[x_n]\}$; hence, $\mathcal{B}(\mathcal{A})$ is generated by the family $\{(2x_1)^2,\dots,(2x_n)^2\}$.
\end{proof}

Recall that if an MV-algebra $\mathcal{A}$ is finitely presented then it is finitely generated. Let $\mathcal{A}=<x_1,\dots,x_n>$ be an MV-algebra as in Theorem \ref{thm:subdirect_product}. From Lemma \ref{Boolgen} it follows that a family of finitely presented perfect MV-algebras that $J_{\mathbb P}$-covers $\mathcal{A}$ is given by $\{\mathcal{A}_1,\dots,\mathcal{A}_{2^n}\}$ (in the notation of Theorem \ref{thm:subdirect_product}), where $a_j^i=(2x_i)^2$ for all $j=1,\dots,2^{i-1}$. Indeed, the iterated quotients of the previous diagram actually remove every non-trivial boolean element, thus every $\mathcal{A}_k$ is perfect. In fact, this argument shows more generally that every finitely generated algebra in $\mathbb{C}$-mod$(\Set)$ can be represented as a direct product of finitely generated perfect MV-algebras.

This result can be alternatively deduced from existing theorems on weak Boolean products of MV-algebras as follows. First, we need a lemma, clarifying the relationship between finite direct products and weak Boolean products of MV-algebras. Recall from \cite{CDM} that a \emph{weak Boolean product} of a family $\{\mathcal{A}_{x} \mid x\in X\}$ of MV-algebras is a subdirect product $\mathcal{A}$ of the given family, in such a way that $X$ can be endowed with a Boolean (i.e. Stone) topology satisfying the following conditions (where $\pi_{x}:\mathcal{A} \to \mathcal{A}_{x}$ are the canonical projections):
\begin{enumerate}[(i)]
\item for all $f,g\in A$, the set $\{x\in X \mid \pi_{x}(f)=\pi_{x}(g)\}$ is open in $X$;

\item for every clopen set $Z$ of $X$ and any $f, g\in A$, there exists a unique element $h\in A$ such that $\pi_{x}(h)=\pi_{x}(f)$ for all $x\in Z$ and $\pi_{x}(h)=\pi_{x}(g)$ for all $x\in X\setminus Z$.
\end{enumerate} 

\begin{lemma}\label{wbb}
Let $\mathcal{A}$ be a weak Boolean product of a finite family $\{\mathcal{A}_{x} \mid x\in X\}$ of MV-algebras. Then the topology of $X$ is discrete and $\mathcal{A}$ is a finite direct product of the $\mathcal{A}_{x}$.   
\end{lemma}

\begin{proof}
It is clear that the only Boolean topology on finite set is the discrete one. To prove that $\mathcal{A}$ is a finite direct product of the $\mathcal{A}_{x}$ via the weak Boolean product projections $\pi_{x}$, it suffices to verify that for every family $\{z_{x}\in A_{x}\}$ of elements of the $\mathcal{A}_{x}$ there exists an element $h\in A$ such that $\pi_{x}(h)=z_{x}$ for all $x\in X$. Since $\mathcal{A}$ is a subdirect product of the $\mathcal{A}_{x}$, the functions $\pi_{x}$ are all surjective. By choosing, for each $x\in X$, an element $a_{x}\in A$ such that $\pi_{x}(a_{x})=z_{x}$ and repeatedly applying condition (ii) to such elements (taking $Z$ to be the singletons $\{x\}$ for $x\in X$), we obtain the existence of an element $h\in A$ such that $\pi_{x}(h)=\pi_{x}(a_{x})=z_{x}$ for each $x\in X$, as required.    
\end{proof}

Now, by Lemma 9.4 in \cite{DiNola3}, every algebra $\mathcal{A}$ in $V(C)$ is quasi-perfect, that is, by Theorem 5.9 in \cite{DiNola4}, it is a weak Boolean product of perfect MV-algebras. By Theorem 6.5.2 in \cite{CDM}, the indexing set of this Boolean product identifies with the set of ultrafilters of the Boolean algebra $\mathcal{B}(\mathcal{A})$. But by Lemma \ref{Boolgen} the set of ultrafilters of $\mathcal{B}(\mathcal{A})$ is finite, and can be identified with the set of atoms of $\mathcal{B}(\mathcal{A})$, since $\mathcal{B}(\mathcal{A})$ is finitely generated and hence finite. By Lemma \ref{wbb}, we can then conclude that the given weak Boolean product is in fact a finite direct product.     


\begin{theorem}\label{StoneChang}
Every MV-algebra in $\mathbb{C}$-mod$(\Set)$ is a weak subdirect product of (finitely presentable) perfect MV-algebras.
\end{theorem}
 
\begin{proof}
Since every MV-algebra in $\mathbb{C}$-mod$(\Set)$ is a filtered colimit of finitely presented MV-algebras in $\mathbb{C}$-mod$(\Set)$, it suffices to prove the statement for the finitely presentable MV-algebras in $\mathbb{C}$-mod$(\Set)$; indeed, an MV-algebra is a weak subdirect product of finitely presentable perfect MV-algebras if and only if the arrows from it to such algebras are jointly monic. But this follows from Theorem \ref{thm:subdirect_product}. 
\end{proof} 

\begin{remark}
Theorem \ref{StoneChang} represents a constructive version of \cite[Lemma 9.6]{DiNola3}, which asserts that every MV-algebra in Chang's variety $V(C)$ is quasi-perfect, i.e. a weak Boolean product of perfect MV-algebras.
\end{remark}

It is natural to wonder if one can intrinsically characterize the class of MV-algebras in $\mathbb{C}$-mod$(\Set)$ which are direct products of perfect MV-algebras. We already know from the discussion above that all the finitely generated MV-algebras in $\mathbb{C}$-mod$(\Set)$ belong to this class. 

The following lemma, which generalizes its finitary version given by Lemmas 6.4.4 and 6.4.5 \cite{CDM} as well as the version for complete MV-algebras given by Lemma 6.6.6 \cite{CDM}, will be useful in this respect. Relevant references on the relationship between direct product decompositions of MV-algebras and Boolean elements are \cite{Jacubik}, \cite{Rodriguez} and sections 6.4-5-6 of \cite{CDM}.

\begin{lemma}\label{lem:gencomplete}
Let $\mathcal{A}$ be an MV-algebra. Then the following two conditions are equivalent:
\begin{enumerate}[(i)]
\item $\mathcal{A}$ is a direct product of MV-algebras $\mathcal{A}_{i}$ (for $i\in I$);

\item There exists a family $\{a_i\mid i\in I\}$ of Boolean pairwise disjoint elements of $\mathcal{A}$ such that every family of elements of the form $\{z_{i}\leq a_{i} \mid i\in I\}$ has a supremum $\mathbin{\mathop{\textrm{\huge $\vee$}}\limits_{i\in I}}z_{i}$ in $\mathcal{A}$ and every element $a$ of $\mathcal{A}$ can be expressed (uniquely) in this form. 
\end{enumerate}
\end{lemma} 

\begin{proof}
Let $\mathcal{A}$ be an MV-algebra that is direct product of a family $\{\mathcal{A}_i\mid i\in I\}$ of MV-algebras. The elements $a_{i}$ of the MV-algebra $\mathbin{\mathop{\prod}\limits_{i\in I}}\mathcal{A}_{i}$ which are $0$ everywhere except at the place $i$ where it is equal to $1$ are Boolean and satisfy the following properties: they are pairwise disjoint (i.e., $a_i\wedge a_{i}=0$ whenever $i\neq i'$), $1=\mathbin{\mathop{\textrm{\huge $\vee$}}\limits_{i\in I}}a_{i}$, every family of elements of the form $\{z_{i}\leq a_{i} \mid i\in I\}$ has a supremum $\mathbin{\mathop{\textrm{\huge $\vee$}}\limits_{i\in I}}z_{i}$ in $\mathbin{\mathop{\prod}\limits_{i\in I}}\mathcal{A}_{i}$ and every element $a$ of $\mathbin{\mathop{\prod}\limits_{i\in I}}\mathcal{A}_{i}$ can be expressed uniquely in this form. 
  
Conversely, suppose that $\{a_{i}\in \mathcal{A} \mid i\in I\}$ is a set of Boolean pairwise disjoint elements of an MV-algebra $\mathcal{A}$ such that every family of elements of the form $\{z_{i}\leq a_{i} \mid i\in I\}$ has a supremum $\mathbin{\mathop{\textrm{\huge $\vee$}}\limits_{j\in J}}z_{j}$ in $\mathcal{A}$ and every element $a$ of $\mathcal{A}$ can be expressed uniquely in this form. Then $\mathcal{A}$ is isomorphic to the product of the MV-algebras $(a_{i}]$ considered in \cite{CDM} (cf. Corollary 1.5.6) via the canonical homomorphism $\mathcal{A} \to \mathbin{\mathop{\prod}\limits_{i\in I}}(a_{i}]$ (equivalently, by Proposition 6.4.3 \cite{CDM}, to the product of the quotient algebras $\mathcal{A}\slash (\neg a_{i})$ via the canonical projections). Indeed, the canonical homomorphism $\mathcal{A} \to \mathbin{\mathop{\prod}\limits_{i\in I}}(a_{i}]$, which sends any element $b$ of $\mathcal{A}$ to the string $(b\wedge a_{i})$ admits as inverse the map sending a tuple $(z_{i})$ in $\mathbin{\mathop{\prod}\limits_{i\in I}}(a_{i}]$ to the supremum $\mathbin{\mathop{\textrm{\huge $\vee$}}\limits_{i\in I}}z_{i}$. This can be proved as follows. The composite of the former homomorphism with the latter is clearly the identity, so it remains to prove the converse. Given an element $b\in \mathcal{A}$, we have to prove that $b=\mathbin{\mathop{\textrm{\huge $\vee$}}\limits_{i\in I}}(b\wedge a_{i})$. Set $b'=\mathbin{\mathop{\textrm{\huge $\vee$}}\limits_{i\in I}}(b\wedge a_{i})$. Clearly, $b'\leq b$. Now, by our hypothesis, we can decompose $b$ in the form $b=\mathbin{\mathop{\textrm{\huge $\vee$}}\limits_{i\in I}}c_{i}$ where $c_{i}\leq a_{i}$ for each $i$. Now, $c_{i}\leq b$, whence $c_{i}\leq a_{i}\wedge b$ and $b=\mathbin{\mathop{\textrm{\huge $\vee$}}\limits_{i\in I}}c_{i}\leq \mathbin{\mathop{\textrm{\huge $\vee$}}\limits_{i\in I}}(b\wedge a_{i})=b'$. So $b=b'$, as required.
\end{proof}

\begin{remark}
The algebras $\mathcal{A}_{i}$ as in the first condition are given by the quotients $\mathcal{A}\slash (\neg a_{i}]$, while the elements $a_{i}$ of the product $\mathbin{\mathop{\prod}\limits_{i\in I}}\mathcal{A}_{i}$ satisfying the second conditions are the tuples which are zero everywhere except at the place $i$ where they are equal to $1$. 
\end{remark}

In order to achieve an intrinsic characterization of the MV-algebras $\mathcal{A}$ which are products of perfect MV-algebras, it remains to characterize the elements $a_{i}$ such that $\mathcal{A}\slash (\neg a_{i})$ is a perfect MV-algebra. Since $\mathcal{A}$ is in $\mathbb{C}$-mod$(\Set)$, this amounts to requiring that $a_i$ is Boolean and for every element $x$ such that $x\wedge \neg x\leq \neg a_{i}$ (equivalently, $x\wedge \neg x \wedge a_{i}=0$), either $x\leq \neg a_{i}$ (equivalently, $x\wedge a_{i}=0$) or $\neg x\leq \neg a_{i}$ (equivalently, $a_{i}\leq x$) but not both. We shall call such elements the \emph{perfect elements} of the algebra $\mathcal{A}$. 

Summarizing, we have the following result.

\begin{theorem}\label{MVatomic}
For a MV-algebra $\mathcal{A}$, the following conditions are equivalent:
\begin{enumerate}[(i)]
\item $\mathcal{A}$ is isomorphic to a direct product of perfect MV-algebras;
\item $\mathcal{A}$ belongs to $\mathbb{C}$-mod$(\Set)$ and there exists a family of Boolean pairwise disjoint perfect elements of $\mathcal{A}$ such that every family of elements of the form $\{z_{i}\leq a_{i} \mid i\in I\}$ has a supremum $\mathbin{\mathop{\textrm{\huge $\vee$}}\limits_{j\in J}}z_{j}$ in $\mathcal{A}$ and every element $a$ of $\mathcal{A}$ can be expressed (uniquely) in this form. 
\end{enumerate}
\end{theorem}

\begin{remark}
By Theorem \ref{StoneChang}, every finitely generated MV-algebra $\mathcal{A}$ in $\mathbb{C}$-mod$(\Set)$ satisfies these conditions. In fact, for every finite set $\{x_1, \ldots, x_n\}$ of generators of $\mathcal{A}$, a family of elements satisfying the hypotheses of Lemma \ref{lem:gencomplete} is given by the family of finite meets of the form $u_{1}\wedge \cdots \wedge u_{n}$ where for each $i$, $u_{i}$ is either equal to $(2x_{i})^{2}$ or its complement $\neg(2x_{i})^{2}$.   
\end{remark}

The class of MV-algebras in $\mathbb{C}$-mod$(\Set)$ naturally generalizes that of Boolean algebras (recall that every Boolean algebra is an MV-algebra, actually lying in $\mathbb{C}$-mod$(\Set)$), with perfect algebras representing the counterpart of the algebra $\{0,1\}$ and powerset algebras, that is products of the algebra $\{0,1\}$, corresponding to products of perfect MV-algebras. The class of Boolean algebras isomorphic to powersets can be intrinsically characterized, thanks to Lindenbaum-Tarski's theorem, as that of complete atomic Boolean algebras. Theorem \ref{StoneChang} represents a natural generalization in this setting of the Stone representation of a Boolean algebra as a field of sets, while Theorem \ref{MVatomic} represents the analogue of Lindenbaum-Tarski's theorem. Note that, as every Boolean algebra with $n$ generators is a product of $2^{n}$ copies of the algebra $\{0,1\}$, so every finitely presented algebra in $\mathbb{C}$-mod$(\Set)$ with $n$ generators is a product of $2^{n}$ finitely presented perfect MV-algebras (cf. Theorem \ref{thm:subdirect_product} above). These relationships are summarized in the table below.

\setlength{\arrayrulewidth}{0.008mm}

\begin{center}
\begin{tabular*}{122mm}{p{5.5cm}lp{5.6cm}}
\textbf{Classical context} & \vline & \textbf{MV-algebraic generalization} \\ \hline\hline
Boolean algebra & \vline & MV-algebra in $\mathbb{C}$-mod$(\Set)$ \\ \hline
$\{0,1\}$ & \vline & Perfect MV-algebra  \\ \hline 
Powerset $\cong$ product of $\{0,1\}$ & \vline & Product of perfect MV-algebras \\ \hline
\raisebox{-1.5ex}[0pt]{Finite Boolean algebra} & \vline & Finitely presentable MV-algebra in $\mathbb{C}$-mod$(\Set)$ \\ \hline
\raisebox{-1.5ex}[0pt]{Complete atomic Boolean algebra} & \vline & MV-algebra in $\mathbb{C}$-mod$(\Set)$ satisfying the hypotheses of Theorem \ref{MVatomic} \\ \hline
Representation theorem for finite Boolean algebras & \vline & \raisebox{-1.5ex}[0pt]{Theorem \ref{thm:subdirect_product}} \\ \hline
Stone representation for Boolean algebras & \vline & \raisebox{-1.5ex}[0pt]{Theorem \ref{StoneChang}} \\ \hline 
Lindenbaum-Tarski's theorem & \vline & Theorem \ref{MVatomic}  
\end{tabular*}
\end{center}

\vspace{0.3cm}

\section{A related Morita-equivalence}\label{sec:relatedMoritaequivalence}

Finally, we discuss the relationship between the category of perfect MV-algebras and that of lattice-ordered abelian groups with strong unit. Generalizing the work \cite{Yosida} of Belluce and Di Nola concerning locally archimedean MV-algebras and archimedean $\ell$-u groups, we establish a Morita-equivalence between a category of pointed perfect MV-algebras and the category of $\ell$-u groups. This will allow us to reinterpret in the context of $\ell$-groups the representation results for the MV-algebras in Chang's variety obtained in the last section.

We call a perfect MV-algebra \textit{pointed} if its radical is generated by a single element. This class of algebras can be axiomatized. Let us extend the signature $\mathcal{L}_{MV}$ by adding a new constant symbol $a$. We call $\mathbb{P}^*$ the theory over this signature whose axioms are those of $\mathbb{P}$ plus:
\begin{itemize}
\item[P*.1] $\top\vdash a\leq \neg a$
\item[P*.2] $x\leq \neg x\vdash_x\mathbin{\mathop{\textrm{\huge $\vee$}}\limits_{n\in\mathbb{N}}}x\leq na$
\end{itemize}

We shall prove that the theory $\mathbb{P}^*$ is Morita-equivalent to the theory $\mathbb{L}_u$. Indeed we can ``restrict'' the functors $\Delta$ and $\Sigma$ respectively to the categories $\mathbb{P}^*$-mod$(\mathscr{E})$ and $\mathbb{L}_u$-mod$(\mathscr{E})$, for every Grothendieck topos $\mathscr{E}$, and show that they are still categorical inverses to each other.

Let $\mathcal{A}=(A,a)$ be a model of $\mathbb{P}^*$ in $\mathscr{E}$. This structure, without the constant $a$, is a perfect MV-algebra in $\E$. We can thus consider $\Delta(\mathcal{A})$ and we know that it is a model of $\mathbb{L}$ in $\E$.

\begin{proposition}
The structure $(\Delta(\mathcal{A}),[a,0])$ is a model of $\mathbb{L}_u$ in $\E$.
\end{proposition}
\begin{proof}
We already know that $\Delta(\mathcal{A})$ is an $\ell$-group in $\E$, so it remains to prove that $[a,0]$ is a strong unit for it.
\begin{itemize}
\item[-] $[a,0]\geq [0,0] \Leftrightarrow \inf([a,0],[0,0])=[0,0]\Leftrightarrow [\inf(a\oplus 0,0\oplus 0),0\oplus 0]=[0,0]\Leftrightarrow [0,0]=[0,0]$. Thus, L$_u$.1 holds.
\item[-] given $[x,y]\in \Delta(\mathcal{A})$ such that $[x,y]\geq [0,0]$, we have that $x,y\in Rad(A)$, i.e. $x\leq \neg x$ and $y\leq \neg y$. By axiom P*.2 we have $\mathbin{\mathop{\textrm{\huge $\vee$}}\limits_{n\in \mathbb{N}}} x\leq na$ and $\mathbin{\mathop{\textrm{\huge $\vee$}}\limits_{n\in\mathbb{N}}}y\leq ma$. Further, by definition of the order relation in $\Delta(A)$
\begin{center}
 $[x,y]\geq [0,0]\Leftrightarrow x\geq y$
\end{center}
Thus $\mathbin{\mathop{\textrm{\huge $\vee$}}\limits_{n\in\mathbb{N}}}y\leq x\leq na$ and $\mathbin{\mathop{\textrm{\huge $\vee$}}\limits_{n\in\mathbb{N}}} [x,y]\leq n[a,0]$. Therefore L$_u$.2 holds.
\end{itemize}
\end{proof}

Let $\mathcal{A}=(A,a)$ and $\mathcal{A}'=(A',a')$ be two models of $\P$ in $\E$ and $h:\mathcal{A}\rightarrow \mathcal{A}'$ an arrow in $\P$-mod$(\E)$, i.e., an MV-homomorphism such that $h(a)=a'$. We can consider $\Delta(h)$. This is an $\ell$-homomorphism satisfying $\Delta(h)([a, 0])=[h(a),0]=[a',0]$. So $\Delta(h)$ defines an ${\mathbb L}_{u}$-model homomorphism $(\Delta(\mathcal{A}), [a, 0]) \to (\Delta(\mathcal{A'}), [a', 0])$. Thus $\Delta$ is a functor from $\P$-mod$(\E)$ to $\mathbb{L}_u$-mod$(\E)$.

In the converse direction, let $\mathcal{G}=(G,u)$ be a model of $\mathbb{L}_u$ in $\E$. We know that $\Sigma(\mathcal{G})$ is a model of $\mathbb{P}$ in $\E$.

\begin{proposition}
The structure $(\Sigma(G),(0,u))$ is a model of $\P$ in $\mathbb{E}$.
\end{proposition}
\begin{proof}
It remains to show that this structure satisfies P*.1 and P*.2. 
\begin{itemize}
\item[-] $\neg(0,u)=(1,u)\geq(0,u)$. Thus, P*.1 holds.
\item[-] let $(c,x)$ be an element of $\Sigma(G)$ such that $(c,x)\leq \neg (c,x)$. By Theorem \ref{Sigma}, $(c,x)=(0,y)$ with $y\geq 0$. Thus, by axiom L$_u$.2, we have $\mathbin{\mathop{\textrm{\huge $\vee$}}\limits_{n\in\mathbb{N}}}y\leq nu$. Hence, $\mathbin{\mathop{\textrm{\huge $\vee$}}\limits_{n\in\mathbb{N}}}(0,y)\leq n(0,u)$ and P*.2 holds.  
\end{itemize}
\end{proof}

It is easily seen that $\Sigma$ is a functor from $\mathbb{L}_u$-mod$(\E)$ to $\P$-mod$(\E)$, i.e., that $\Sigma(h)$ is an MV-homomorphism which preserves the generating element of the radical for every $\ell$-unital homomorphism $h$.

\begin{theorem}
The categories $\P$-mod$(\E)$ and $\mathbb{L}_u$-mod$(\E)$ are equivalent, naturally in $\E$. Hence the theories $\mathbb{P}^{\ast}$ and ${\mathbb L}_{u}$ are Morita-equivalent.
\end{theorem}
\begin{proof}
This immediately follows from Theorem \ref{Moritaeq1} noticing that the isomorphisms $\beta_A:A\rightarrow \Sigma\circ\Delta(A)$ and $\varphi_G: G\rightarrow \Delta\circ\Sigma(G)$ defined in the proof of Theorem \ref{equivalence} satisfy:

$\beta_A(a)=(0,[a,0])$;

$\varphi_G(u)=[(0,u),(0,0)]$.
\end{proof}

\begin{obs}
From the main result of \cite{Russo} we obtain that the theory $\P$ is Morita-equivalent to the theory $\mathbb{MV}$.
\end{obs}

\begin{obs}
In \cite{l-u groups_perfect}, the authors characterized the class of $\ell$-groups with strong unit that corresponds to that of perfect MV-algebras via the $\Gamma$ functor of Mundici's equivalence. These groups are called \textit{antiarchimedean}. Specifically, let $\mathbb{A}nt$ be the quotient of the theory $\mathbb{L}_u$ of $\ell$-groups with strong unit obtained by adding the following axioms:
\begin{center}
$(0\leq x\wedge x\leq u\vdash_{x} \sup(0,2\inf(2x,u)-u)=\inf(u,2\sup(2x-u,0))$

$(0\leq x\wedge x\leq u\wedge \inf(2x,u)=x\vdash_{x} x=0\vee x=u )$
\end{center}

By the results of \cite{Russo}, the quotient $\mathbb{A}nt$ is Morita-equivalent to the theory $\mathbb{P}$; hence, the theory $\mathbb{A}nt$ is Morita-equivalent to the theory of lattice-ordered abelian groups $\mathbb{L}$, by Theorem \ref{Moritaeq1}. It follows that an $\ell$-u group is antiarchimedean if and only if it is isomorphic to a $\ell$-u group of the form $\mathbb{Z}\times_{lex} \mathcal{G}$, for an $\ell$-group $\mathcal{G}$.
\end{obs}

\section{Some applications}\label{sec:someapplications}

In this section we describe some applications, obtained by applying the technique `toposes as bridges' of \cite{Caramello1}, of the main results of the paper.

Recall from \cite{Caramello5} that, given a geometric theory $\mathbb T$ over a signature $\Sigma$, its \emph{cartesianization} is the sub-theory of $\mathbb T$ consisting of all the $\mathbb T$-cartesian sequents which are provable in $\mathbb T$.

\begin{proposition}\label{prop:cartesianization}
The theory $\mathbb C$ axiomatizing Chang's variety $V(C)$ coincides with the cartesianization of the theory $\mathbb P$ of perfect MV-algebras. That is, for any $\mathbb C$-cartesian sequent $\sigma=(\phi \vdash_{\vec{x}} \psi)$, $\sigma$ is provable in $\mathbb C$ (equivalently, valid in all algebras in $V(C)$) if and only if it is provable in $\mathbb P$ (that is, valid in all perfect MV-algebras). 

Moreover, for any $\mathbb C$-cartesian formulae $\phi(\vec{x})$ and $\psi(\vec{y})$ and a geometric formula $\theta(\vec{x}, \vec{y})$, $\theta$ is $\mathbb P$-provably functional from $\phi(\vec{x})$ to $\psi(\vec{y})$ if and only if it is $\mathbb C$-provably functional from $\phi(\vec{x})$ to $\psi(\vec{y})$. 
\end{proposition}

\begin{proof}
The theory $\mathbb{C}$ is algebraic, hence it is of presheaf type. By \cite[Corollary D3.1.2]{SE}, the universal model $U_{\mathbb{C}}$ of $\mathbb{C}$ in its classifying topos $[\textrm{f.p.}\mathbb{C}\textrm{-mod}(\Set),\Set]$ is given by $Hom_{\textrm{f.p.}\mathbb{C}\textrm{-mod}(\Set)}(F, -)$, where $F$ is the free $\mathbb{C}$-algebra on one generator. Since $J_{\mathbb{P}}$ is subcanonical, the model $U_{\mathbb{C}}$ is also a universal model of $\mathbb{P}$ in the topos $\Sh(\textrm{f.p.}\mathbb{C}\textrm{-mod}(\Set)^{\textrm{op}}, J_{\mathbb P})$ (cf. \cite[Lemma 2.1]{Caramello3}). Now, given a geometric theory $\mathbb{T}$, a geometric sequent over its signature is provable in $\mathbb{T}$ if and only if it is satisfied in its universal model $U_{\mathbb{T}}$ (cf. \cite[Theorem D1.4.6]{SE}). From this the first part of the proposition follows at once.

The second part follows from the fact that the canonical functor $r:{\mathcal C}_{\mathbb C}^{\textrm{cart}}\simeq \textrm{f.p.}\mathbb{C}\textrm{-mod}(\Set)^{\textrm{op}}\to \Sh(\textrm{f.p.}\mathbb{C}\textrm{-mod}(\Set)^{\textrm{op}}, J_{\mathbb P})$ is full and faithful since the topo\-logy $J_{\mathbb{P}}$ is subcanonical. Recalling from \cite[Theorem 2.2]{Caramello3} that, given a universal model $U$ of a geometric theory $\mathbb T$ in its classifying topos $\mathscr{E}$, for any geometric formulas $\{\vec{x}. \phi\}$ and $\{\vec{y}. \psi\}$ over the signature of $\mathbb T$, the arrows $[[\vec{x}. \phi]]_{U}\to [[\vec{y}. \psi]]_{U}$ in $\mathscr{E}$ correspond exactly to the $\mathbb{T}$-provably functional formulae from $\{\vec{x}. \phi\}$ to $\{\vec{y}. \psi\}$, the thesis follows immediately.
\end{proof}

Another application concerns definability and functional completeness. 

\begin{proposition}

The following definability properties of the theory $\mathbb P$ in relation to the theory $\mathbb C$ hold:

\begin{enumerate}[(i)] 
\item Every property $P$ of tuples $\vec{x}$ of elements of perfect MV-algebras which is preserved by arbitrary MV-algebra homomorphisms and by filtered colimits of perfect MV-algebras is definable by a geometric formula $\phi(\vec{x})$ over the signature of $\mathbb P$. For any two geometric formulae $\phi(\vec{x})$ and $\psi(\vec{y})$ over the signature of $\mathbb P$, every assignment $M\to f_{M}:[[\vec{x}. \phi]]_{M}\to [[\vec{y}. \psi]]_{M}$ (for finitely presented perfect MV-algebras $M$) which is natural in $M$ is definable by a $\mathbb P$-provably functional formula $\theta(\vec{x}, \vec{y})$ from $\phi(\vec{x})$ to $\psi(\vec{y})$.

\item The properties $P$ of tuples $\vec{x}$ of elements of perfect MV-algebras which are preserved by arbitrary MV-algebra homomorphisms and by filtered colimits of perfect MV-algebras are in natural bijection with the properties $Q$ of tuples $\vec{x}$ of elements of algebras in $\mathbb{C}\textrm{-mod}(\Set)$ which are preserved by filtered colimits of algebras in $\mathbb{C}\textrm{-mod}(\Set)$ and such that for any finitely presented algebra $\mathcal{A}$ in $\mathbb{C}\textrm{-mod}(\Set)$ and any Boolean element $a$ of $\mathcal{A}$, the canonical projections ${\mathcal A}\to \mathcal{A}\slash (a)$ and ${\mathcal A}\to \mathcal{A}\slash (\neg a)$ jointly reflect $Q$.
\end{enumerate} 
\end{proposition}

\begin{proof}
The first part of the theorem follows from Theorem \ref{definability} in light of the fact that $\mathbb P$ is of presheaf type.

The second part follows from the fact that the properties $P$ of tuples $\vec{x}=(x_{1}, \ldots, x_{n})$ of elements of perfect MV-algebras which are preserved by arbitrary MV-algebra homomorphisms and by filtered colimits of perfect MV-algebras correspond precisely to the subobjects of $U\times \cdots \times U$ in the classifying topos of $\mathbb P$, where $U$ is a universal model of $\mathbb P$ inside it. But, as we have observed above in the proof of Proposition \ref{prop:cartesianization}, the universal model $U_{\mathbb{C}}=Hom_{\textrm{f.p.}\mathbb{C}\textrm{-mod}(\Set)}(F, -)$ (where $F$ is the free $\mathbb{C}$-algebra on one generator), of $\mathbb C$ in its classifying topos is also a universal model of $\mathbb{P}$ in its classifying topos $\Sh(\textrm{f.p.}\mathbb{C}\textrm{-mod}(\Set)^{\textrm{op}}, J_{\mathbb P})$. Now, the subobjects of $U\times \cdots \times U$ in $\Sh(\textrm{f.p.}\mathbb{C}\textrm{-mod}(\Set)^{\textrm{op}}, J_{\mathbb P})$ are precisely the $J_{\mathbb P}$-closed sieves on $F\times \cdots \times F$ in  $\textrm{f.p.}\mathbb{C}\textrm{-mod}(\Set)^{\textrm{op}}$ (cf. section \ref{back-groth}). From this our thesis follows at once.   
\end{proof}

The following proposition provides an explicit reformulation of the subcanonicity property of the Grothendieck topology $J_{\mathbb P}$.

\begin{proposition}
Let $M$ be a finitely presented algebra in Chang's variety\\ $\mathbb{C}\textrm{-mod}(\Set)$ and $\phi(\vec{x})$ a $\mathbb C$-cartesian formula. For any family of tuples $\vec{a_{i}}\in [[\vec{x}. \phi]]_{M_{i}}$ indexed by the MV-homomorphisms $f_{i}:M\to M_{i}$ from $M$ to finitely presented perfect MV-algebras $M_{i}$ such that for any MV-homomorphism $g:M_{i}\to M_{j}$ such that $g\circ f_{i}=f_{j}$, $g(\vec{a_{i}})=\vec{a_{j}}$, there exists a unique tuple $\vec{a}\in [[\vec{x}.\phi]]_{M}$ such that $f_{i}(\vec{a})=\vec{a_{i}}$ for all $i$.  
\begin{center}
\begin{tikzpicture}
\node (0) at (0,0) {$M$};
\node (1) at (-2,2) {$M_i$};
\node (2) at (2,2) {$M_j$};
\draw[->] (1) to node [above]{$g$}(2);
\draw[->] (0) to node [left] {$f_i$} (1);
\draw[->] (0) to node [right]{$f_j$} (2);
\end{tikzpicture}
\end{center}  
\end{proposition}

\begin{proof}
This immediately follows from the subcanonicity of the topology $J_{\mathbb P}$ (cf. Theorem \ref{thm:subdirect_product}) in view of the equivalence $\mathscr{C}_{\mathbb C}^{\textrm{cart}}\simeq \textrm{f.p.}\mathbb{C}\textrm{-mod}(\Set)^{\textrm{op}}$.
\end{proof}

The following proposition provides a characterization of the 
$\mathbb P$-equivalence classes of geometric sentences in terms of the theory $\mathbb C$.

\begin{proposition}
The $\mathbb P$-equivalence classes of geometric sentences are in natural bijection, besides with the ideals on $f.p.\mathbb{P}\textrm{-mod}(\Set)^{\textrm{op}}$ (cf. Lemma \ref{ideal-sentences}), with the $J_{\mathbb P}$-ideals on $\textrm{f.p.}\mathbb{C}\textrm{-mod}(\Set)^{\textrm{op}}$, that is with the sets $S$ of finitely presented algebras in $\mathbb{C}\textrm{-mod}(\Set)$ such that for any homomorphism $f:\mathcal{A}\to \mathcal{B}$ in $\textrm{f.p.}\mathbb{C}\textrm{-mod}(\Set)$, $\mathcal{A}\in S$ implies $\mathcal{B}\in S$ and for any $\mathcal{A}\in \mathbb{C}\textrm{-mod}(\Set)$ and any Boolean element $a$ of $\mathcal{A}$, $\mathcal{A}\slash (a)\in S$ and $\mathcal{A}\slash (\neg a)\in S$ imply $\mathcal{A}\in S$.
\end{proposition}

\begin{proof}
The thesis follows immediately from the fact that the subterminal objects of the topos $\Sh(\textrm{f.p.}\mathbb{C}\textrm{-mod}(\Set)^{\textrm{op}}, J_{\mathbb P})$ can be naturally identified with the $J_{\mathbb P}$-ideals on the category $\textrm{f.p.}\mathbb{C}\textrm{-mod}(\Set)^{\textrm{op}}$ (cf. section \ref{review}).
\end{proof}

Finally, we consider from the perspective of the two different representations $\Sh({\mathscr{C}_{\mathbb P}}, J_{\mathbb P})$ and $[\textrm{f.p.}\mathbb{P}\textrm{-mod}(\Set), \Set]$ of the classifying topos of $\mathbb P$ the invariant property of satisfying De Morgan's law. 

\[  
\xymatrix {
 & & [\textrm{f.p.}\mathbb{P}\textrm{-mod}(\Set), \Set]\simeq\Sh({\mathscr{C}_{\mathbb P}}, J_{\mathbb P})  \ar@/^12pt/@{--}[drr] & & \\
\textrm{f.p.}\mathbb{P}\textrm{-mod}(\Set)          \ar@/^12pt/@{--}[urr] & & & &  ({\mathscr{C}_{\mathbb P}}, J_{\mathbb P})      }
\]

For any small category $\mathscr{C}$, the topos $[\mathscr{C}, \Set]$ satisfies De Morgan's law if and only if the category $\mathscr{C}$ satisfies the \emph{amalgamation property} (AP), that is the property that every pair of arrows in $\mathscr{C}$ with common domain can be completed to a commutative square (cf. Example D4.6.3 \cite{SE}). On the other hand, for every coherent theory $\mathbb T$ over a signature $\Sigma$, the classifying topos $\Sh({\mathscr{C}_{\mathbb T}}, J_{\mathbb T})$ of $\mathbb T$ satisfies De Morgan's law if and only if for every geometric formula $\phi(\vec{x})$ over $\Sigma$ there exists a $\mathbb T$-Boolean coherent formula $\chi(\vec{x})$ in the same context such that for every geometric formula $\psi(\vec{x})$ over $\Sigma$, $(\psi \vdash_{x} \chi)$ is provable in $\mathbb T$ if and only if $(\psi \wedge \phi \vdash_{\vec{x}} \bot)$ is provable in $\mathbb T$ (cf. \cite{Caramello6}).

From the fact that the category of finitely presented $\ell$-groups satisfies the amalgamation property (since the theory $\mathbb L$ is cartesian), it follows that the category $\textrm{f.p.}\mathbb{P}\textrm{-mod}(\Set)$ satisfies AP as well and hence that $P$ satisfies the above-mentioned syntactic property. Summarizing, we have the following

\begin{proposition}
For every geometric formula $\phi(\vec{x})$ over the signature $\mathcal{L}_{MV}$ there exists a $\mathbb P$-Boolean coherent formula $\chi(\vec{x})$ in the same context such that for every geometric formula $\psi(\vec{x})$ over $\Sigma$, $(\psi \vdash_{x} \chi)$ is provable in $\mathbb P$ (equivalently, valid in all perfect MV-algebras) if and only if $(\psi \wedge \phi \vdash_{\vec{x}} \bot)$ is provable in $\mathbb P$ (equivalently, valid in all perfect MV-algebras).
\end{proposition}

\begin{remark}
Recall from section \ref{review} that, for a coherent theory of presheaf type $\mathbb T$, the $\mathbb T$-Boolean coherent formulas are precisely the formulas defining the prope\-rties of tuples of elements $\mathbb T$-models which are preserved by filtered colimits of models and both preserved and reflected by arbitrary homomorphisms of models. We can thus identify the $\mathbb P$-Boolean coherent formulas with the properties of tuples of elements of perfect MV-algebras which are preserved under filtered colimits of perfect MV-algebras and which are both preserved and reflected by arbitrary homomorphisms of perfect MV-algebras. 
\end{remark}

\subsection{Transferring results for $\ell$-groups with strong unit}\label{sec:transferring}

In this section we transfer some of the representation results that we obtained for MV-algebras in Chang's variety to $\ell$-u groups. 

\begin{proposition}\label{double}
Under Mundici's equivalence 
\[
\mathbb{MV}\textrm{-mod}(\Set)\simeq \mathbb{L}_{u}\textrm{-mod}(\Set)
\]
\begin{enumerate}[(i)]
\item the injective homomorphisms of MV-algebras correspond precisely to the injective homomorphisms of $\ell$-u groups;

\item the finitely generated MV-algebras correspond precisely to the finitely generated $\ell$-u groups.
\end{enumerate}

\end{proposition}

\begin{proof}
(i) Clearly, the homomorphisms of MV-algebras which correspond under Mundici's equivalence to injective homomorphisms of $\ell$-u groups are injective. The converse implication can be proved as follows. Given a homomorphism $f:{\mathcal G}\to {\mathcal G}'$ of $\ell$-u groups, $f$ is injective if and only if for every $x\in {\mathcal G}$, $f(x)=0$ implies $x=0$. Now, for any $x\in {\mathcal G}$, there exists $n$ such that $x\leq nu$. Using Riesz's interpolation property, there is $0\leq z\leq u$ such that $nz=x$. Now, $f(x)=f(nz)=nf(z)=0$, whence $f(z)=0$ since $\ell$-groups are torsion-free. So, from the fact that $f|_{[0, u]}$ is injective, it follows that $z=0$ and hence that $x=0$, as required.   

(ii) It is clear that the MV-algebra corresponding to a finitely generated $\ell$-u groups is finitely generated. Conversely, since by point (i) of the proposition the category of MV-algebra and injective homomorphisms between them and the category of $\ell$-u groups and injective homomorphisms between them are equivalent and every finitely generated MV-algebra is a finitely presentable object of the former category, every $\ell$-u group which corresponds to a finitely generated MV-algebra under Mundici's equivalence is finitely presentable as an object of the category of $\ell$-u groups and injective homomorphisms between them. Now, since every $\ell$-u group $\mathcal{G}$ is the filtered union of its finitely generated $\ell$-u subgroups, if $\mathcal{G}$ is finitely presentable as an object of the category of $\ell$-u groups and injective homomorphisms between them then $\mathcal{G}$ is finitely generated. This implies our thesis. 
\end{proof}

We are now in the position to transfer the representation results for the MV-algebras in Chang's variety that we obtained in section \ref{classifying topos} to the context of $\ell$-u groups.

First, we need to describe the quotient $\mathbb{A}$ of $\mathbb{L}_u$ corresponding to the quotient $\mathbb C$ of $\mathbb{MV}$ axiomatizing Chang's variety $V(C)$: this is clearly obtained from $\mathbb{L}_u$ by adding the following sequents:

\begin{center}
$(0\leq x\wedge x\leq u\vdash_{x} \sup(0,2\inf(2x,u)-u)=\inf(u,2\sup(2x-u,0))$

\

$((0\leq x\wedge x\leq u\vdash_{x}\inf(u,2\sup(0,2\inf(2x,u)-u))=\sup(0,2\inf(2x,u)-u))$
\end{center}

In view of Proposition \ref{double}, we immediately obtain the following result, representing the translation of Theorems \ref{thm:subdirect_product} and \ref{StoneChang}.

\begin{theorem}
Every $\ell$-u group which is a model of $\mathbb{A}$ is a weak subdirect product of antiarchimedean $\ell$-u groups. 

Every finitely generated (resp. finitely presentable) $\ell$-u group which is a model of $\mathbb{A}$ is a finite direct product of antiarchimedean (resp. antiarchimedean finitely presentable) $\ell$-u groups.
\end{theorem}

One could also, by using the same method as that leading to the proof of Theorem \ref{MVatomic}, intrinsically characterize the $\ell$-u groups which are direct products of antiarchimedean $\ell$-u groups.

\vspace{0.7cm}

\textbf{Acknowledgements:} We thank Antonio Di Nola for many helpful exchanges on the subject matter of this paper. We are also grateful to Giacomo Lenzi for useful discussions.

\vspace{1cm}
\textsc{Olivia Caramello}

{\small \textsc{Institut des Hautes \'Etudes Scientifiques\\
Bures-sur-Yvette, 35 route de Chartres, 91440, France}\\
\emph{E-mail addresses:} \texttt{olivia@ihes.fr}; \texttt{olivia@oliviacaramello.com}}. 

\vspace{0.5cm}

\textsc{Anna Carla Russo}

{\small \textsc{Dipartimento di Matematica e Informatica, Universit\' a degli Studi di Salerno, Via Giovanni Paolo II, 132, 84084 Fisciano (SA)}\\
\emph{E-mail address:} \texttt{anrusso@unisa.it}}

\end{document}